\documentclass[11pt, oneside]{article} 

\usepackage[utf8]{inputenc}
\usepackage[T1]{fontenc}
\usepackage{color} 
\usepackage[dvipsnames]{xcolor}

\usepackage{amsmath}
\usepackage{amssymb} 
\usepackage{mathrsfs} 
\usepackage{bm} 
\usepackage{bbm} 

\usepackage{amsthm} 
\theoremstyle{plain}
\newtheorem{theorem}{Theorem} 
\newtheorem{definition}{Definition}
\newtheorem{condition}{Condition}
\newtheorem{example}{Example}
\newtheorem{lemma}{Lemma}

\newtheorem{remark}{Remark}

\usepackage{algorithm2e} 
\SetKwInput{KwInputs}{Inputs}
\usepackage[noend]{algpseudocode}
\makeatletter
\def\BState{\State\hskip-\ALG@thistlm}
\makeatother

\usepackage{graphicx} 
\graphicspath{{Figures/}}
\usepackage{subcaption} 

\usepackage{tikz} 
\usetikzlibrary{positioning ,shadows,matrix}

\usepackage[english]{babel} 

\usepackage[square,numbers]{natbib} 
\usepackage[nottoc]{tocbibind} 

\usepackage{hyperref} 
\hypersetup{
    pdfpagemode={UseOutlines},
    bookmarksopen,
    pdfstartview={FitH},
    colorlinks,
    linkcolor={blue}, 
    citecolor={blue}, 
    urlcolor={blue} 
}

\usepackage[letterpaper,margin=1.25in]{geometry} 

\usepackage{enumitem} 

\def\iid{\overset{\textnormal{iid}}{\sim}} 

\makeatletter  
\@ifundefined{@currsizeindex} 
  {\RequirePackage{relsize}\let\dolarger\relsize} 
  {\def\dolarger#1{\larger[#1]}} 
\newcommand*\@@bigtimes[2]{\vphantom{\prod} 
  \vcenter{\hbox{\dolarger{4}$\m@th#1\mkern-2mu\times\mkern-2mu$}}} 
\newcommand*\bigtimes{\mathop{\mathpalette\@@bigtimes\relax}\displaylimits} 
\makeatother

\def\G{\mathbb{G}}\def\H{\mathbb{H}}\def\N{\mathbb{N}}\def\R{\mathbb{R}}\def\T{\mathbb{T}}\def\Z{\mathbb{Z}}\def\1{\mathbbm{1}}

\def\Dcal{\mathcal{D}}\def\Fcal{\mathcal{F}}\def\Hcal{\mathcal{H}}\def\Lcal{\mathcal{L}}\def\Qcal{\mathcal{Q}}

\newcommand{\eps}{\varepsilon}
\def\supp{\textnormal{supp}}

\title{\bf Semiparametric Bernstein--von Mises theorems for reversible diffusions}
\author{Matteo Giordano and Kolyan Ray \\ \\ University of Turin and Imperial College London}

\date{} 

\begin{document}

\maketitle

\abstract{
We establish a general semiparametric Bernstein--von Mises theorem for Bayesian nonparametric priors based on continuous observations in a periodic reversible multidimensional diffusion model. We consider a wide range of functionals satisfying an approximate linearization condition, including several nonlinear functionals of the invariant measure. Our result is applied to Gaussian and Besov-Laplace priors, showing these can perform efficient semiparametric inference and thus justifying the corresponding Bayesian approach to uncertainty quantification. Our theoretical results are illustrated via numerical simulations.\\

}

%



%

\tableofcontents

%
%
%
%
%

    \section{Introduction}
\label{Sec:Introduction}

Let $(X_t = (X_t^1,\dots,X_t^d): t \geq 0)$ be the multidimensional diffusion process arising as the solution to the SDE 
\begin{equation}\label{Eq:SDE}
    dX_t = \nabla B(X_t)dt + dW_t, \qquad t\ge 0, \qquad X_0=x_0\in\R^d,
\end{equation}
where $(W_t = (W_t^1,\dots,W_t^d): \ t\ge0)$ is a standard $d$-dimensional Brownian motion and $B:\R^d \to \R$ is a twice-continuously differentiable scalar potential function. We observe the continuous trajectory $X^T = (X_t: 0 \leq t \leq T)$ over a time horizon $T>0$, and consider statistical inference for low-dimensional functionals of the potential $B$ when this is modelled using a Bayesian nonparametric prior, leading to a semiparametric inference problem.

The SDE \eqref{Eq:SDE} describes the position of a particle diffusing in a potential energy field that exerts a force directed towards its local extrema, see Figure \ref{Fig:RevDiff}. By a classic result of Kolmogorov, the drift taking the form of a gradient vector field $\nabla B$ is equivalent to time reversibility of the process $X$ (e.g.~\cite{bakry2014}, p.~46). Reversible systems are widespread in the natural sciences \cite{K40,S80,PS10,pinski2012}, and one must thus model the scalar potential $B$ to correctly incorporate such physical dynamics. Furthermore, $B$ typically has a strong physical interpretation and estimating various aspects of it is often of significant interest. This supports directly modelling $B$, which is the approach we take here, assigning to $B$ a Bayesian prior.

We study statistical inference for low-dimensional functionals $\Psi(B)$ satisfying an approximate linearization condition, which includes several interesting nonlinear functionals. In the reversible setting \eqref{Eq:SDE}, there is a one-to-one correspondence between the potential $B$ and invariant measure $\mu_B$, see \eqref{Eq:invariant_measure} below, so that one can further embed functionals of the invariant measure into this framework. This allows us to treat several new and physically interesting cases, such as the entropy or integrated square root of the invariant measure.

The natural Bayesian approach is to assign a nonparametric prior to $B$ and consider the induced marginal posterior for $\Psi(B)$. We provide rigorous frequentist guarantees for this approach in the shape of a semiparametric Bernstein--von Mises (BvM) theorem as the time-horizon $T\to\infty$. It gives general conditions under which the marginal posterior for $\Psi(B)$ behaves asymptotically like a normal distribution centered at an efficient estimator of $\Psi(B)$ and with posterior variance equal to the inverse efficient Fisher information, which in model \eqref{Eq:SDE} can be expressed as the abstract solution to an elliptic PDE. This implies that the limiting covariances obtained coincide with the semiparametric information lower bounds for these estimation problems. In particular, such a result guarantees the validity of semiparametric Bayesian uncertainty quantification in the sense that posterior credible intervals for $\Psi(B)$ are also asymptotic frequentist confidence intervals of the correct level, see for instance \cite{castillo2012}. This is especially relevant since uncertainty quantification is a major motivation for using Bayesian methods in practice.

%

\begin{figure}[t]
\centering
\includegraphics[height=5cm,width=7cm]{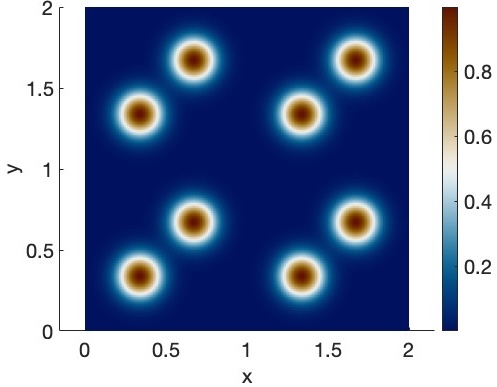}
\includegraphics[height=5cm,width=7cm]{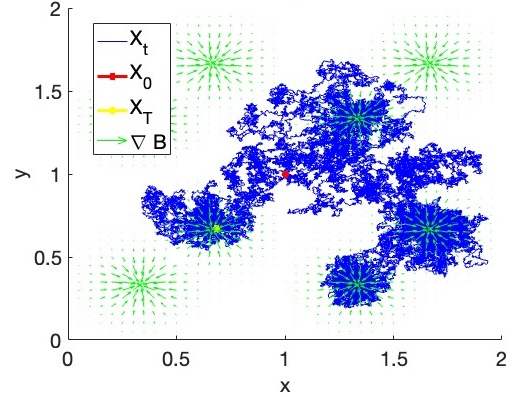}
\caption{Left: a periodic potential energy field $B$. Right: a continuous trajectory $(X_t: 0\le t\le T)$ started at $X_0=(1,1)$ and run until time $T=1$.}
\label{Fig:RevDiff}
\end{figure}

We apply our general theorem to two classes of priors: Gaussian processes and Besov-Laplace priors. Gaussian processes are widely used in diffusion models, partly due to computational advances in applying them to `real world' discrete data \cite{PPRS12,RBO13,GS14,vdMS17,BRO18,giordano2024statistical}. While often the standard Bayesian approach for such problems, Gaussian priors are known to be unsuited to modelling spatially inhomogeneous functions \cite{agapiou2021rates,giordano2022inability,AW21}, which motivated using heavier
tailed priors, such as Besov-Laplace priors, in the inverse problems and imaging communities \cite{LSS09,DHS12,KLNS12,BH15,ADH20,AW21}. The latter priors are better suited to modelling such inhomogeneities due to several attractive properties, such as edge-preservation and sparse solution representations, whilst also maintaining a log-concave structure amenable to posterior sampling. Our results provide first statistical guarantees for semiparametric inference using both these prior classes in a reversible diffusion setting. We further illustrate the applicability of results in numerical simulations.

We study statistical inference in the $T\to \infty$ regime, where one can use the average behaviour of the particle trajectory for inference. This requires a suitable notion of statistical ergodicity to ensure the particle exhibits enough recurrence to use long-time averages. We ensure this by following \cite{PPRS12,PSvZ13,vWvZ16,A18,GSMH19,nickl2020nonparametric,giordano2022nonparametric} in restricting to periodic potentials $B$. Periodicity simplifies several technical arguments, notably the elliptic PDE techniques, leading to a cleaner exposition of the main statistical ideas. Note that a periodic potential $B$ still implies the corresponding (periodised) diffusion is reversible (\cite{GSMH19}, Proposition 2), which maintains the modelling link between reversibility and potential functions behind our approach. For further discussion on alternative approaches to recurrence, such as confining potentials or reflecting boundary conditions, see Section 3.2 of \cite{giordano2022nonparametric}.

Bayesian nonparametric theory for drift estimation of diffusions is well-studied in dimension $d=1$ (e.g.~\cite{vdMeulen2006,PSvZ13,vWvZ16,NS17,A18}). In general dimensions $d\geq 2$, there has been recent progress on posterior contraction rates for continuous observations \cite{nicklray2020,giordano2022nonparametric} and discrete data \cite{nickl2024,hoffmann2025}, though little is known regarding the performance of Bayesian uncertainty quantification. For $d\leq 3$, Nickl and Ray \cite{nickl2020nonparametric}  obtain nonparametric BvM results for certain \textit{non-reversible} drift vector fields, which amongst other things imply semiparametric BvMs for smooth linear functionals. However, their approach relies on specific properties of truncated Gaussian series product priors on the drift $b=(b_1,\dots,b_d)$, which crucially uses that the priors for each coordinate are independent and supported on the same finite-dimensional projection spaces. Since product priors for $b$ draw gradient vector fields $b = \nabla B$ with probability zero, these cannot model reversibility. Their approach thus deals with a fundamentally different physical model, whereas here we deal with reversible dynamics, general dimension $d\geq 1$ and possibly nonlinear functionals. 
Convergence guarantees have also been obtained in multidimenions for various frequentist methods, typically kernel methods, see for instance \cite{DR07,S13,S15,S16,S18,aeckerle2022} and the references therein. For $d=1$, frequentist semiparametric inference for functionals of the invariant measure has been studied in \cite{aeckerle2022} and, for moment-type functionals, in \cite{KutoyantsYoshida2007}. We are not aware of analogous frequentist results for $d\geq 2$.

Our proof builds on the semiparametric BvM ideas of Castillo and Rousseau \cite{castillo2015}, extending these to the reversible diffusion model \eqref{Eq:SDE}. Elliptic PDE theory plays a key role in our results and proofs, for instance using mapping properties of the generator of the underlying semigroup
(via the Poisson equation and It\^o’s formula) to access martingale techniques to deal with the dependent data. Moreover, the representor $\psi_L$ of the linearization of the functional $\Psi(B)$ in the information (LAN) norm in model \eqref{Eq:SDE} can typically only be expressed implicitly as the solution to an elliptic PDE, whose regularity properties therefore play a central role in the analysis of semiparametric inference in this setting. In particular, the usual ``prior invariance condition'' \cite{castillo2012,castillo2015}, which reflects how well the prior aligns with the model and functional and is known to play an key role in the quality of semiparametric Bayesian inference, involves regularity properties of the solution map $\psi_L$.

The present \textit{reversible} diffusion setting also shares many connections with nonlinear inverse problems, where Bayesian methods have found significant use \cite{Stuart2010,nickl2023bayesian}. Modelling $B$ is equivalent to modelling the invariant measure (see \eqref{Eq:invariant_measure}), which leads to a nonlinear regression problem, see \cite{giordano2022nonparametric} for discussion regarding posterior contraction rates. In our context, this connection can be seen via the semiparametric information bounds in Section \ref{Sec:general_BvM}, which involve the solution to an elliptic PDE, see \cite{giordano2020bernstein} for a similar situation in inverse problems. In contrast, in the non-reversible case studied in \cite{nickl2020nonparametric}, such bounds are simpler weighted $L^2$-norms.


\paragraph{Notation.} Let $\T^d = (0,1]^d$ denote the $d$-dimensional torus, $L^p(\T^d)$ the usual Lebesgue spaces on $\T^d$, $\langle \cdot,\cdot \rangle_2$ the inner product on $L^2(\T^d)$ and
$$\dot{L}^2(\T^d) = \left\{ f\in L^2(\T^d): \int_{\T^d} f(x) dx = 0 \right\},$$
the subspace of centered $L^2(\T^d)$ functions. For an invariant measure $\mu_0$,  set $\langle f,g \rangle_{\mu_0} = \int f g d\mu_0$ with corresponding norm $\|f\|_{\mu_0}^2 = \int_{\T^d} f^2 d\mu_0$. For $f\in L^2(\T^d)$, define the empirical process 
$$\G_T[f] = \frac{1}{\sqrt{T}}\int_0^T f(X_s) dx - \int_{\T^d} f(x) d\mu_0(x),$$
where $\mu_0$ is the invariant measure of $X$ on $\T^d$, see \eqref{Eq:invariant_measure} below.

Let $C(\T^d)$ be the space of continuous functions on $\T^d$ equipped with $\|\cdot\|_\infty$. For $s>0$, denote by $C^s(\T^d)$ the usual H\"older space of $\lfloor s \rfloor$-times continuously differentiable functions whose $\lfloor s \rfloor^{\text{th}}$-derivative is $(s- \lfloor s \rfloor)$-H\"older continuous. We let $H^s(\T^d)$, $s\in \R$, denote the usual $L^2$-Sobolev spaces on $\T^d$, defined by duality when $s<0$. We further define the Sobolev norms $\|f\|_{W^{1,q}} = \|f\|_q + \sum_{i=1}^d \|\partial_{x_i}f\|_q$, $q\ge1$, and note that $\|\cdot\|_{W^{1,2}}$ is equivalent to $\|\cdot\|_{H^1}$.
	
Let $\{\Phi_{l r}: ~ l \in \{-1,0\} \cup \N, ~r=0,\dots,\max(2^{l d}-1,0)\}$
be an orthonormal tensor product wavelet basis of $L^2(\T^d)$, obtained from a periodised Daubechies wavelet basis of $L^2(\T)$, which we take to be $S$-regular for $S\in \N$ large enough; see Section 4.3. in \cite{ginenickl2016} for details. For $0 \leq t \leq S$, $1\le p,q\le\infty$, define the Besov spaces via their wavelet characterisation:
$$
	B^t_{p q}(\T^d)=\left\{ f\in L^p(\T^d): \ \|f\|_{B^t_{p q}}^q
	:=\sum_{l}2^{q l \left(t+\frac{d}{2}-\frac{d}{p}\right)}
	\left(\sum_r|\langle f,\Phi_{l r}\rangle_2 |
	^p\right)^\frac{q}{p}<\infty \right\},
$$
replacing the $\ell_p$ or $\ell_q$-norm above with $\ell_\infty$ if $p=\infty$ or $q=\infty$, respectively. Recall that $H^t(\T^d)=B^t_{22}(\T^d)$ and the continuous embedding  $C^s(\T^d)\subseteq B^s_{\infty\infty}(\T^d)$ for $s\ge0$, see Chapter 3 in \cite{ST87}. We sometimes suppress the dependence of the function spaces on the underlying domain, writing for example $B^t_{pq}$ instead of $B^t_{pq}(\T^d)$. We also employ the same function space notation for vector fields $f=(f_1,\dots,f_d)$. For instance, $f \in H^s = (H^s)^{\otimes d}$ will mean each $f_i \in H^s$ and the norm on $H^s$ is $\|f\|_{H^s} = \sum_{i=1}^d \|f_i\|_{H^s}$. Similarly, $\|\nabla g\|_p = \sum_{i=1}^d \|\partial_{x_i} g\|_p$. Finally, for a function space $X\subset L^2$, we denote by $\dot{X}=X\cap \dot{L^2}$

	We write $\lesssim$, $\gtrsim$ and $\simeq$ to denote one- or two-sided inequalities up to multiplicative constants that may either be universal or `fixed' in the context where the symbols appear. We also write $a_+ = \max(a,0)$ and $a \vee b=\max(a,b)$ for real numbers $a,b$. The $\eps$-covering number of a set $\Theta$ for a semimetric $d$, denoted $N(\Theta,d,\eps)$, is the minimal number of $d$-balls of radius $\eps$ needed to cover $\Theta$.

%
%
%
%
%


\section{Main results}
\label{Sec:BayesInf}

\subsection{Reversible diffusions with periodic drift and Bayesian inference}

Consider the SDE \eqref{Eq:SDE} with $B:\R^d \to \R$ a twice-continuously differentiable and one-periodic potential, i.e.~$B(x+m) = B(x)$ for all $m\in\Z^d$. There exists a strong pathwise solution $X = (X_t:t \geq 0)$ to \eqref{Eq:SDE} on the path space $C([0,\infty),\R^d)$, see Chapter 24 and 39 in \cite{B11}. We denote the law of the corresponding observed process $X^T = (X_t: 0 \leq t \leq T)$ by $P_B = P_B^T$, omitting the depending on the initial condition $X_0 = x_0$ since this does not affect our results.

Using the periodicity, we can analogously consider $B$ as a function on $\T^d$. Since the law of $X$ in \eqref{Eq:SDE} depends on $B$ only through the drift $\nabla B$, potentials differing only by additive constants yield the same law. To ensure identifiability, we thus consider the unique equivalence class with $\int_{\T^d} B(x) dx = 0$, namely $B \in \dot{L}^2(\T^d).$

In the present model, the process lives on all of $\R^d$, but will \textit{not} be globally recurrent in this space. However, the periodicity of $B$ means that the values of $(X_t)_t$ modulo $\Z^d$ encode all the relevant statistical information about $\nabla B$ contained in the trajectory $X^T$. The periodic model thus effectively restricts the diffusion to the bounded state space $\T^d$, providing a notion of statistical recurrence that suffices to ensure ergodicity in the asymptotic limit $T\to\infty$. Specifically, one can define an invariant measure on $\T^d$ since it holds that (arguing as in the proof of Lemma 6 of \cite{nicklray2020}):
\begin{align*}
    \frac{1}{T}\int_0^T \varphi (X_s) ds \to^{P_B} \int_{\T^d} \varphi(x) d\mu_B(x), \qquad \forall \varphi \in C(\T^d),
\end{align*}
as $T \to\infty$, where $\mu_B$ is a uniquely defined probability measure on $\T^d$ and we identify $\varphi$ with its periodic extension on the left-side of the last display. In the reversible diffusion model \eqref{Eq:SDE}, the potential function $B$ uniquely defines the invariant measure via its density, also denoted $\mu_B$:
\begin{align}\label{Eq:invariant_measure}
        \mu_B(x) = \frac{e^{2B(x)}}{\int_{\T^d}e^{2B(z)}dz}, \qquad x\in\T^d,
\end{align}
see p.~45-47 of \cite{bakry2014}. As usual, $\mu_B$ can equivalently be identified as the solution to the PDE $L_B^* u = 0$ for $L_B^*$ the $L^2$-adjoint of the generator $L_B$ defined in eq.~\eqref{Eq:generator} in the Supplement \cite{supp}.


The log-likelihood of $B\in \dot{C}^2(\T^d)$ is given by Girsanov's theorem (e.g.~\cite{B11}, Section 17.7):
\begin{align}\label{Eq:likelihood}
        \ell_T(B):=\log \frac{dP_B}{dP_{B=0}}(X^T)=-\frac{1}{2}\int_0^T\|\nabla B(X_t)\|^2 dt + \int_0^T\nabla B(X_t).dX_t,
\end{align}
where $P_{B=0}$ is the law of a $d$-dimensional Brownian motion $(W_t: 0 \leq t \leq T)$. Assigning to $B$ a possibly $T$-dependent prior $\Pi = \Pi_T$, which is supported on $\dot{C}^2(\T^d)$, the posterior is given by Bayes formula:
$$
    d\Pi(B|X^T)=\frac{e^{\ell_T(B)}d\Pi(B)}{\int_{\dot{C}^2(\T^d)}e^{\ell_T(B')}d\Pi(B')},
    \qquad B\in \dot{C}^2(\T^d).
$$
For a given functional $\Psi:\dot{C}^2(\T^d) \to \R$, the Bayesian uses the marginal posterior distribution, whose law equals the pushforward $\Pi(\cdot|X^T) \circ \Psi^{-1}$, for inference on $\Psi(B)$. More concretely, we can generate posterior samples $\Psi(B)|X^T$ from $B \sim \Pi(\cdot|X^T)$ via evaluations of the functional. This provides a principled Bayesian approach to inference on $\Psi(B)$ that, as we will show below, can be optimal from an information theoretic point of view. Note that by the injectivity of the map $B \mapsto \mu_B$ in \eqref{Eq:invariant_measure}, functionals $\Phi(\mu_B)$ of the invariant measure can equivalently be considered as functionals of $B$, and hence are contained within this framework.

We will study the behaviour of the marginal posterior distribution assuming the data $X^T \sim P_{B_0}$ is generated from a `ground truth' law following the SDE \eqref{Eq:SDE} with potential $B_0$. We will often write $P_0 := P_{B_0}$ for notational convenience, differentiating this from the notation $P_{B=0}$ used in the likelihood \eqref{Eq:likelihood}.

%
%
%
%
%

\subsection{A general semiparametric Bernstein--von Mises theorem}\label{Sec:general_BvM}

We provide here a general result giving conditions on the prior and functional under which the marginal posterior provides asymptotically optimal semiparametric inference. Let $\Psi:\dot{C}^2(\T^d) \to \R$ be a one-dimensional functional admitting an expansion about $B_0$ of the form
\begin{equation}\label{Eq:psi_exp}
    \Psi(B) = \Psi(B_0) + \langle \psi, B-B_0 \rangle_2 + r(B,B_0),
\end{equation}
where $\psi \in L^2(\T^d)$ and $r(B,B_0)$ is a remainder term of size $o(T^{-1/2})$ for $B$ in a neighbourhood of $B_0$ in a sense made precise below. Since $B,B_0 \in \dot{C}^2(\T^d)$, we may take $\psi$ to satisfy $\int_{\T^d} \psi = 0$ since this does not change \eqref{Eq:psi_exp} (e.g.~by considering Parseval's theorem in the Fourier basis).
If $\Psi$ is a linear functional, then simply $r(B,B_0) = 0$. Thus $\psi$ is the Riesz representor of the linearization of $\Psi$ in the $L^2(\T^d)$-inner product.

The optimal asymptotic variance in semiparametric estimation theory is connected to the information or local asymptotic normality (LAN) norm induced by the statistical model (e.g.~Chapter 25 of \cite{vdvaart1998}). In the present reversible diffusion setting, the LAN inner product is $\langle B,\bar{B} \rangle_L := \langle \nabla B, \nabla \bar{B} \rangle_{\mu_0}$ (see Lemma \ref{lem:LAN} in the Supplement \cite{supp}), giving a corresponding functional expansion 
\begin{equation*}
    \Psi(B) = \Psi(B_0) + \langle \psi_L, B-B_0 \rangle_L + r(B,B_0).
\end{equation*}
Define the second order elliptic partial differential operator $A_{\mu}: H^2(\T^d) \to L^2(\T^d)$ given by
\begin{align}\label{Eq:A_mu}
    A_{\mu}u:=\nabla\cdot(\mu\nabla u )
    = \mu \Delta u + \nabla \mu . \nabla u.
\end{align}
The functional representors are then related by $\psi_L = A_{\mu_0}^{-1}[\psi]$, i.e.~$\psi_L$ solves the elliptic PDE $A_{\mu_0} \psi_L = \psi$. Thus, the corresponding information bound in this model is $\|\psi_L\|_L^2 = \|\nabla \psi_L\|_{\mu_0}^2 = \|\nabla A_{\mu_0}^{-1}\psi\|_{\mu_0}^2.$ However, for most functionals $\Psi$, especially nonlinear ones, we typically have access to $\psi$ rather than $\psi_L$ (see below for examples), since the latter is usually only defined implicitly as the solution to an elliptic PDE. We therefore consider the expansion \eqref{Eq:psi_exp} to make our regularity conditions concrete, deducing the regularity of the corresponding $\psi_L$ by elliptic regularity theory.

Assuming $\Psi$ is differentiable in a suitable sense, a sequence of estimators $\hat{\Psi}_T$ is asymptotically efficient for estimating $\Psi(B)$ at the true parameter $B_0$ if
\begin{equation}\label{Eq:efficient_est}
\begin{split}
\hat{\Psi}_T &= \Psi(B_0) + \frac{1}{T} \int_0^T \nabla \psi_L(X_t).dW_t  + o_{P_0}(T^{-1/2}) \\
&= \Psi(B_0) + \frac{1}{T} \int_0^T \nabla A_{\mu_0}^{-1}\psi (X_t).dW_t + o_{P_0}(T^{-1/2}), 
\end{split}
\end{equation}
see equation (25.22) in \cite{vdvaart1998}. Using the martingale central limit theorem, the sequence $\sqrt{T}(\hat{\Psi}_T-\Psi(B_0))$ is then asymptotically normal with mean zero and variance $\|\psi_L\|_{\mu_0}^2 = \|\nabla A_{\mu_0}^{-1}\psi\|_{\mu_0}^2$, which is the best possible in a local minimax sense.

We write $\mathcal{L}_\Pi(\sqrt{T}(\Psi(B)-\hat{\Psi}_T)|X^T)$ for the marginal posterior distribution of $\sqrt{T}(\Psi$ $(B)-\hat{\Psi}_T)$, where $\hat{\Psi}_T$ is any random sequence satisfying \eqref{Eq:efficient_est}. The semiparametric Bernstein--von Mises theorem states that this distribution asymptotically resembles a centered normal distribution with variance $\|\nabla A_{\mu_0}^{-1}\psi\|_{\mu_0}^2$. We now make this statement precise via the bounded Lipschitz distance $d_{BL}$ on probability distributions on $\R$ (see Chapter 11 of \cite{dudley2002}) before stating our general result.

\begin{definition}
Let $X^T = (X_t:0\leq t \leq T)$ denote an observation from the SDE \eqref{Eq:SDE} with potential $B_0$, whose distribution we denote by $P_{B_0}=P_0$. We say the posterior satisfies the semiparametric Bernstein--von Mises (BvM) for a functional $\Psi$ satisfying expansion \eqref{Eq:psi_exp} if, for $\hat{\Psi}_T$ satisfying \eqref{Eq:efficient_est} and as $T \to \infty,$
$$d_{BL} \left( \mathcal{L}_\Pi(\sqrt{T}(\Psi(B)-\hat{\Psi}_T)|X^T) , N(0,\|\nabla A_{\mu_0}^{-1}\psi\|_{\mu_0}^2) \right) \to^{P_0} 0.$$
\end{definition}

\begin{theorem}[Semiparametric Bernstein--von Mises]\label{thm:BvM}
Let $\Pi=\Pi_T$ be a prior for $B$ that is supported on $\dot{C}^2(\T^d)$. Suppose  $B_0 \in \dot{C}^{(d/2+1+\kappa)\vee 2}(\T^d)$ for some $\kappa>0$, and let $\Psi:\dot{C}^2(\T^d) \to \R$ be a functional satisfying the expansion \eqref{Eq:psi_exp} with representor $\psi \in \dot{H}^t(\T^d)$ with $t>(d/2-1)_+$. For $1\le p\le 2$ and $2\le q\le \infty$ such that $1/p+1/q=1$, assume there exist measurable sets $\Dcal_T$ satisfying $\Pi(\Dcal_T|X^T) \overset{P_0}{\longrightarrow}1$ and such that
\begin{equation}\label{Eq:D_T_conditions}
\begin{split}
        \Dcal_T \subseteq  & \big\{ B\in \dot C^2(\T^d):\|\nabla B - \nabla B_0\|_{p}\le M \eps_T, ~ \|B-B_0\|_{H^{d/2+1+\kappa}}\leq \zeta_T,\\
        & \qquad \|B\|_{H^{d+1+\kappa}} \leq M, ~ |r(B,B_0)| \leq \xi_T/\sqrt{T} \big\}
\end{split}
\end{equation}
for some $M>0$ and $\eps_T,\zeta_T,\xi_T \to 0$ with $\sqrt{T}\eps_T\to\infty$. Let $\gamma_T\in \dot{H}^{d/2+1+\kappa}(\T^d)$ be a sequence of fixed functions such that, as $T\to\infty$,
\begin{align}\label{Eq:gamma_T_conditions}
    \|\gamma_T\|_{H^{d/2+1+\kappa}}=O(1);
    \qquad  \|A^{-1}_{\mu_0}\psi - \gamma_T\|_{W^{1,q}} = o(1/(\sqrt{T}\eps_T)),
\end{align}
where $A_{\mu_0}$ is the second-order operator \eqref{Eq:A_mu}. For $u\in\R$, define the perturbations
\begin{align*}
    B_u = B - u \gamma_T / \sqrt{T}.
\end{align*}
If for every $u\in\R$ in a neighbourhood of 0,
\begin{align}\label{Eq:change_of_measure}
    \frac{\int_{\Dcal_T} e^{\ell_T(B_u)} d\Pi(B)}{\int_{\Dcal_T} e^{\ell_T(B)} d\Pi(B)} \overset{P_0}{\longrightarrow}1
\end{align}
as $T\to\infty$, then the semiparametric Bernstein--von Mises holds for $\Psi$.
\end{theorem}

The condition \eqref{Eq:D_T_conditions} requires the posterior to concentrate on sets around the true $B_0$ on which one can perform a LAN expansion of the likelihood with uniform control of the remainder terms. These conditions can be verified using general tools for proving posterior contraction as in \cite{ghosal2017}, which have been applied in the multivariate diffusion setting with continuous observations in \cite{nickl2020nonparametric,giordano2022nonparametric}. The condition $r(B,B_0) = o(1/\sqrt{T})$ means the functional $\Psi(B)$ is approximately linear with expansion \eqref{Eq:psi_exp}, which nonetheless allows to cover several interesting nonlinear functionals.

Condition \eqref{Eq:change_of_measure} requires invariance of the prior for the full parameter $B$ under a shift $B_u = B-u\gamma_T/\sqrt{T}$ in the approximate least favourable direction $\gamma_T$, which should be close to the true least favourable direction $\psi_L = \nabla A_{\mu_0}^{-1} \psi$. This condition reflects how well the prior aligns with the model and functional, and if not satisfied may prevent the $\sqrt{T}$-rate in the semiparametric BvM theorem (see \cite{castillo2012,castillo2012b,ghosal2017} for further discussion). Considering an approximation $\gamma_T$ to $\nabla A_{\mu_0}^{-1} \psi$ allows to weaken the condition \eqref{Eq:change_of_measure} for concrete priors. Note that verifying \eqref{Eq:change_of_measure} for specific priors may impose additional smoothness conditions.

We next consider examples of functionals covered by Theorem \ref{thm:BvM}, including functionals of the invariant measure $\mu_B$. Proofs of the expansions \eqref{Eq:psi_exp} and remainder terms can be found in Section \ref{Sec:functionals} of the Supplement.

\begin{example}[Linear functionals]\label{Ex:linear_B}
    If $\Psi(B) = \int_{\T^d} B(x) a(x) dx$ for $a\in L^2(\T^d)$, then $\psi = a - \int_{\T^d} a$ and $r(B,B_0) = 0.$
\end{example}

\begin{example}[Square functional]\label{Ex:square}
    If $\Psi(B) = \int_{\T^d} B(x)^2 dx$, then $\psi = 2B_0$ and $r(B,B_0) = \|B-B_0\|_2^2.$
\end{example}

\begin{example}[Power functionals]\label{Ex:power}
    Let $\Psi(B) = \int B(x)^q dx$ for $q\geq 2$ an integer. Then $\psi = q[B_0^{q-1} - \int B_0^{q-1}]$ and for any $K,M>0$ and $\nu_T \to 0,$
    \begin{align}\label{Eq:remainder_text}
    \sup_{\substack{B,B_0:\|B\|_\infty, \|B_0\|_\infty \leq K \\ \| B -  B_0\|_2 \leq M \nu_T}}|r(B,B_0)| = O(\nu_T^2).
    \end{align}
\end{example}

\begin{remark}[Posterior contraction rates and remainder]\label{Rem:remainder}
For $B,B_0 \in \dot{L}^2(\T^d)$, i.e.~satisfying $\int B = \int B_0 = 0$, the Poincar\'e inequality implies $\|B-B_0\|_2 \leq C_d \|\nabla B - \nabla B_0\|_2$ (e.g.~p.~290 in \cite{evans2010}) for some $C_d>0$. One can thus replace the $L^2$-norm in the above remainder by $\|\nabla B-\nabla B_0\|_2$. When $p=2$, the other conditions in \eqref{Eq:D_T_conditions} together with \eqref{Eq:remainder_text} and the Sobolev embedding theorem imply that $|r(B,B_0)| = O(\eps_T^2)$ uniformly over $\Dcal_T$, i.e.~we may take $\xi_T = \sqrt{T}\eps_T^2$ (provided $\eps_T=o(T^{-1/4})$). If only a posterior contraction rate for $\|\nabla B - \nabla B_0\|_p$ is available, it may still be possible to exploit additional prior support properties to  deduce that $\|B - B_0\|_2\lesssim \nu_T$ for some $\nu_T =o(T^{-1/4})$, generally slower than $\eps_T$. This is done for Besov-Laplace priors in Section \ref{Sec:LaplPriors}.
\end{remark}

\begin{example}[Linear functionals of the invariant measure]\label{Ex:linear}
    Let $\Psi(B) = \int_{\T^d} \mu_B(x) \varphi(x)$ $ dx$ for $\varphi \in L^\infty(\T^d)$, where $\mu_B = e^{2B}/\int_{\T^d} e^{2B}$ is the invariant measure. Then $\psi =2\mu_{B_0}[\varphi - \int \mu_{B_0} \varphi]$ and the remainder term $r(B,B_0)$ satisfies \eqref{Eq:remainder_text}.
    Note that $\Psi$ is \textit{nonlinear} in $B$.
\end{example}

\begin{example}[Entropy of the invariant measure]\label{Ex:entropy}
    Let $\Psi(B) = \int_{\T^d} \mu_B(x)\log \mu_B(x) dx$. Then $\psi = 2\mu_{B_0}$ $ [\log \mu_{B_0} - \int \mu_{B_0} \log \mu_{B_0}]$ and the remainder term $r(B,B_0)$ satisfies \eqref{Eq:remainder_text}.
\end{example}

\begin{example}[Square-root of the invariant measure]\label{Ex:square_root}
    Let $\Psi(B) = \int_{\T^d}\sqrt{\mu_B(x)} dx$. Then $\psi = \mu_{B_0}[\frac{1}{\sqrt{\mu_{B_0}}}$ $ - \int \sqrt{\mu_{B_0}} ]$ and the remainder term $r(B,B_0)$ satisfies \eqref{Eq:remainder_text}.
\end{example}

\begin{example}[Power functionals of the invariant measure]\label{Ex:power_mu}
    Let $\Psi(B) = \int \mu_B(x)^q dx$ for $q\geq 2$ an integer. Then $\psi = 2\mu_{B_0} [q \mu_{B_0}^{q-1} - q \int \mu_{B_0}^q]$ and the remainder term $r(B,B_0)$ satisfies \eqref{Eq:remainder_text}.
\end{example}

The condition $r(B,B_0) \leq \xi_T/\sqrt{T}$ in \eqref{Eq:D_T_conditions} of Theorem \ref{thm:BvM} is thus satisfied in Examples \ref{Ex:square} - \ref{Ex:power_mu} as soon as $\nu_T = o(T^{-1/4})$ (in Example \ref{Ex:linear_B} it automatically follows with $\xi_T = 0$). For functionals of the invariant measure, the remainder condition \eqref{Eq:remainder_text} also ensures $\mu_B$ is bounded away from zero and infinity, so that the regularity of the various integrands matches that of $\mu_B$, and hence $B$. For example, it is known that the square root of an infinitely differentiable function near zero need not be more than $C^1$ in general \cite{bony2006} without additional assumptions, e.g.~\cite{ray2017}. In such situations, it is unclear if $\sqrt{T}$-rates are attainable since the functional $B \mapsto \Psi(B)$ may be non-differentiable, and non-regular settings can further lead to non-Gaussian posterior limits, see \cite{BochkinaGreen2014} for a parametric example. Taking $\mu_B$ bounded away from zero rules out such situations, where statistical estimation problems can behave qualitatively differently \cite{RaySchmidt2016,Patschkowski2016,RaySchmidt2019}. 

\begin{remark}[Functional regularity]
Employing a standard Bayesian nonparametric prior for $B$ and considering the induced marginal posterior for $\Psi(B)$ can be viewed as a ``plug-in'' method. In terms of \textit{uniformly} controlling linear functionals of $B$, one expects that a smoothness condition of the form $t>d/2-1$ as in Theorem \ref{thm:BvM} is sharp for such methods when $d\geq 2$, see Sections 2.4-2.5 of \cite{nickl2020nonparametric} and also \cite{CastilloNickl2013} for more general discussion. For certain models and functionals, Bayesian methods can be tailored to provide efficient semiparametric inference for lower regularities, for instance using targeted priors \cite{RaySzabo2019,RayVdVaart2020} or posterior corrections \cite{castillo2015,Yiu2025}, but this is beyond the scope of this article.
\end{remark}

\begin{remark}[Inference on invariant probabilities]
Taking $\varphi = 1_C$ for a measurable set $C \subseteq \T^d$ in Example \ref{Ex:linear} gives $\Psi(B) = \mu_B(C)$, the invariant probability of $C$. If $C \subseteq \T^d$ has finite perimeter, then $1_C \in BV(\T^d) \cap L^\infty(\T^d)$, where the first set consists of functions of bounded variation. Using standard approximation theory for BV functions (e.g. Theorem 5.3 in \cite{evansgariepy2015}), the smooth translation estimate $\|u(\cdot+h)-u\|_{L^1(\T^d)} \leq \|h\||Du|(\T^d)$ extends to functions $u\in BV(\T^d)$.
If also $u \in L^\infty (\T^d)$, then
$$\|u(\cdot+h)-u\|_{L^2}^2 \leq 2\|u\|_\infty \|u(\cdot+h)-u\|_{L^1} \leq 2 \|u\|_\infty \|h\| |Du|(\T^d).$$
By the modulus of smoothness characterisation of Besov spaces (\cite{ginenickl2016}, p.328--329), this implies $u\in B_{2\infty}^{1/2}(\T^d)$, and hence $u\in H^{1/2-\kappa}(\T^d)$ for any $\kappa>0$ by Proposition 4.3.6 of \cite{ginenickl2016}.
The usual multiplier inequality for Sobolev spaces then gives
$\|\psi\|_{H^{1/2-\kappa}} \lesssim \|\mu_{B_0}\|_{C^{1/2-\kappa}} (1+\|1_C\|_{H^{1/2-\kappa}})<\infty$,
so that the representor $\psi\in H^t$ for some $t>(d/2-1)_+$ when $d=1,2$, and hence our results apply to $\mu_B(C)$.
A process-level or nonparametric BvM for $\{\mu_B(C):C\in\mathcal{C}\}$, where $\mathcal{C}$ contains sets with uniformly bounded perimeter, would require making our arguments uniform over a class of functionals simultaneously, see \cite{nickl2020nonparametric} for such results in the non-reversible setting when $d\leq 3$.
\end{remark}

%
%
%

\subsection{Gaussian priors}
\label{Sec:GP}

We first apply our general theory to Gaussian priors, which are widely used for estimating the coefficients of SDEs \cite{PPRS12,PSvZ13,RBO13,GS14,vdMS17,BRO18,giordano2024statistical}. For continuously-observed multidimensional reversible diffusions, Giordano and Ray \cite{giordano2022nonparametric} showed that Gaussian priors are conjugate with explicitly available formulae for posterior inference, and derived minimax-optimal posterior contraction rates under suitable tuning of the hyperparameters. We now show that Gaussian priors also lead to efficient semiparametric inference for a large class of functionals of the potential.

As in \cite{giordano2022nonparametric}, we consider rescaled Gaussian priors constructed from a base probability measure satisfying the following mild regularity condition. For definitions and background information on Gaussian processes and measures, see \cite[Chapter 2]{ginenickl2016}, or \cite[Chapter 11]{ghosal2017}.

\begin{condition}\label{Condition:BaseGP}
For $s>d$
and some $\kappa>0$, let $\Pi_W=\Pi_{W,T}$ be a (possibly $T$-dependent) centred Gaussian Borel probability measure on the Banach space $\dot{C}(\T^d)$ that is supported on a separable (measurable) linear subspace of $\dot{H}^{d+1+\kappa}(\T^d)\cap \dot{C}^{(d/2+\kappa)\vee 2}(\T^d)$,
and assume that its reproducing kernel Hilbert space (RKHS) $(\H_W,\|\cdot\|_{\H_W})$ is continuously embedded into $\dot{H}^{s+1}(\T^d)$.
\end{condition}

Note that the above constant $\kappa>0$ can be arbitrarily small. Concrete choices of Gaussian priors satisfying Condition \ref{Condition:BaseGP} are given in Examples \ref{Ex:Matern} and \ref{Ex:WavSeries} below. For $W\sim \Pi_W$ with $\Pi_W$ satisfying Condition \ref{Condition:BaseGP}, we construct the rescaled Gaussian prior $\Pi=\Pi_T$ by taking the law of the random function
\begin{equation}
\label{Eq:RescaledGP}
    B(x):=\frac{W(x)}{T^{d/(4s+2d)}},
    \qquad x\in\T^d.
\end{equation}
By linearity, $\Pi$ is also a centred Gaussian Borel probability measure on $\dot{C}(T^d)$, with the same support as $\Pi_W$ and with RKHS $\H $ satisfying $\H  = \H_W$ and $\|\cdot\|_{\H } = T^{d/(4s+2d)}\|\cdot\|_{\H_W}$ (cf.~Exercise 2.6.5 in \cite{ginenickl2016}).

\begin{remark}[Rescaling]\label{Rem:Resc}
Rescaling by the diverging term $T^{d/(4s+2d)}$ in \eqref{Eq:RescaledGP} is a technique borrowed from the theoretical literature on Bayesian nonlinear inverse problems, e.g.~\cite{MNP20,giordano2020consistency,nickl2023bayesian}, which ensures the posterior for $B$ places most of its mass on sets of bounded higher-order smoothness. In that setting, it is typically needed to control stability estimates, uniformly over the bulk of the posterior. In the context of multidimensional reversible diffusions, this was similarly used in the posterior contraction rate analysis of \cite{giordano2022nonparametric} to control the nonlinearity of the map $B \mapsto \mu_B$ in \eqref{Eq:invariant_measure}.
Via the localising sets $\Dcal_T$ in \eqref{Eq:D_T_conditions}, Theorem \ref{thm:BvM} requires the posterior to contract about the truth in order to perform a LAN expansion of the likelihood, while the bounded Sobolev regularity of the posterior is needed to control stochastic bias terms in the likelihood (e.g.~Lemma \ref{Lem:EmpProc} in the Supplement \cite{supp}). The rescaling in \eqref{Eq:RescaledGP} ensures such conditions are satisfied. Deriving posterior contraction rates, let alone BvM results, for non-rescaled Gaussian priors in such nonlinear settings is currently an open problem.
\end{remark}

\begin{theorem}\label{Theo:GaussBvM}
Let $\Pi$ be the rescaled Gaussian prior from \eqref{Eq:RescaledGP} with $W\sim \Pi_W$ satisfying Condition \ref{Condition:BaseGP} for some $s>d$, some $\kappa>0$ and RKHS $\H_W$. Set $\eps_T = T^{-s/(2s+d)}$, suppose that $B_0\in \dot{H}^{s+1}(\T^d)\cap  \dot{C}^{(d+1/2+\kappa)\vee 2}(\T^d)$ and that there exists a sequence $B_{0,T}\in\H_W$ such that $\|B_{0,T}\|_{\H_W}=O(1)$ and $\|B_0 - B_{0,T}\|_{C^1}=O(\eps_T)$ as $T\to\infty$.

Let $\Psi:\dot{C}^2(\T^d) \to \R$ be a functional satisfying the expansion \eqref{Eq:psi_exp} with representor $\psi \in \dot{H}^t(\T^d)$ for some $t>(d/2-1)_+$ and remainder satisfying \eqref{Eq:remainder_text}, and assume  there exists a sequence $\gamma_T\in \H_W$ such that
\begin{equation}\label{Eq:GammaTNorms}
    \|\gamma_T\|_{H^{d/2+1+\kappa}}=O(1);\qquad
    \|\gamma_T\|_{H^{d+1+\kappa}}=o(\sqrt T);
    \qquad
    \|\gamma_T\|_{\H_W}=o(1/(\sqrt T\eps_T^2)),  
\end{equation}
as well as
\begin{equation}\label{Eq:GammaTApprox}
    \|A^{-1}_{\mu_0}\psi - \gamma_T\|_{H^1} = o(1/(\sqrt{T}\eps_T)).
\end{equation}
Then the semiparametric Bernstein--von Mises holds for $\Psi$.
\end{theorem}

Theorem \ref{Theo:GaussBvM} shows that suitably rescaled Gaussian priors satisfy a semiparametric BvM under the mild regularity Condition \ref{Condition:BaseGP}. The requirements on the ground truth $B_0$ parallel those in Theorem 2.1 in \cite{giordano2022nonparametric}, which establishes posterior contraction rates that are used in the proof of Theorem \ref{Theo:GaussBvM} to construct the localising sets $\Dcal_T$ from \eqref{Eq:D_T_conditions}. These conditions entail that $B_0$ be (at least) $(s+1)$-Sobolev-regular, and that it be well-approximated by elements of the RKHS $\H_W$. For Gaussian priors modelling Sobolev smooth functions, the approximating sequence $B_{0,T}$ can be readily constructed; cf.~Examples \ref{Ex:Matern} and \ref{Ex:WavSeries} below. The smoothness condition $s>d$, similar to the one imposed in the contraction rate analysis of \cite{giordano2022nonparametric}, is slightly stronger than the typical assumptions for BvM results in simpler nonparametric statistical models; it is possibly an artefact of the proof arising from the need to control the underlying nonlinearity.

%

In Theorem \ref{Theo:GaussBvM}, we restrict to functionals $\Psi$ with remainders satisfying \eqref{Eq:remainder_text} for concreteness, since the remainder condition can then be more easily verified via Remark \ref{Rem:remainder}. As shown by Examples \ref{Ex:linear_B} - \ref{Ex:power_mu}, this already allows to cover several interesting nonlinear instances. The additional norm bounds in \eqref{Eq:GammaTNorms} for the approximate LAN representors $\gamma_T$ compared to \eqref{Eq:gamma_T_conditions} are used in the verification of the asymptotic invariance property \eqref{Eq:change_of_measure} for Gaussian priors. For priors modelling Sobolev regular functions, conditions \eqref{Eq:GammaTNorms} and \eqref{Eq:GammaTApprox} can typically be verified via standard elliptic regularity estimates and approximation theory for any $\psi\in \dot{H}^t(\T^d)$, $t>(d/2-1)_+$, and thus impose no additional restrictions.

In the following examples, to enforce the identifiability condition $\int_{\T^d} B(x) dx = 0$, we set the coefficient of the constant function ($e_0\equiv 1$ for the Fourier basis and $\Phi_{-10}\equiv 1$ for wavelets) equal to zero. For more general Gaussian processes, one can directly recenter priors draws $B \mapsto B - \int_{\T^d} B(x) dx$.

\begin{example}[Periodic Matérn process]\label{Ex:Matern}
For $s>3d/2$, consider a base Gaussian prior $\Pi_W$ given by the law of a periodic Matérn process of regularity $s+1-d/2$ (cf.~Example 11.8 in \cite{ghosal2017}), namely the centred Gaussian process $W=\{W(x),\ x\in\T^d\}$ with covariance kernel
\begin{equation}\label{Eq:Kper}
    K_{\textnormal{per}}(x,y) := (2\pi)^d \sum_{k\in\Z^d\backslash\{0\}} \frac{1}{(1+4\pi^2|k|^2)^{(s+1)/2}}e_k(x)\overline{e_k(y)},
    \qquad x,y\in\T^d,
\end{equation}
where $\{e_k(x) = e^{2\pi k.x}: k\in\Z^d\}$ is the Fourier basis of $L^2(\T^d)$, see Section A.1.1 of \cite{giordano2022nonparametric} for further details. By the Fourier series characterization of periodic Sobolev spaces, it has RKHS $\H_W = \dot{H}^{s+1}(\T^d)$ with RKHS norm equivalence $\|\cdot\|_{\H_W}\simeq \|\cdot\|_{H^{s+1}}$. 

The periodic Mat\'ern process satisfies Condition \ref{Condition:BaseGP}, being supported in $\dot{C}^{s+1-d/2-\eta}(\T^d)$ for any $\eta>0$, which is a separable linear subspace of $\dot{H}^{d+1+\kappa}(\T^d)\cap \dot{C}^{(d/2+\kappa)\vee 2}(\T^d)$ for sufficiently small $\kappa>0$ since $s>3d/2$ (see Section A.1.1 of \cite{giordano2022nonparametric} for details). Given $B_0\in \dot{H}^{s+1}(\T^d) = \H$, we may apply Theorem \ref{Theo:GaussBvM} with the trivial choice $B_{0,T}\equiv B_0$.

Let $\Psi:\dot{C}^2(\T^d) \to \R$ be a functional with expansion \eqref{Eq:psi_exp}, representor $\psi \in \dot{H}^t(\T^d)$ for some $t>(d/2-1)_+$ and remainder satisfying \eqref{Eq:remainder_text}. Since $B_0\in \dot{H}^{s+1}(\T^d)$ and $s>3d/2$, we also have $\mu_0\propto e^{2B_0}\in H^{s+1}(\T^d)\subset C^{d+1+\kappa}(\T^d)$, the last inclusion holding by the Sobolev embedding. Hence by Lemma \ref{lem:PDE} in the Supplemt, there exists a unique element $A^{-1}_{\mu_0}\psi \in  \dot H^{(d/2+1+\kappa)\vee 2}(\T^d)$ such that $A_{\mu_0}A^{-1}_{\mu_0}\psi = \psi$ almost everywhere, and $
\|A^{-1}_{\mu_0}\psi\|_{H^{(d/2+1+\kappa)\vee 2}}\lesssim  \|\psi\|_{H^{(d/2-1+\kappa)_+}}$ for any $\kappa>0$ small enough. Define its truncated Fourier series
$$
    \gamma_T:=\sum_{k\in\Z^d\backslash\{0\} : |k|_\infty\le K}
    \langle A^{-1}_{\mu_0}\psi,e_k\rangle_2 e_k,
$$
for $K = K_T \simeq T^{1/(2s+d)}$. Using the Fourier characterization of Sobolev spaces, it holds that $\gamma_T\in \dot{H}^{r}(T^d)$ for all $r\ge 0$ with the estimate
\begin{equation}\label{Eq:Fourier_norm}
\|\gamma_T\|_{H^{r}} \lesssim K^{\max\{ r-(d/2+1+\kappa)\vee 2,0\} } \|A_{\mu_0}^{-1}\psi\|_{H^{(d/2+1+\kappa)\vee 2}} \lesssim
    \begin{cases}
        T^{\frac{(r-d/2-1-\kappa)_+}{2s+d}}, & \text{if } d\geq 2,\\
        T^{\frac{(r-2)_+}{2s+1}}, & \text{if } d=1,
    \end{cases}
\end{equation}
using which one can readily verify \eqref{Eq:GammaTNorms}. Similarly, \eqref{Eq:GammaTApprox} holds since $\| A_{\mu_0}^{-1}\psi - \gamma_T\|_{H^1} \lesssim K^{-(d/2+\kappa)\vee 1} $ $\|A_{\mu_0}^{-1}\psi\|_{H^{(d/2+1+\kappa)\vee 1}} \lesssim T^{-\frac{(d/2+\kappa)\vee 1}{2s+d}}$. We may thus apply Theorem \ref{Theo:GaussBvM} to the periodic Matérn base prior \eqref{Eq:Kper} and any $B_0\in \dot H^{s+1}(\T^d)$, $s>3d/2$, for any functional $\Psi:\dot{C}^2(\T^d) \to \R$ admitting expansion \eqref{Eq:psi_exp} with representor $\psi \in \dot{H}^t(\T^d)$, $t>(d/2-1)_+$, and remainder satisfying \eqref{Eq:remainder_text}.
\end{example}

The above priors are equivalent to mean-zero Gaussian processes with covariance operator equal to an inverse power of the Laplacian \cite{PSvZ13,vWvZ16}, for which posterior inference
based on discrete data can be computed using finite element methods \cite{PPRS12}. Another approach is to truncate a Gaussian series expansion in a suitable basis. We illustrate this in the following example.

\begin{example}[Truncated wavelet series]\label{Ex:WavSeries}
Let $\{\Phi_{lr}, \ l\in\{-1,0\}\cup\N,\ r=0,\dots,(2^{ld}-1)\vee0 \}$ be an orthonormal tensor product wavelet basis of $L^2(\T^d)$, obtained from $S$-regular periodised Daubechies wavelets in $L^2(\T)$; see Section 4.3 in \cite{ginenickl2016} for details. Consider a base truncated Gaussian wavelet series prior
\begin{equation}
\label{Eq:WavPrior}
    \Pi_W := \Lcal(W),
    \quad W(x) := \sum_{l=0}^J\sum_r 2^{-l(s+1)}g_{lr}\Phi_{lr}(x)
    \quad g_{lr}\iid N(0,1),
    \quad x\in\T^d, 
\end{equation}
for some $s>d$ and $J=J_T\in\N$ such that $2^{J_T}\simeq T^{1/(2s+d)}$. As shown in Example 2.2 of \cite{giordano2022nonparametric}, $\Pi_W$ satisfies Condition \ref{Condition:BaseGP} with support equal to the wavelet approximation space $V_J:=\textnormal{span}(\Phi_{lr}, \ l\in\{0\}\cup\N,\ r=0,\dots,(2^{ld}-1)\vee0)$, and has RKHS $\H_W=V_J$ with $\|h\|_{\H_W} = \|h\|_{H^{s+1}}$. If $B_0\in\dot{C}^{s+1}(\T^d)$, 
then by standard wavelet approximation properties (e.g.~Section 4.3 in \cite{ginenickl2016}), the projection $B_{0,T}:=\sum_{l =0}^ J\sum_r\langle B_0,\Phi_{lr}\rangle_2\Phi_{lr}$ satisfies $\|B_{0,T}\|_{H^{s+1}}\le \|B_0\|_{H^{s+1}}<\infty$ and $\|B_0 - B_{0,T}\|_{C^1}\lesssim 2^{-Js} \simeq T^{-s/(2s+d)} =\eps_T$
as required for Theorem \ref{Theo:GaussBvM}. 

%

Let $\Psi:\dot{C}^2(\T^d) \to \R$ be a functional admitting  expansion \eqref{Eq:psi_exp} with representor $\psi \in \dot{H}^t(\T^d)$ for some $t>(d/2-1)_+$ and remainder satisfying \eqref{Eq:remainder_text}. Since $B_0\in \dot C^{s+1}(\T^d)$ with $s>d$, arguing as in Example \ref{Ex:Matern} implies that for sufficiently small $\kappa>0$, there exists a unique element $A^{-1}_{\mu_0}\psi \in  \dot H^{(d/2+1+\kappa)\vee 2}(\T^d)$ such that $
\|A^{-1}_{\mu_0}\psi\|_{H^{(d/2+1+\kappa)\vee 2}}\lesssim  \|\psi\|_{H^{(d/2-1+\kappa)_+}}$. Consider the truncated wavelet series
$$
    \gamma_T:=\sum_{l=0}^ J\sum_r
    \langle A^{-1}_{\mu_0}\psi,\Phi_{lr}\rangle_2\Phi_{lr},
$$
which satisfies $\gamma_T\in \dot H^{r}(T^d)$ for all $r\ge 0$ and the norm estimate \eqref{Eq:Fourier_norm}. As in Example \ref{Ex:Matern}, one can verify \eqref{Eq:GammaTNorms} and \eqref{Eq:GammaTApprox}. Theorem \ref{Theo:GaussBvM} therefore applies to the truncated Gaussian wavelet series prior \eqref{Eq:WavPrior} and any $B_0\in  \dot C^{s+1}(\T^d)$, $s>d$, for any functional $\Psi:\dot{C}^2(\T^d) \to \R$ with expansion \eqref{Eq:psi_exp}, representor $\psi \in \dot{H}^t(\T^d)$, $t>(d/2-1)_+$, and remainder satisfying \eqref{Eq:remainder_text}.
\end{example}

\subsection{Besov-Laplace priors}
\label{Sec:LaplPriors}

We next consider Besov-Laplace priors, constructed via random wavelet expansions with i.i.d.~random coefficients following the Laplace (or double exponential) distribution with density $\lambda(z) =e^{-|z|}/2$, $z\in\R$. Specifically, we employ rescaled and truncated priors obtained starting from a base probability measure $\Pi_W$ given by the law of the random function
\begin{equation}
\label{Eq:BaseLaplPrior}
    W(x) := \sum_{l=0}^J\sum_r2^{-l(s+1-d/2)}g_{lr}\Phi_{lr}(x),
    \qquad g_{lr}\iid \lambda,
    \qquad x\in\T^d,
\end{equation}
with $J=J_T\in\N$ such that $2^{J}\simeq T^{1/(2s+d)}$, and where $\{\Phi_{lr}, \ l\in\{-1,0\}\cup\N,\ r=0,\dots,(2^{ld}-1)\vee0 \}$ is an $S$-regular orthonormal tensor product wavelet basis of $L^2(\T^d)$ (with $S\in\N$ fixed but arbitrarily large), cf.~Example \ref{Ex:WavSeries}. This extends the construction of Gaussian wavelet series priors, and represents a specific instance of the more general class of $p$-exponential (or Besov) priors \cite{LSS09,agapiou2021rates,AgapiouSavva2024}, which prescribe random wavelet coefficients with tail behaviour between the Laplace (corresponding to $p=1$) and the Gaussian distribution ($p=2$). Besov-Laplace priors have recently enjoyed significant popularity within the inverse problems and imaging communities \cite{LSS09,DHS12,KLNS12,BH15,AW21,Magra2025} since they exhibit attractive sparsity-promoting and edge-preserving properties, while also maintaining a log-concave structure favourable to computation and theoretical analysis. Note that in \eqref{Eq:BaseLaplPrior}, the wavelet coefficient corresponding to the constant function $\Phi_{-10}\equiv 1$ has been set to zero to ensure the zero-integral identifiability condition.

In a similar spirit to the Gaussian priors studied in Section \ref{Sec:GP}, we construct rescaled truncated Besov-Laplace priors $\Pi=\Pi_T$, given by, for $W\sim\Pi_W$ as in \eqref{Eq:BaseLaplPrior},
\begin{equation}
\label{Eq:RescaledBesovPrior}
    B(x)=\frac{W(x)}{T^{d/(2s+d)}},
    \qquad x\in\T^d.
\end{equation}
In the present setting, these priors are known to achieve minimax-optimal contraction rates \cite{giordano2022nonparametric} and the following result shows that they also satisfy a semiparametric BvM.

\begin{theorem}\label{Theo:LaplBvM}
For $s>d+(d/2)\vee 2$, let $\Pi$ be the rescaled Besov-Laplace prior from \eqref{Eq:RescaledBesovPrior}, and suppose $B_0\in \dot{H}^{s+1}(\T^d)$. Let $\Psi:\dot{C}^2(\T^d) \to \R$ be a functional admitting expansion \eqref{Eq:psi_exp} with representor $\psi \in \dot C^{t}(\T^d)$ for some $t>(d/2-1)_+$ and remainder satisfying \eqref{Eq:remainder_text}. Then the semiparametric Bernstein--von Mises holds for $\Psi$.
\end{theorem}

The conditions on $\Pi$ and $B_0$ in Theorem \ref{Theo:LaplBvM} match those in Theorem 2.4 of \cite{giordano2022nonparametric}, which establishes posterior contraction for these priors.
In Theorem \ref{Theo:LaplBvM}, we again restrict for concreteness to functionals $\Psi$ with remainder satisfying \eqref{Eq:remainder_text}. The (slightly stronger) H\"older-smoothness condition imposed on $\psi$, as opposed to the Sobolev requirements in Theorems \ref{thm:BvM} and \ref{Theo:GaussBvM}, is due to the necessity of approximating the LAN representor $A_{\mu_0}^{-1}\psi$ in the stronger $\|\cdot\|_{W^{1,\infty}}$-norm (cf.~the second display in  \eqref{Eq:gamma_T_conditions}), which we approached by invoking H\"older-type regularity estimates for elliptic PDEs and wavelet approximation properties in sup-norm. We note that, however, the full smoothness range $t>(d/2-1)_+$ is allowed in Theorem \ref{Theo:LaplBvM}.

%
%
%
%
%

\section{Numerical illustrations}
\label{Sec:Numerics}

In this section, we illustrate our theoretical findings in finite sample sizes via numerical simulation studies. We first consider Gaussian priors, which were shown to be conjugate in \cite{giordano2022nonparametric}, with explicit formulae for the posterior mean and covariance, cf.~\eqref{Eq:Conjugate} below. We then move to non-conjugate Besov-Laplace priors which require Markov chain Monte Carlo (MCMC) methods to approximately sample from the associated posterior distributions.

For three different periodic `ground truths' $B$ defined on the bi-dimensional torus $\T^2 =(0,1]^2$, cf.~eq.~\eqref{Eq:Truths} and Figure \ref{Fig:Truths} (top row), we simulate the continuous diffusion trajectories $X^T$ for increasing $T$, and compute (suitable approximations of) the posterior distributions of $B|X^T$ based on Gaussian and Besov-Laplace priors. We then select three nonlinear real-valued functionals $\Psi$ and obtain the posteriors for $\Psi(B)$, reporting coverage scores, lengths of the $95\%$-credible intervals and the estimation errors of the posterior means. We run each simulation 250 times, recording average values and standard deviations when relevant. All experiments were carried out on a MacBook Pro with M1 processor and 8GB RAM. The MATLAB code to reproduce these simulations is available at: 
\url{https://github.com/MattGiord/Rev-Diff}.

\paragraph{Data generation.} We take the three true potential energy fields on $\T^2$ to be:
\begin{equation}
\label{Eq:Truths}
\begin{split}
    B^{(1)}(x,y) &=e^{-(7.5x-5)^2-(7.5y-5)^2}
    +e^{-(7.5x-2.5)^2-(7.5y-2.5)^2};\\
    B^{(2)}(x,y)&=2+e^{-(7.5x-5)^2-(7.5y-5)^2}
    -e^{-(7.5x-2.5)^2-(7.5y-2.5)^2};\\
    B^{(3)}(x,y)&=e^{-(7.5x-5.5)^2-(7.5y-5.5)^2}
    +0.75e^{-(5x-1.25)^2-(7.5y-5.5)^2}\\
    &\quad+1.25e^{-(7.5x-5.5)^2-(5y-1.25)^2}
    +e^{-(7.5x-2)^2-(7.5y-2)^2}.
\end{split}
\end{equation}
For each given $B$, we simulate the continuous trajectory $(X_t:t\ge0)$ via the Euler-Maruyama scheme
$$
    x_{r+1} = x_r+ \nabla B(x_r)\delta_t + \sqrt{\delta_t}W_r,
    \qquad W_r\iid N(0,I_2), 
    \qquad r\ge0.
$$
Across all the experiments, we set the time stepsize to $\delta_t = 10^{-4}$, resulting in realistic approximations of the continuous diffusion paths, cf.~Figure \ref{Fig:RevDiff} (right). The (uninfluential) initial condition was fixed to  $x_0=(1,1)$. We repeat the Euler-Maruyama scheme for $5\times 10^5$ and $10^6$ times, yielding time horizons $T=50$ and $T=100$, respectively. 

\paragraph{Functionals.} We illustrate the empirical semiparametric inference performance over a range of nonlinear functionals covered by our theoretical results, specifically $\Psi_1(B) = \int B(x)^2 dx$ (Example \ref{Ex:square}), $\Psi_2(B) = \int B(x)^4 dx$ (Example \ref{Ex:power} with $q=4$) and $\Psi_3(B) = \int \mu_B(x) \log \mu_B(x)dx$ (Example \ref{Ex:entropy}).

\begin{figure}[t]
\centering
\includegraphics[width=4.5cm]{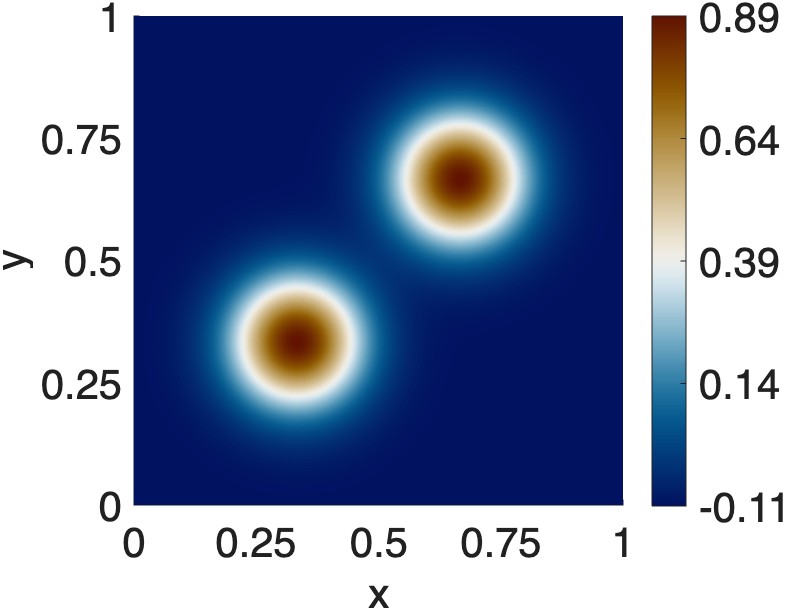}
\includegraphics[width=4.5cm]{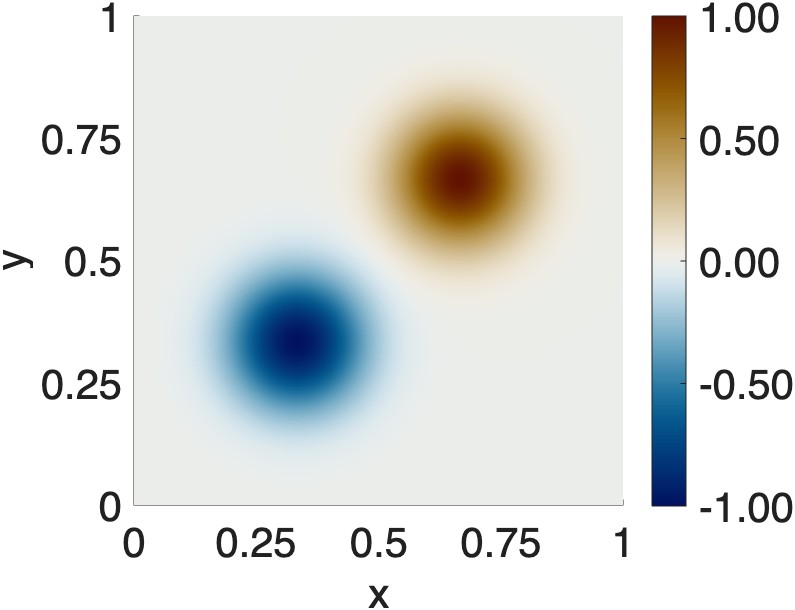}
\includegraphics[width=4.5cm]{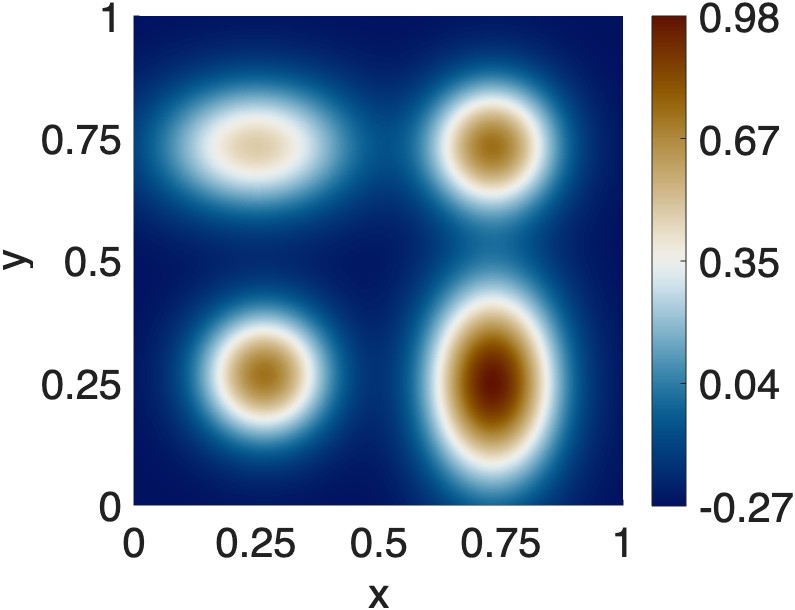}
\includegraphics[width=4.5cm]{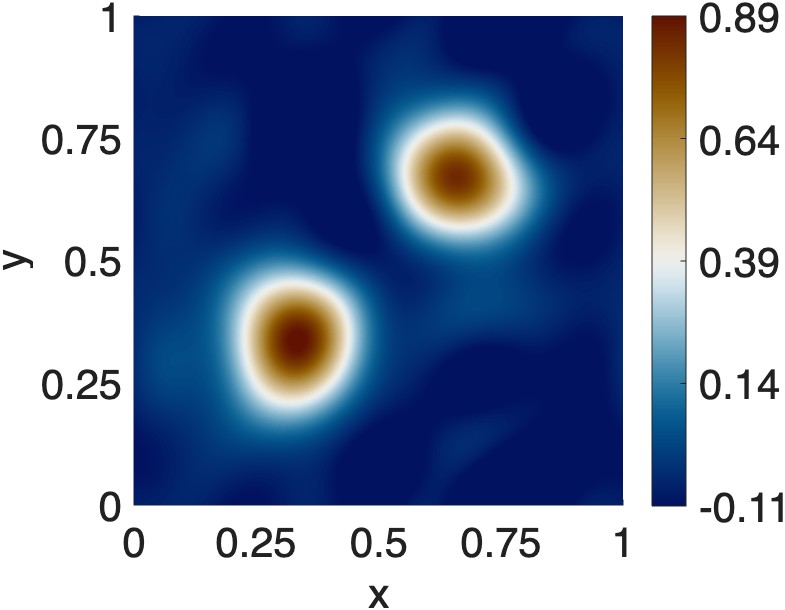}
\includegraphics[width=4.5cm]{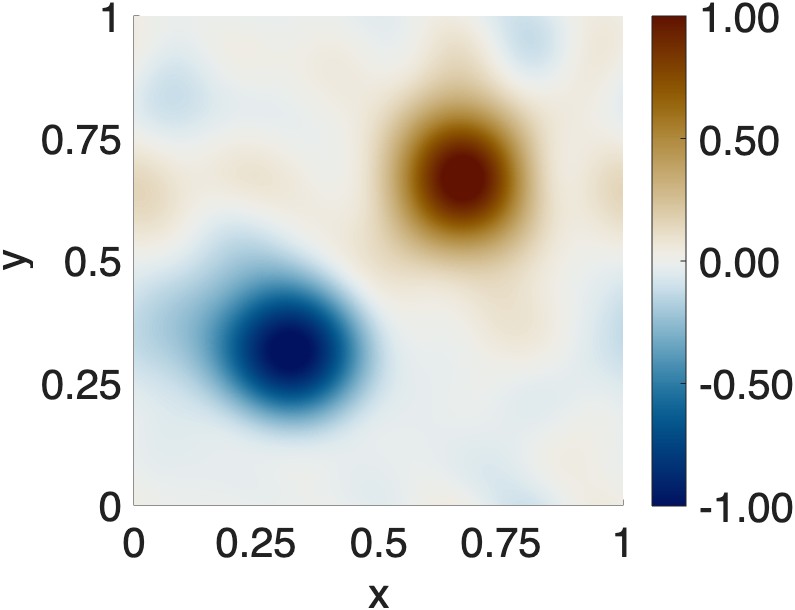}
\includegraphics[width=4.5cm]{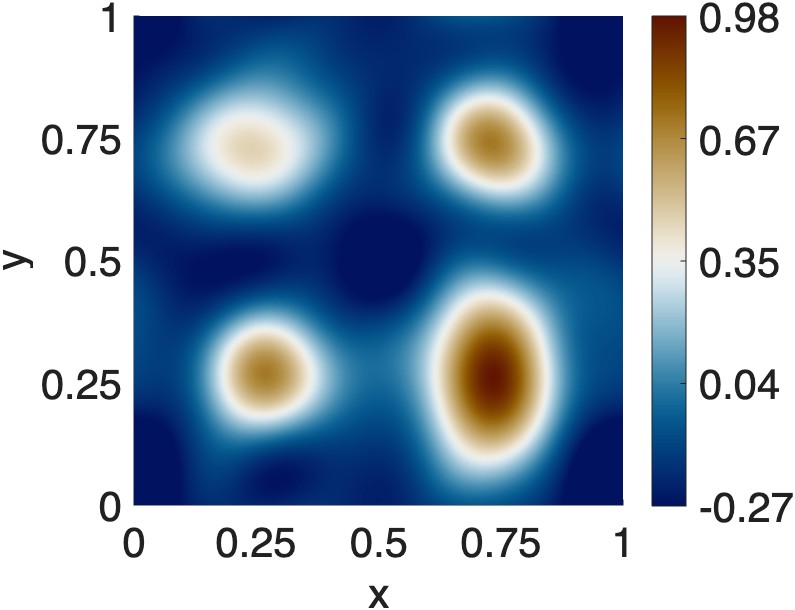}
\includegraphics[width=4.5cm]{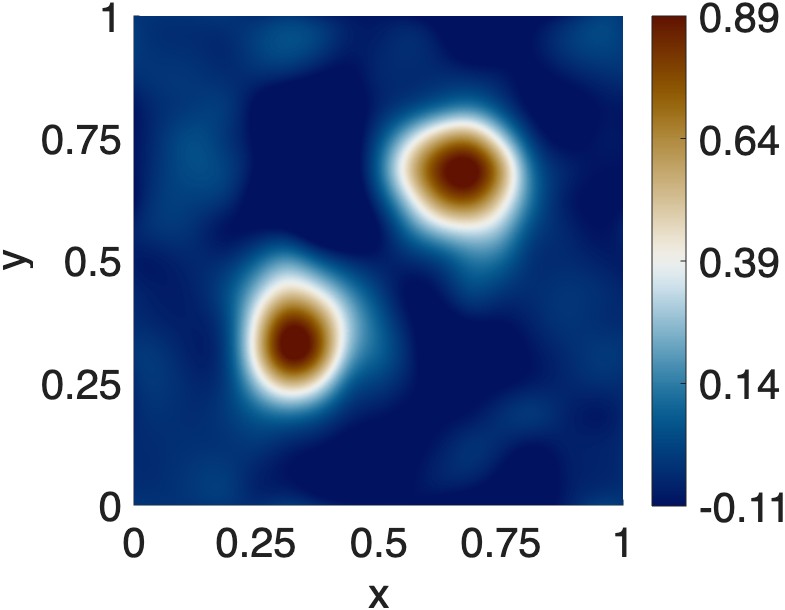}
\includegraphics[width=4.5cm]{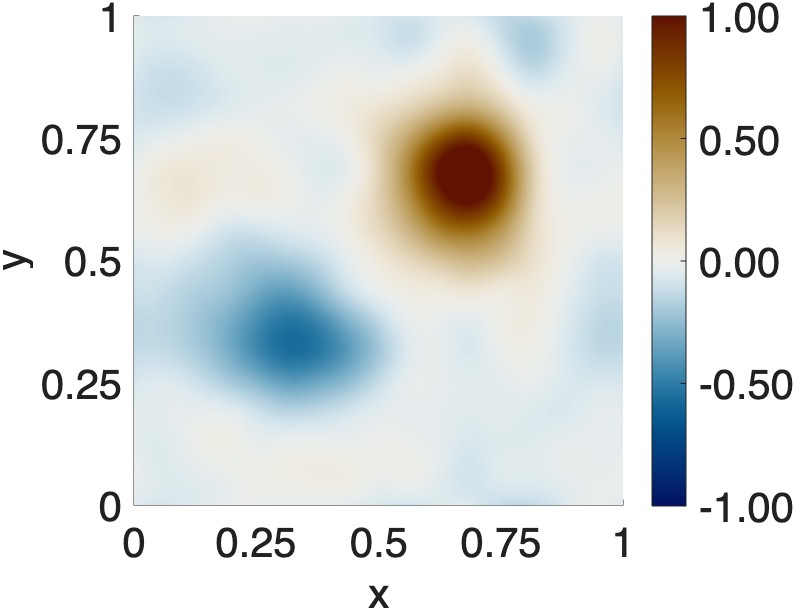}
\includegraphics[width=4.5cm]{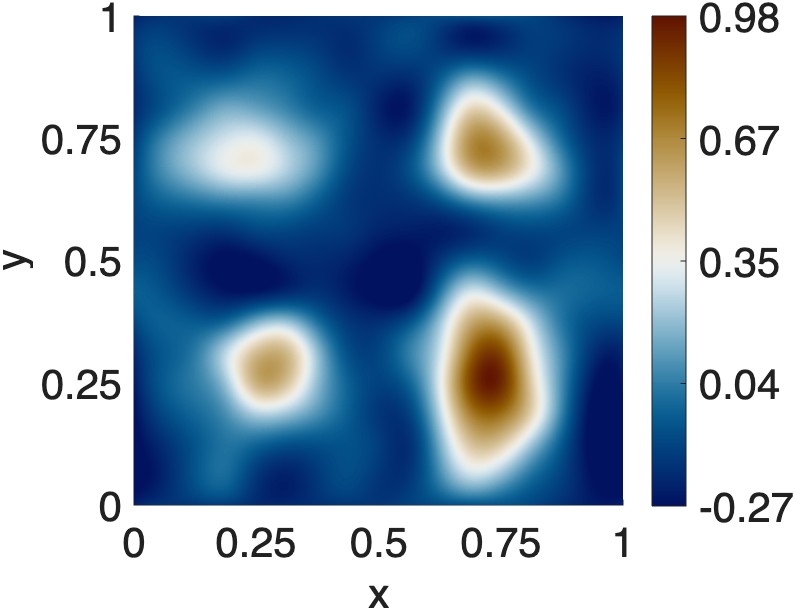}
\caption{Top row: the three potential energy fields $B^{(i)}$, $i=1,2,3$, from \eqref{Eq:Truths}. Central row: the corresponding posterior means $\bar B^{(i)}_{T}$, $i=1,2,3$, at time $T=100$, based on the Matérn process prior \eqref{Eq:ConcrGaussPrior}. The relative $L^2$-estimation errors $\|B^{(i)} - \bar B^{(i)}_T\|_2/\|B^{(i)}\|_2$ are equal to 0.21, 0.023, 0.12, respectively. Bottom row: posterior means based on the Besov-Laplace prior. Relative $L^2$-estimation errors: 0.22, 0.028, 0.17, respectively}.
\label{Fig:Truths}
\end{figure}

%
%
%

\subsection{Experiments with Gaussian priors}
\label{Sec:GaussNumerics}

\paragraph{Prior specification and conjugate formulae.} We employ the periodic Matérn process prior from Example \ref{Ex:Matern}, modelling
\begin{equation}
\label{Eq:ConcrGaussPrior}
    B(x)= \sum_{k\in\Z^2\backslash\{0\}: |k|_\infty\le K}
    v_kg_k e_k(x),
    \qquad g_k\iid N(0,1),
    \qquad x\in\T^2,
\end{equation}
where $v_k = (1+|k|^2)^{-(s+1)/2}T^{-1/(2s+2)}$ with $s=3=3d/2$, $\{e_k: k\in\Z^2\}$ is the bi-dimensional Fourier basis, and $K\in\N$ is a sufficiently high truncation level to ensure that all discretisation errors are negligible relative to the statistical ones. The factor $T^{-1/(2s+2)}$ corresponds to the rescaling from \eqref{Eq:RescaledGP}, required in our theoretical analysis. In the present experimental set-up, this induces only a slight variance adjustment, even at the larger time horizon $T=100$, where $T^{-1/(2s+2)} = 0.56$. In fact, throughout all simulations, we found the rescaling to have little impact on empirical performance. It was nevertheless incorporated in the prior specification for the results presented below to faithfully illustrate the procedures considered in our theorems. See the related discussion in Section \ref{Sec:Discussion}.

Identifying any function $B=\sum_{k\in\Z^2\backslash\{0\}: |k|_\infty\le K}B_ke_k$ with its Fourier coefficient vector $B=(B_k)_k$, this corresponds to assigning the multivariate Gaussian prior with diagonal covariance matrix,
$$
    B\sim N(0,\Upsilon), \qquad \Upsilon = \text{diag}[(v_k^2)_k],
$$
whence the conjugate computation in Section 2.3.3 of \cite{giordano2022nonparametric} yields the Gaussian posterior distribution
\begin{equation}\label{Eq:Conjugate}
    B|X^T\sim N((\Sigma + \Upsilon^{-1})^{-1}H,(\Sigma + \Upsilon^{-1})^{-1}),
\end{equation}
where
$$
    \Sigma = \left[ \int_0^T \nabla e_k(X_t).\nabla e_{k'}(X_t)dt\right]_{k,k'},
    \qquad H = \left[\int_0^T \nabla e_k(X_t).dX_t \right]_k.
$$
For each simulated continuous trajectory, we numerically compute the above matrix $\Sigma$ and vector $H$ by approximating the integrals with Riemann sums, which we then use to evaluate the posterior mean and covariance matrix in \eqref{Eq:Conjugate}.

For nonlinear functionals $\Psi:\dot{C}^2(\T^2)\to\R$, the plug-in posterior distribution of $\Psi(B)|X^T$ is generally non-Gaussian and not available in closed form. To compute the corresponding posterior mean and credible intervals, we use Monte Carlo approximation, which is straightforward to implement by sampling $B_{(1)},\dots,B_{(M)}$ from the explicitly available posterior distribution \eqref{Eq:Conjugate} of $B|X^T$, and computing samples $\Psi(B_{(1)}),\dots,$ $\Psi(B_{(M)})$. For each Monte Carlo approximation, we used $M=1000$ samples.

\paragraph{Results.} The central row of Figure \ref{Fig:Truths} shows the obtained posterior means at time $T=100$ for the three considered true potentials $B^{(1)}$, $B^{(2)}$, $B^{(3)}$. The estimates display an excellent visual quality in the reconstruction, correctly recovering the geometric features of the ground truths.

For the three nonlinear functionals $\Psi_1$, $\Psi_2$, $\Psi_3$, Table \ref{Tab:Results} reports the coverage and average lengths of the $95\%$ credible intervals, at times $T=50,100$, each obtained through 250 repeated experiments. For each combination of ground truth $B$ and functional $\Psi$, the obtained coverages are higher at the larger time horizon; in particular, for $T=100$, they are very close to the nominal level $95\%$ predicted by Theorem \ref{Theo:GaussBvM}. The average lengths of the credible interval are also seen to decrease as the time horizon increases. Table \ref{Tab:Results} further reports the average estimation errors for the posterior means, with the associated standard deviation relative to the 250 experiments. As expected from our theoretical results, the estimation errors (and standard deviations) become smaller across the board as $T$ increases. Among the three functionals, $\Psi_1(B)=\int B^2dx$ exhibits the largest finite-sample deviations from nominal coverage, particularly at $T=50$. This behaviour is consistent across all three ground truths and appears to reflect a larger finite-sample bias. Nevertheless, the coverage rapidly improves as T increases and approaches the nominal level at $T=100$.

\begin{table}[t]
\caption{Coverage scores (and average lengths) of the $95\%$ credible intervals, and average estimation errors (with standard deviations) of the posterior mean for three nonlinear functionals $\Psi(B)$, obtained over 250 repeated experiments, under the Gaussian prior \eqref{Eq:ConcrGaussPrior}.}
\label{Tab:Results}
\centering
\setlength{\tabcolsep}{4pt}
\begin{tabular}{cc||cc||cc}
\hline
 &  & \multicolumn{2}{c||}{\textbf{Coverage (Length)}} 
 & \multicolumn{2}{c}{\textbf{Error (SD)}} \\
\textbf{Truth} & \(\boldsymbol{\Psi(B)}\) 
& \textbf{\(T=50\)} & \textbf{\(T=100\)}
& \textbf{\(T=50\)} & \textbf{\(T=100\)} \\
\hline
\hline
\(B^{(1)}\) & \(\Psi_1\) & 0.61 (0.033) & 0.91 (0.023) & 0.013 (0.008) & 0.006 (0.004) \\
        & \(\Psi_2\) & 0.88 (0.022) & 0.95 (0.014) & 0.005 (0.008) & 0.002 (0.002) \\
        & \(\Psi_3\) & 0.82 (0.024) & 0.96 (0.017) & 0.008 (0.006) & 0.004 (0.002) \\
\hline
\(B^{(2)}\) & \(\Psi_1\) & 0.72 (0.042) & 0.92 (0.025) & 0.016 (0.007) & 0.007 (0.006) \\
        & \(\Psi_2\) & 0.95 (0.038) & 0.96 (0.016) & 0.008 (0.007) & 0.005 (0.005) \\
        & \(\Psi_3\) & 0.86 (0.029) & 0.93 (0.016) & 0.008 (0.005) & 0.005 (0.004) \\
\hline
\(B^{(3)}\) & \(\Psi_1\) & 0.66 (0.029) & 0.91 (0.024) & 0.013 (0.009) & 0.005 (0.004) \\
        & \(\Psi_2\) & 0.84 (0.021) & 0.96 (0.025) & 0.009 (0.007) & 0.003 (0.003) \\
        & \(\Psi_3\) & 0.76 (0.021) & 0.97 (0.019) & 0.008 (0.006) & 0.003 (0.003) \\
\hline
\end{tabular}
\end{table}

\begin{figure}[t]
\centering
\includegraphics[width=4.5cm,height=3.5cm]{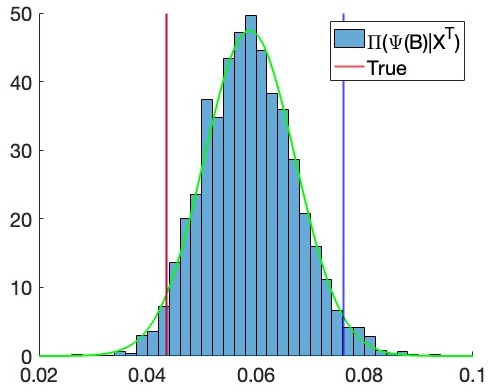}
\includegraphics[width=4.5cm,height=3.5cm]{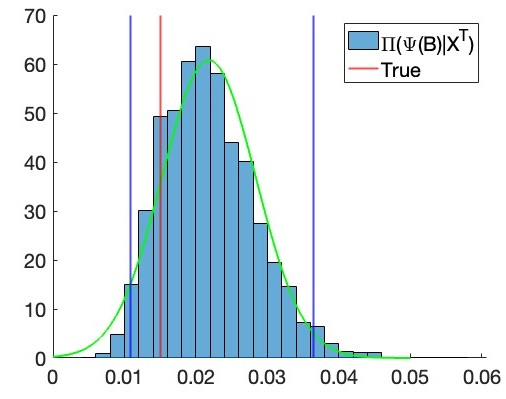}
\includegraphics[width=4.5cm,height=3.5cm]{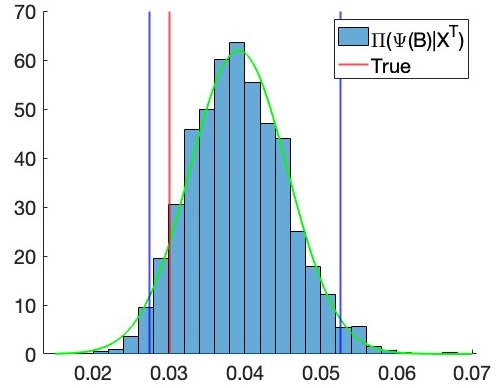}
\includegraphics[width=4.5cm,height=3.5cm]{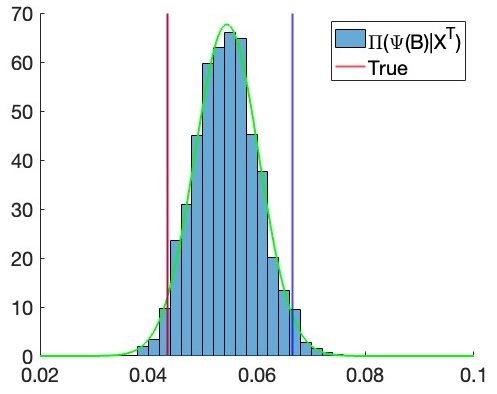}
\includegraphics[width=4.5cm,height=3.5cm]{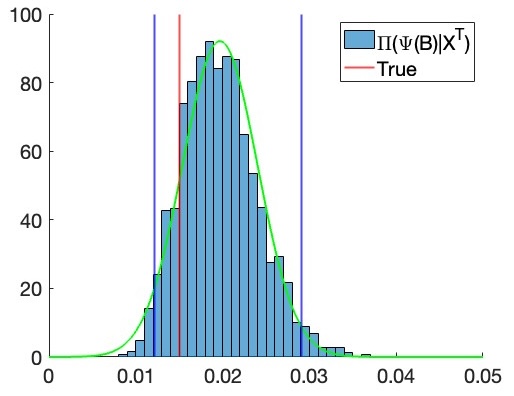}
\includegraphics[width=4.5cm,height=3.5cm]{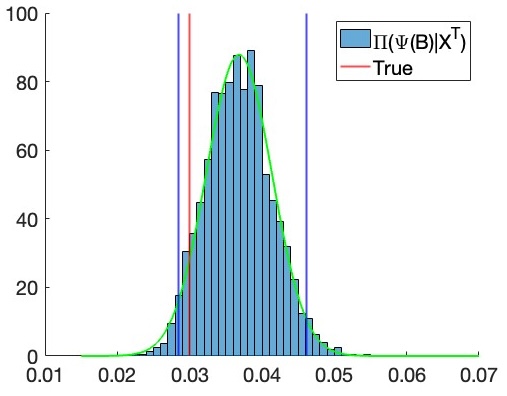}
\caption{Top row: plug-in posterior distributions of $\Psi_i(B)|X^T$, $i=1,2,3$, at $T=50$, based on the Gaussian prior \eqref{Eq:ConcrGaussPrior}. Bottom row: plug-in posteriors at $T=100$. The vertical red lines indicate the `ground truths' $\Psi_i(B^{(1)})$, for $B^{(1)}$ as in Figure \ref{Fig:Truths} (top-left). The vertical blue lines identify the $95\%$ credible intervals. The green line corresponds to a normal PDF centred at the posterior mean and with variance equal to the posterior variance.}
\label{Fig:PlugInPost}
\end{figure}

For the ground truth $B^{(1)}$, Figure \ref{Fig:PlugInPost} compares individual realisations of the plug-in posterior distributions of $\Psi_i(B)|X^T$, $i=1,2,3$, at times $T=50$ and $T=100$, obtained via Monte Carlo approximation. The plot displays a progressively more accurate Gaussian approximation as $T$ increases for the plug-in posterior of all three functionals, in line with the findings of Theorem \ref{Theo:GaussBvM}. For these specific realisations, the $95\%$ credible intervals (vertical blue lines) correctly cover the ground truths.

%
%
%

\subsection{Experiments with Besov-Laplace priors}
\label{Sec:LaplNumerics}

\paragraph{Prior specification.} We employ the rescaled and truncated Besov-Laplace prior from Section \ref{Sec:LaplPriors}, cf.~\eqref{Eq:RescaledBesovPrior}, with $s = 4 = d + 2$. By Bayes' formula, the posterior density for any function $B = \sum_{l\leq J,r} B_{lr}\Phi_{lr}$ is given by $d\Pi(B|X^T)\propto e^{-\Lambda_T(B)}$, where
$$
    \Lambda_T(B)
    =
    -\ell_T(B)
    +
    T^{\frac{d}{2s+d}}
    \sum_{l = 0}^J\sum_r
    2^{\,l(s+1+d/2-d)}
    |B_{lr}|,
$$
and the log-likelihood takes the form
\begin{equation}
\label{Eq:WavLogLik}
    \ell_T(B)
    =
    -\frac12
    \sum_{l,r}
    \sum_{l',r'}
    B_{lr}B_{l'r'}
    \int_0^T
    \nabla\Phi_{lr}(X_t).
    \nabla\Phi_{l'r'}(X_t)\,dt+
    \sum_{l,r}
    B_{lr}
    \int_0^T
    \nabla\Phi_{lr}(X_t).dX_t.   
\end{equation}
Note that the coefficients of the above quadratic form are analogous to the quantities $\Sigma$ and $H$ appearing in the Gaussian posterior computation from Section~\ref{Sec:GaussNumerics}, the only difference being that they are now evaluated relative to the wavelet basis $\{\Phi_{lr}\}$.

\paragraph{Posterior sampling with Besov-Laplace prior.} Besov-Laplace priors give rise to log-concave posterior distributions that are non-conjugate, unlike in the Gaussian case. Therefore, we resort to MCMC methods to draw
approximate posterior samples, from which posterior means and  credible intervals can be computed via their MCMC counterparts. Specifically, we employ the whitened preconditioned
Crank-Nicolson (wpCN) algorithm, which is a dimension-robust technique of Metropolis-Hastings type developed by \cite{chen2018robust} for priors that can be expressed as a transformation of a Gaussian white noise.

For the Besov-Laplace prior from Section \ref{Sec:LaplPriors}, we use that the random function $T^{-d/(2s+d)}W$ from \eqref{Eq:RescaledBesovPrior} is equal in distribution to
$$
    \Theta^{(T)}(\xi)
    :=
    \sum_{l=0}^J\sum_r
    \Theta^{(T)}_{lr}(g_{lr})\Phi_{lr},
$$
where
$$
    \xi(x)
    :=
    \sum_{l=0}^J\sum_r
    g_{lr}\Phi_{lr}(x),
    \qquad x\in\T^d,
    \qquad
    g_{lr}\iid N(0,1),
$$
is a Gaussian white noise process indexed by $\T^d$, and the `whitening transformation' $\Theta^{(T)}$ is defined coefficient-wise by
$$
    \Theta^{(T)}_{lr}(g_{lr})
    :=
    T^{-\frac{d}{2s+d}}
    2^{-l(s+1+d/2-d)}
    \operatorname{sgn}(g_{lr})
    \Bigl[
        -\log\!\bigl(2-2\Phi(|g_{lr}|)\bigr)
    \Bigr],
$$
with $\Phi$ the standard Gaussian cumulative distribution function. Starting from an initial white-noise sample $\xi_{(0)}$ and the corresponding transformed sample $B_{(0)}=\Theta^{(T)}(\xi_{(0)})$, the wpCN algorithm iterates the following steps:

\begin{itemize}

    \item Construct the whitened proposal $\xi^\ast=\sqrt{1-2b}\,\xi_{(m-1)}+\sqrt{2b}\,\zeta$, 
    where \(b\in(0,1/2)\) is a step-size and $\zeta$ is an independent Gaussian white noise.
    \item Set
    $$
        \xi_{(m)}
        =
        \begin{cases}
        \xi^\ast,
        &
        \text{with probability }
        \min\left\{
        1,
        e^{
            \ell_T\!\big(\Theta^{(T)}(\xi^\ast)\big)
            -
            \ell_T(B_{(m-1)})}
        \right\},
        \\[1ex]
        \xi_{(m-1)},
        &
        \text{otherwise}.
        \end{cases}
    $$

    \item Set $B_{(m)}=\Theta^{(T)}(\xi_{(m)})$.
\end{itemize}
The resulting Markov chain is reversible with respect to the posterior distribution \cite{chen2018robust}, and inherits the dimension-robust mixing properties characteristic of pCN-type algorithms \cite{cotter2013mcmc}. Posterior distributions and means for functionals $\Psi(B)$ are then computed from the retained MCMC samples after burn-in.

The evaluation of the acceptance probability in the second wpCN step requires computation of the proposal log-likelihood $\ell_T\!\big(\Theta^{(T)}(\xi^\ast)\big)$. Under the employed parametrisation, this operation is computationally efficient to perform, as it only involves vector-matrix multiplications. Periodic  wavelets are implemented via MATLAB Wavelet Toolbox using a regular grid of 2,048 nodes on $[0,1]^2$. Wavelet gradients are computed using finite differences, and the integrals $\int_0^T\nabla\Phi_{lr}(X_t).\nabla\Phi_{l'r'}(X_t)\,dt$ and $\int_0^T\nabla\Phi_{lr}(X_t).dX_t$ appearing in $\ell_T\!\big(\Theta^{(T)}(\xi^\ast)\big)$ are approximated once at the beginning of the procedure via numerical quadrature, similarly to the Gaussian case.

\paragraph{Results.}
The bottom row of Figure \ref{Fig:Truths} shows the obtained posterior means at time $T=100$ for the three considered potentials $B^{(1)}$, $B^{(2)}$, $B^{(3)}$. The estimates are visually close to the ones obtained under the Gaussian prior in Section \ref{Sec:GaussNumerics}. We note slightly larger estimation errors, which may be due to the additional numerical approximations required by the wpCN method, including the finite-difference evaluation of wavelet gradients.

For the three nonlinear functionals $\Psi_1,\Psi_2,\Psi_3$, Table~\ref{Tab:LaplResults} reports the coverage and average lengths of the $95\%$-credible intervals arising from the employed Besov-Laplace prior, obtained over 250 repeated experiments, together with the corresponding average estimation errors of the plug-in posterior means. The results are again qualitatively similar to those observed under the
Gaussian prior in Section~\ref{Sec:GaussNumerics}. Across all three ground truths, empirical coverages increase with the observation horizon and are close to the nominal $95\%$ level at $T=100$, in agreement with the Bernstein--von Mises phenomenon predicted by Theorem~\ref{Theo:LaplBvM}. Likewise, the average lengths of the credible intervals and the estimation errors decrease as $T$ increases. 

For the ground truth $B^{(1)}$, Figure~\ref{Fig:LaplPlugInPost}
compares individual realisations of the posterior of $\Psi_i(B)$, $i=1,2,3$, at times $T=50$ and $T=100$, obtained from the retained wpCN samples. While more noticeable deviations from Gaussianity remain at the smaller time horizon $T=50$, the marginal posteriors become progressively closer to Gaussian distributions as the amount of data increases. For these realisations, the associated credible intervals are also seen to correctly cover the true functional values.

\begin{table}[t]
\caption{Coverage scores (and average lengths) of the $95\%$ credible intervals, and average estimation errors (with standard deviations) of the posterior means for three nonlinear functionals $\Psi(B)$, obtained over 250 repeated experiments, under the Besov-Laplace prior.}
\label{Tab:LaplResults}
\centering
\setlength{\tabcolsep}{4pt}
\begin{tabular}{cc||cc||cc}
\hline
 &  & \multicolumn{2}{c||}{\textbf{Coverage (Length)}}
 & \multicolumn{2}{c}{\textbf{Error (SD)}} \\
\textbf{Truth} & \(\boldsymbol{\Psi(B)}\)
& \textbf{\(T=50\)} & \textbf{\(T=100\)}
& \textbf{\(T=50\)} & \textbf{\(T=100\)} \\
\hline
\hline
\(B^{(1)}\) & \(\Psi_1\) & 0.66 (0.036) & 0.90 (0.021) & 0.011 (0.007) & 0.008 (0.003) \\
        & \(\Psi_2\) & 0.89 (0.021) & 0.94 (0.017) & 0.004 (0.004) & 0.003 (0.004) \\
        & \(\Psi_3\) & 0.79 (0.021) & 0.95 (0.019) & 0.009 (0.008) & 0.006 (0.003) \\
\hline
\(B^{(2)}\) & \(\Psi_1\) & 0.73 (0.052) & 0.93 (0.028) & 0.018 (0.006) & 0.006 (0.004) \\
        & \(\Psi_2\) & 0.94 (0.037) & 0.97 (0.017) & 0.009 (0.008) & 0.006 (0.006) \\
        & \(\Psi_3\) & 0.85 (0.031) & 0.93 (0.017) & 0.007 (0.004) & 0.004 (0.003) \\
\hline
\(B^{(3)}\) & \(\Psi_1\) & 0.67 (0.032) & 0.92 (0.027) & 0.015 (0.011) & 0.005 (0.004) \\
        & \(\Psi_2\) & 0.83 (0.021) & 0.96 (0.024) & 0.008 (0.007) & 0.002 (0.004) \\
        & \(\Psi_3\) & 0.79 (0.023) & 0.96 (0.023) & 0.005 (0.003) & 0.002 (0.001) \\
\hline
\end{tabular}
\end{table}

\begin{figure}[t]
\centering
\includegraphics[width=4.5cm,height=3.5cm]{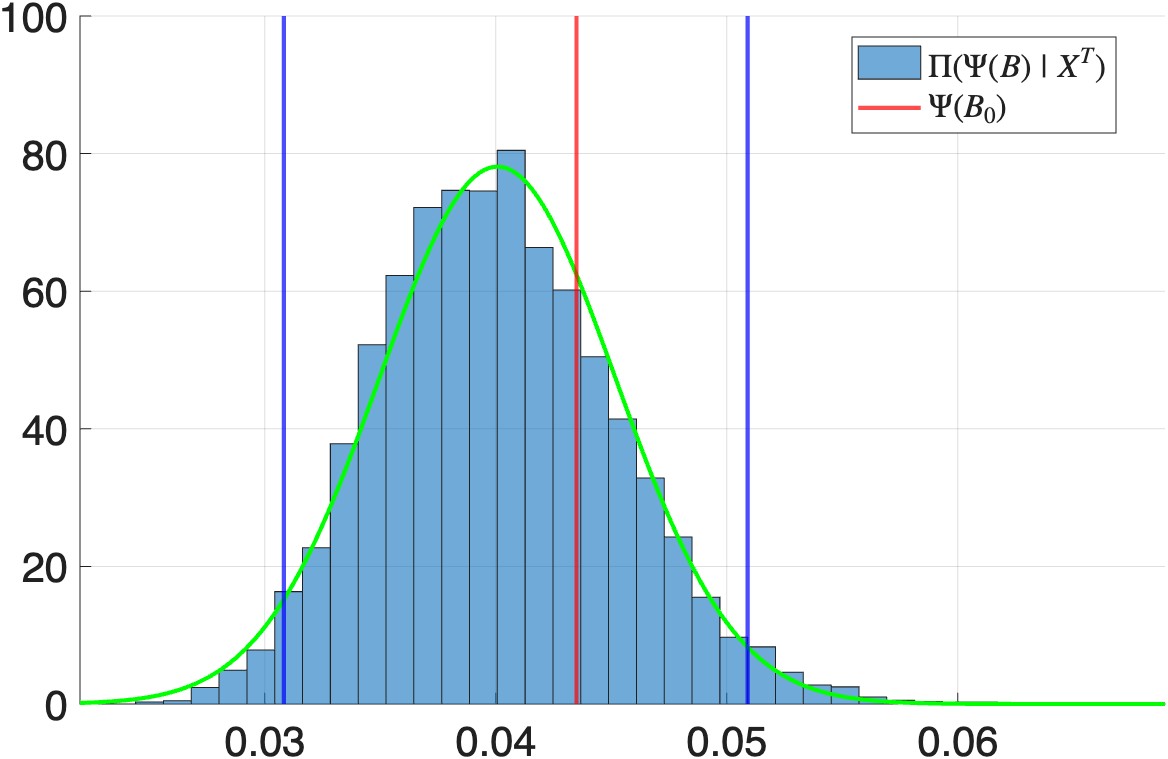}
\includegraphics[width=4.5cm,height=3.5cm]{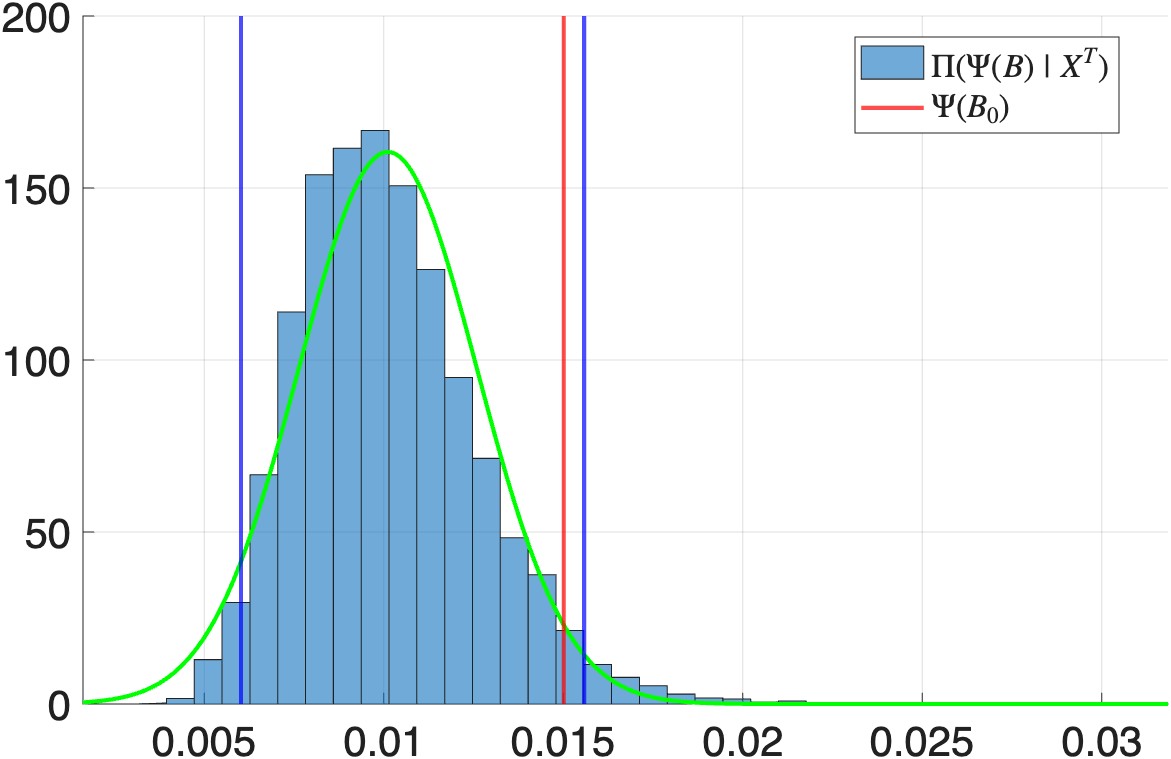}
\includegraphics[width=4.5cm,height=3.5cm]{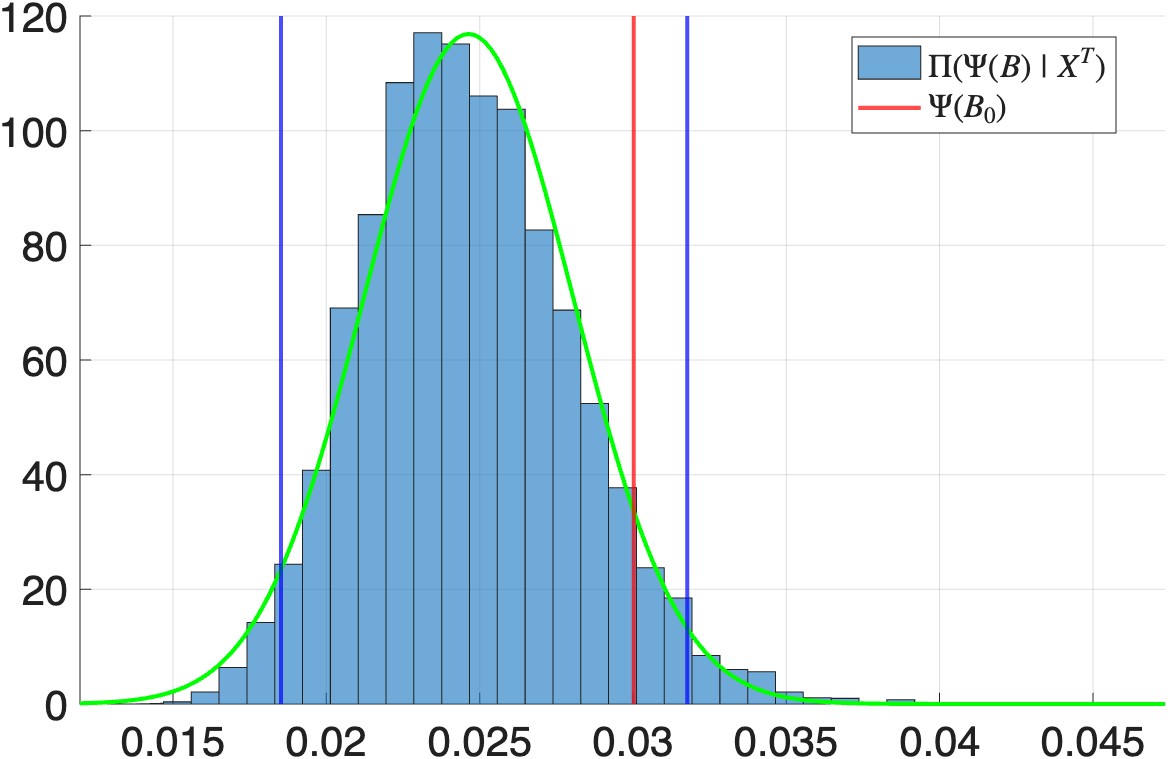}
\includegraphics[width=4.5cm,height=3.5cm]{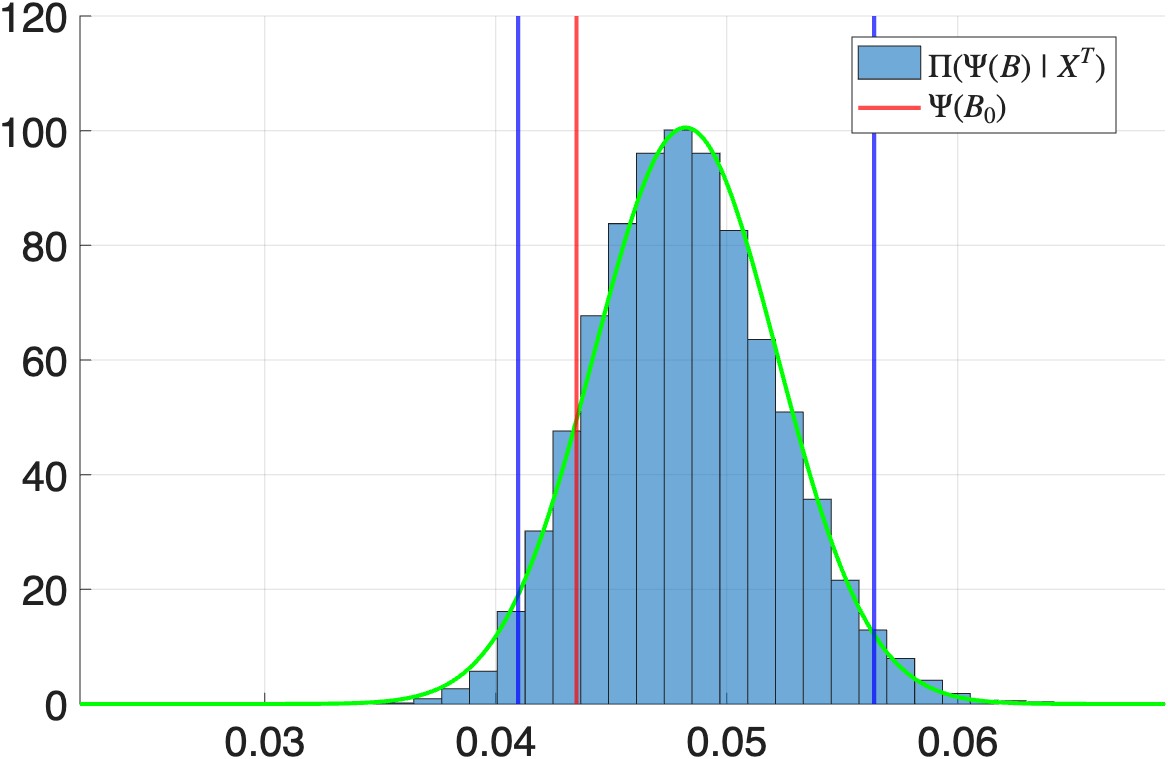}
\includegraphics[width=4.5cm,height=3.5cm]{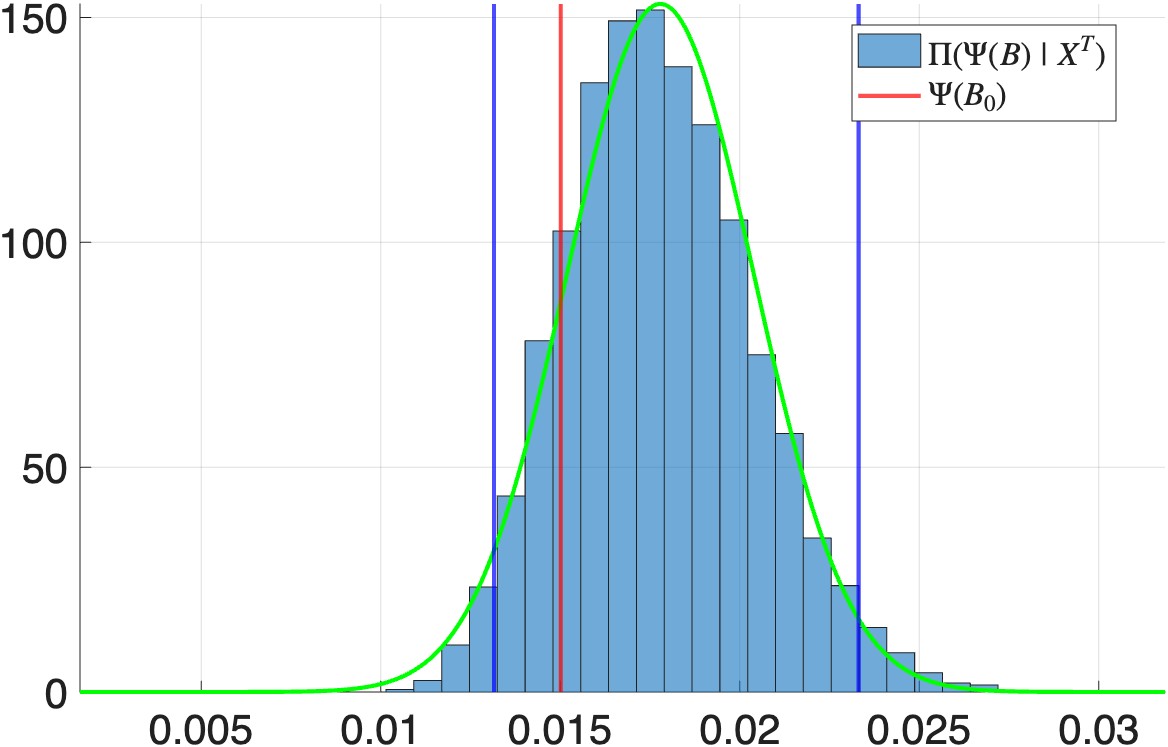}
\includegraphics[width=4.5cm,height=3.5cm]{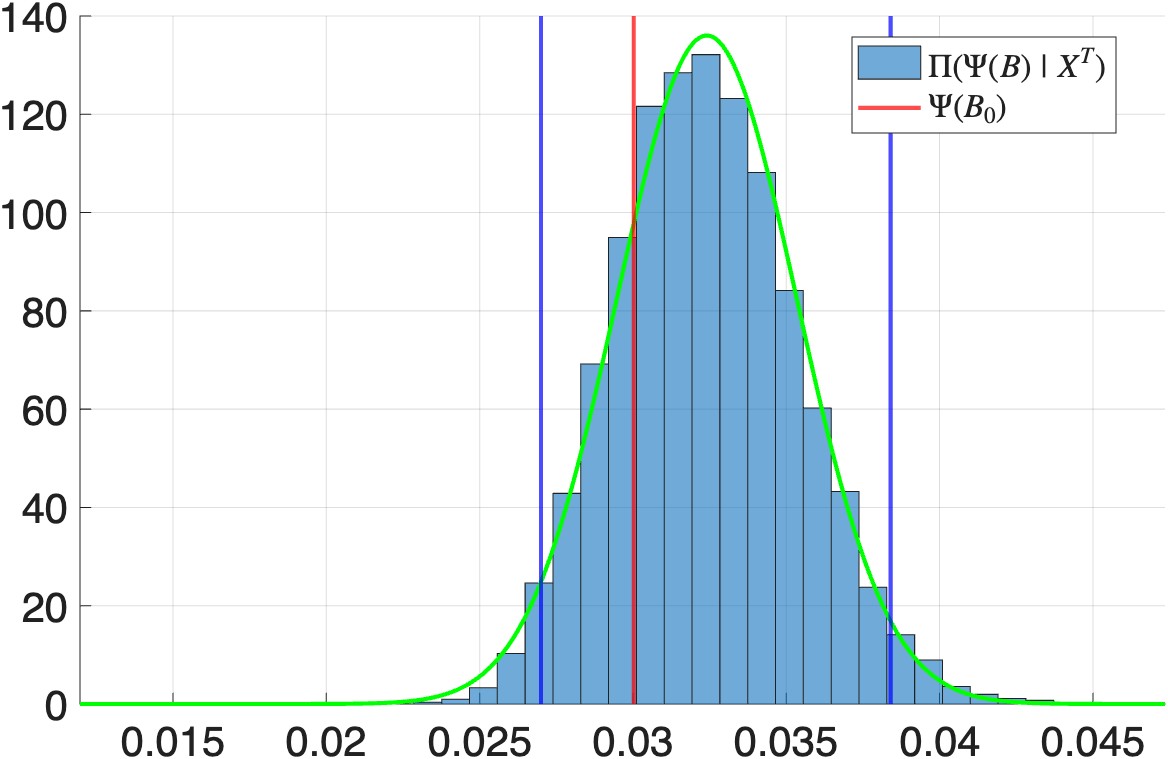}
\caption{Top row: plug-in posterior distributions of $\Psi_i(B)|X^T$, $i=1,2,3$, at $T=50$, based on the Besov-Laplace prior. Bottom row: plug-in posteriors at $T=100$. The vertical red lines indicate the `ground truths' $\Psi_i(B^{(1)})$, for $B^{(1)}$ as in Figure \ref{Fig:Truths} (top-left). The vertical blue lines identify the $95\%$ credible intervals. The green line corresponds to a normal PDF centred at the posterior mean and with variance equal to the posterior variance.}
\label{Fig:LaplPlugInPost}
\end{figure}

\paragraph{MCMC diagnostics.} Across all experiments, each wpCN run was initialized at the `cold start' $\xi_{(0)} = 0$ and iterated for 20,000 steps, with the first 5,000 samples discarded as burn-in. The step-size parameter $b$ was initialized at $0.05$ and adaptively updated every 50 iterations during burn-in using a Robbins–Monro-type update on the logit scale (e.g.~\cite{andrieu2008tutorial}), constrained in the interval $[0.001,0.1]$, so as to target an acceptance rate of approximately 30\%. The final value of $b$ was then kept fixed throughout the sampling phase, see Figure \ref{Fig:wpCNDiagnostics}, left. The resulting chains display stable trace-plots after burn-in and produce consistent posterior summaries across independent runs. Figure \ref{Fig:wpCNDiagnostics}, right, shows the evolution of the log-likelihood along representative wpCN runs for the three different ground truths $B^{(1)},B^{(2)},B^{(3)}$. The log-likelihoods are seen to rapidly move away from the
initialization point and to stabilize, after burn-in, around the log-likelihoods of the ground truths. These diagnostics indicate satisfactory mixing behaviour of the wpCN algorithm and support the reliability of the posterior summaries reported above.

\begin{figure}[t]
\centering
\includegraphics[width=7cm]{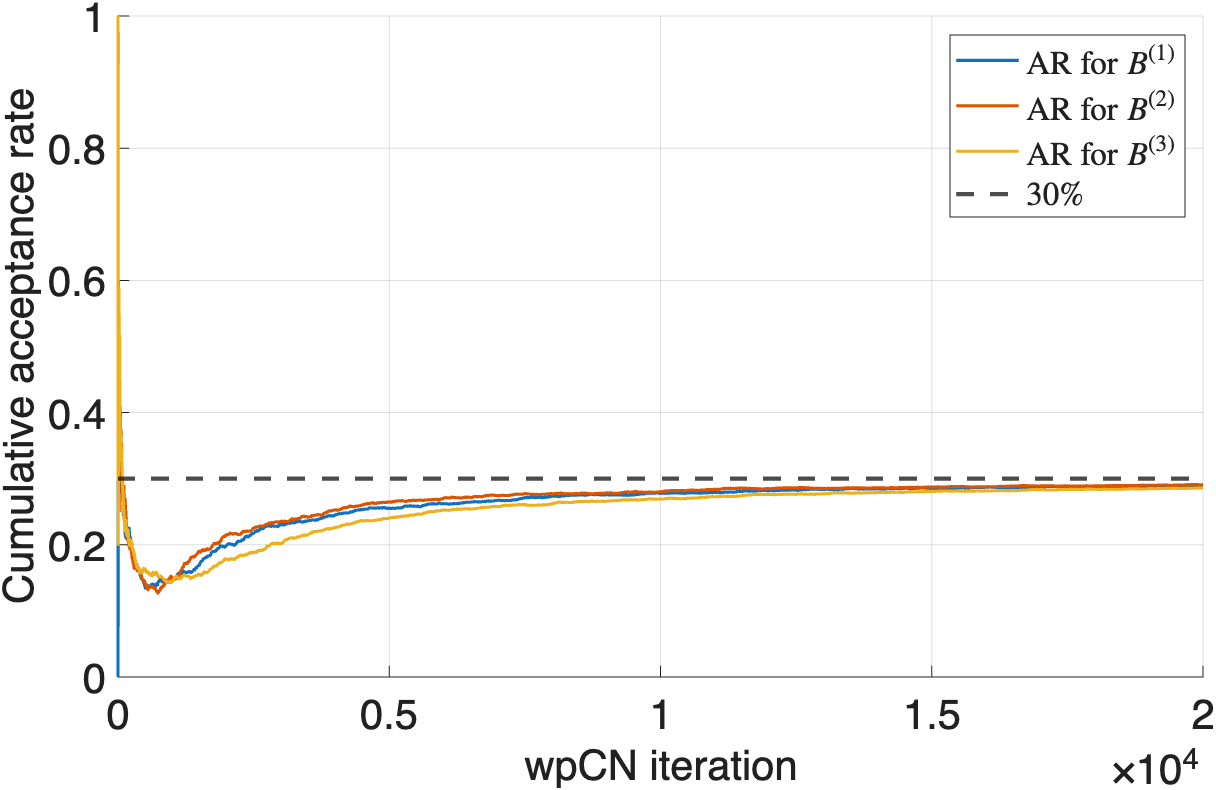}
\includegraphics[width=7cm]{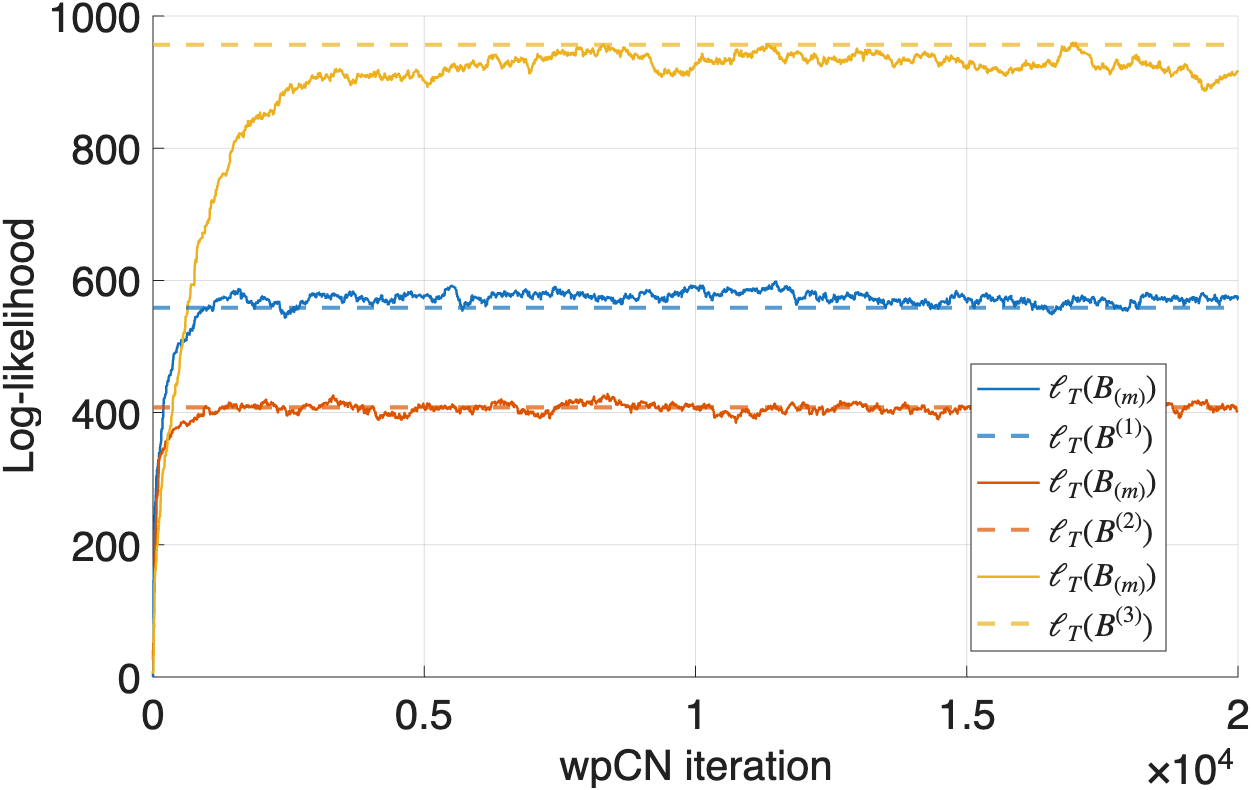}
\caption{
Left: cumulative acceptance rate (AR) along representative wpCN runs for the ground truths $B^{(1)},B^{(2)},B^{(3)}$.
Right: evolution of the log-likelihood along the same wpCN runs. The dashed lines indicate the log-likelihood of the ground truths.
}
\label{Fig:wpCNDiagnostics}
\end{figure}

%
%
%
%
%

\section{Discussion}
\label{Sec:Discussion}

We investigated the semiparametric BvM phenomenon in a  reversible multidimensional diffusion model with continuous-time observations. Our general result, Theorem \ref{thm:BvM}, establishes abstract conditions under which semiparametric Bayesian inference for a large class of possibly nonlinear functionals is asymptotically efficient. The general theory was applied to two important classes of prior distributions, namely Gaussian and Besov-Laplace priors. The practical performance of the resulting procedures was illustrated via simulation studies, with the empirical findings largely consistent with the theoretical predictions.

We conclude by briefly discussing several open research directions, including extensions to smoothness-adaptive procedures, non-periodic diffusion models, and other priors. The first two issues are discussed in the context of posterior contraction rates in \cite{giordano2022nonparametric} -- 
see Section~3 of that paper for a more in-depth treatment.

Our results are non-adaptive since the rescaled Gaussian
and Besov-Laplace priors require correct calibration based on the (typically unknown) regularity $s$ of the truth. A natural Bayesian approach to adaptation is to employ a hierarchical model by endowing the smoothness hyper-parameter with a hyper-prior. However, achieving adaptive contraction rates is already a challenging problem in the present setting due the difficulty stemming from the nonlinearity of the map $B \mapsto \mu_B$, see Section 3.1 of \cite{giordano2020bernstein}. While we expect adaptive contraction rates are possible, the situation is less clear for the semiparametric BvM. In density estimation, when the regularity of the functional and the density are very different, adaptive hierarchical Bayes procedures can already induce non-negligble biases for the centering of the marginal posterior, see \cite{rivoirard2012,castillo2015} for some examples. We expect similar subtleties to occur in our setting, further complicated by the nonlinearity of the problem.

As in \cite{giordano2022nonparametric}, we have restricted to diffusions with periodic coefficients. This assumption guarantees recurrence and mixing of the diffusion while simplifying the elliptic PDE arguments underlying both the posterior contraction and BvM analyses. In particular, periodicity allows one to work with Poisson equations on a compact domain, where standard regularity theory for elliptic PDEs  is readily available. The two most natural alternative ergodic models would entail working on bounded domains endowed with appropriate (reflecting) boundary conditions, or considering diffusions on $\R^d$ satisfying suitable drift conditions that prevent the particle from escaping to infinity. For the first scenario, extensions of our results are plausible in view of the existing regularity theory for elliptic PDEs with Neumann-type boundary conditions. However, sophisticated probabilistic tools are then needed to deal with the additional boundary local-time terms that would arise in the LAN expansion in the presence of reflection at the boundary, see \cite{nickl2024,hoffmann2025}. For diffusions on $\R^d$, the main challenge is instead the loss of compactness of the solution operator of the underlying Poisson equation, which prevents the application of the PDE techniques employed throughout this paper, see Section~3.2 of \cite{giordano2022nonparametric} for a more detailed discussion of these challenges.

Lastly, while we have focused on Gaussian and Besov--Laplace priors, the general framework developed here is not intrinsically tied to these classes. Both the abstract Bernstein--von Mises result in Theorem~\ref{thm:BvM} and the general posterior contraction rate theorem of \cite{giordano2022nonparametric} are formulated under conditions that closely resemble the standard assumptions employed in the asymptotic analysis of nonparametric Bayesian procedures. The main additional requirement in the present setting is the need to uniformly control higher-order norms of the induced invariant measures $\mu_B$, and any prior satisfying this would fit our framework, for example random series priors with bounded coefficients as in \cite{nickl2020bernstein}. Going beyond such priors remains a general challenge for Bayesian nonlinear inverse problems.

\section{Proofs}
\label{Sec:Proofs}

\subsection{Proof of Theorem \ref{thm:BvM}: general semiparametric BvM}

Following the approach of \cite{rivoirard2012,castillo2015}, the proof proceeds by showing that the Laplace transform of the rescaled marginal posterior for $\Psi(B)$ converges to that of the limiting Gaussian distribution, which then implies weak convergence to this Gaussian limit. We first localize to sets on which the full posterior concentrates.

Let $\Dcal_T\subset \dot{C}^2(\T^d)$ be measurable sets such that $\Pi(\Dcal_T|X^T)\overset{P_0}{\longrightarrow} 1$ as $T\to\infty$, and let
\begin{equation}\label{Eq:prior_conditioned}
    \Pi^{\Dcal_T}(\cdot)=\frac{\Pi(\cdot\cap \Dcal_T)}{\Pi(\Dcal_T)}
\end{equation}
denote the prior $\Pi$ conditioned to $\Dcal_T$. The corresponding posteriors satisfy $\|\Pi(\cdot|X^T) - \Pi^{\Dcal_T}(\cdot|X^T)\|_{\textnormal{TV}} \leq 2\Pi(\Dcal_T^c|X^T)$ (e.g.~p142 of \cite{vdvaart1998}), which tends to zero in $P_{0}$-probability as $T\to\infty$. Since convergence in total variation is stronger than weak convergence, it suffices to show the desired result for the posterior based on the conditioned prior $\Pi^{\Dcal_T}$ instead of $\Pi$. 
The next lemma expands the posterior Laplace transform of $\sqrt{T}(\Psi(B)-\widetilde{\Psi}_T)$ for the conditioned prior and a suitable centering $\widetilde{\Psi}_T$:
\begin{equation*}
      E^{\Pi^{\Dcal_T}}
    \Big[e^{u\sqrt T (\Psi(B) - \widetilde\Psi_T) }\Big| X^T\Big] = \frac{\int_{\Dcal_T} e^{u\sqrt T (\Psi(B) - \widetilde\Psi_T)}  e^{\ell_T(B)}d\Pi(B)}{\int_{\Dcal_T} e^{\ell_T(B)} d\Pi(B)}.
\end{equation*}

\begin{lemma}\label{lem:LaplTransAsymp}
Suppose $B_0 \in \dot{C}^{(d/2+1+\eta)\vee 2}(\T^d)$ for some $\eta>0$, and let $\Psi:\dot{C}^2(\T^d) \to \R$ be a functional satisfying the expansion \eqref{Eq:psi_exp} with representor $\psi \in \dot{L}^2(\T^d)$. For $1\le p,q\le\infty$ such that $1/p+1/q=1$, $M>0$ and $\eps_T,\zeta_T,\xi_T \to 0$, let
\begin{align*}
        \Dcal_T \subseteq  & \big\{ B\in \dot C^2(\T^d):\|\nabla B - \nabla B_0\|_{p}\le M \eps_T, ~ \|B-B_0\|_{H^{d/2+1+\eta}}\leq \zeta_T,\\
        & \qquad \|B\|_{H^{d+1+\eta}} \leq M, ~ |r(B,B_0)| \leq \xi_T/\sqrt{T} \big\}.
\end{align*}
Let $\gamma_T\in \dot{H}^{d/2+\eta}(\T^d) \cap \dot{C}^1(\T^d)$ be a sequence of fixed functions such that as $T\to\infty$,
$$
   \|\gamma_T\|_{C^1}\vee \|\gamma_T\|_{H^{(d/2+\eta)\vee 1}}=O(1); \qquad \|A^{-1}_{\mu_0}\psi - \gamma_T\|_{W^{1,q}} = o(1/(\sqrt{T}\eps_T)),
$$
where $A_{\mu_0}$ is the second-order operator \eqref{Eq:A_mu}. For fixed $u\in\R$, define the perturbations $B_u = B - u \gamma_T / \sqrt{T}$.
Then the (localized) posterior Laplace transform satisfies
\begin{align*}\label{Eq:Laplace_trans}
        E^{\Pi^{\Dcal_T}}
    \Big[e^{u\sqrt T (\Psi(B) - \widetilde\Psi_T) }\Big| X^T\Big] = e^{ \frac{u^2}{2T}\int_0^T \|\nabla\gamma_T(X_t)\|^2dt} \frac{\int_{\Dcal_T} e^{\ell_T(B_u)}d\Pi(B)}{\int_{\Dcal_T} e^{\ell_T(B)} d\Pi(B)} (1+o_{P_0}(1))
\end{align*}
as $T \to\infty$, with centering
\begin{equation}\label{Eq:psi_tilde}
     \widetilde{\Psi}_T =
    \Psi(B_0) + \frac{1}{T}\int_0^T\nabla\gamma_T(X_t).dW_t.
\end{equation}
\end{lemma}

\begin{proof}
Using Bayes formula and setting $Z_T = \int_{\Dcal_T} e^{\ell_T(B)} d\Pi(B)$ to be the normalizing constant,
\begin{align*}
    I_T(u) & := E^{\Pi^{\Dcal_T}} [e^{u\sqrt{T}(\Psi(B) - \widetilde{\Psi}_T)}|X^T] \\
    &= e^{-\frac{u}{\sqrt T}\int_0^T\nabla\gamma_T(X_t).dW_t}
    \frac{1}{Z_T} \int_{\Dcal_T}e^{u\sqrt T (\langle B - B_0, \psi\rangle_2 + r(B,B_0)) } e^{\ell_T(B)-\ell_T(B_u)}
    e^{\ell_T(B_u)}d\Pi(B),
\end{align*}
where we have used the functional expansion \eqref{Eq:psi_exp}. Using the LAN expansion given in Lemma \ref{lem:LAN},
\begin{align*}
    \ell_T(B) - \ell_T(B_u) &= \frac{u}{\sqrt T}\int_0^T\nabla\gamma_T(X_t).dW_t 
    + \frac{u^2}{2T}\int_0^T \|\nabla\gamma_T(X_t)\|^2dt \\
    & \quad -\frac{u}{\sqrt T}\int_0^T\nabla(B - B_0)(X_t).\nabla\gamma_T(X_t)dt \\
    &= \frac{u}{\sqrt T}\int_0^T\nabla\gamma_T(X_t).dW_t 
    + \frac{u^2}{2T}\int_0^T \|\nabla\gamma_T(X_t)\|^2dt \\
    & \quad - u\G_T[\nabla(B-B_0).\nabla \gamma_T] - u \sqrt{T} \langle \nabla (B-B_0),\nabla \gamma_T \rangle_{\mu_0}.
\end{align*}
Using this and $\sup_{B \in \Dcal_T} |r(B,B_0)| =o(1/\sqrt{T})$, the second to last display equals
\begin{align*}
    &e^{ \frac{u^2}{2T}\int_0^T \|\nabla\gamma_T(X_t)\|^2dt +o_{P_0}(1)} \\
    & \quad \times \frac{1}{Z_T}
    \int_{\Dcal_T}e^{-u\G_T[\nabla(B-B_0).\nabla \gamma_T]} e^{u\sqrt T [\langle B - B_0, \psi\rangle_2 - \langle \nabla(B-B_0),\nabla \gamma_T \rangle_{\mu_0}]} e^{\ell_T(B_u)}d\Pi(B).
\end{align*}
By Lemma \ref{Lem:EmpProc}, $\sup_{B\in \Dcal_T} |\G_T[\nabla(B-B_0).\nabla \gamma_T]| = o_{P_0}(1)$. Integrating by parts,
\begin{align*}
    \langle\nabla(B-B_0),\nabla\gamma_T\rangle_{\mu_0}
    &= \int_{\T^d}\nabla(B - B_0)(x).
    \nabla\gamma_T(x) \mu_0(x) dx\\
    &=\int_{\T^d}(B - B_0)(x)\nabla\cdot(\mu_0
    \nabla\gamma_T)(x) dx\\
    &=\langle B - B_0, A_{\mu_0}\gamma_T\rangle_2
\end{align*} 
for $A_{\mu_0}$ the second order elliptic operator defined in \eqref{Eq:A_mu}. Therefore, writing $\psi = A_{\mu_0} A_{\mu_0}^{-1}\psi$ (cf.~Lemma \ref{lem:PDE}), for $B \in \Dcal_T$,
\begin{align*}
    \sqrt{T}|\langle B - B_0, \psi\rangle_2 - \langle \nabla(B-B_0),\nabla \gamma_T \rangle_{\mu_0}| &= \sqrt{T} |\langle B-B_0,\psi - A_{\mu_0} \gamma_T \rangle_2| \\
    &=\sqrt T |\langle \nabla(B - B_0), \mu_0\nabla(A^{-1}_{\mu_0}\psi -  \gamma_T) \rangle_2|\\
    &\lesssim 
    \sqrt T \|\mu_0\|_\infty \|\nabla( B - B_0)\|_p 
    \|\nabla(A^{-1}_{\mu_0}\psi - \gamma_T)\|_q\\
    &\lesssim 
    \sqrt T \eps_T \|A^{-1}_{\mu_0}\psi - \gamma_T\|_{W^{1,q}}
    =o(1)
\end{align*}
by assumption and since $\mu_0$ is bounded on $\T^d$. The Laplace transform in question then equals
\begin{align*}
    I_T(u) &= e^{ \frac{u^2}{2T}\int_0^T \|\nabla\gamma_T(X_t)\|^2dt + o_{P_0}(1)}  \frac{1}{Z_T}
    \int_{\Dcal_T} e^{\ell_T(B_u)}d\Pi(B)
\end{align*}
as desired.
\end{proof}

\begin{proof}[Proof of Theorem \ref{thm:BvM}]
    Since $\Dcal_T$ satisfies $\Pi(\Dcal_T|X^T) \overset{P_0}{\longrightarrow} 1$ by assumption, it suffices to prove the result for the prior \eqref{Eq:prior_conditioned} conditioned to $\Dcal_T$ by the argument following that equation. We shall show that the resulting posterior Laplace transform $J_T(u) = E^{\Pi^{\Dcal_T}} [e^{u\sqrt{T}(\Psi(B) - \hat{\Psi}_T)}|X^T]$ converges in probability to $\exp(u^2\|\nabla A_{\mu_0}^{-1}\psi\|_{\mu_0}^2/2)$, which is the Laplace transform of a $N(0,\|\nabla A_{\mu_0}^{-1} \psi\|_{\mu_0}^2)$ distribution, for every $u\in \R$ in a neighbourhood of 0. Since convergence of such conditional Laplace transforms in $P_0$-probability implies conditional convergence in distribution in $P_0$-probability (e.g.~Lemma 1 of the supplement of \cite{castillo2015} or Corollary 2 of \cite{Ray2021}), this will complete the proof.

    By Lemma \ref{Lem:LimCov}, we have $\hat{\Psi}_T = \widetilde{\Psi}_T + o_{P_0}(1/\sqrt{T})$, so we may replace the centering in $J_T(u)$ by $\widetilde{\Psi}_T$ defined in \eqref{Eq:psi_tilde} at the cost of a multiplicative $e^{o_{P_0}(1)}$ term. The resulting Laplace transform is then exactly the one considered in Lemma \ref{lem:LaplTransAsymp}. Since the sets $\Dcal_T$ and functions $(\gamma_T)$ in the present theorem  satisfy the conditions of Lemma \ref{lem:LaplTransAsymp}, that lemma implies
    $$
        J_T(u)  = e^{ \frac{u^2}{2T}\int_0^T \|\nabla\gamma_T(X_t)\|^2dt + o_{P_0(1)}} \frac{\int_{\Dcal_T} e^{\ell_T(B_u)}d\Pi(B)}{\int_{\Dcal_T} e^{\ell_T(B)} d\Pi(B)}.
    $$
Applying Lemma \ref{Lem:LimCov} to the first term (since $\|\nabla A_{\mu_0}^{-1}\psi - \nabla\gamma_T\|_{\mu_0}\lesssim \| A_{\mu_0}^{-1}\psi - \gamma_T\|_{W^{1,q}}=o(1)$ for $q\ge2$) and using assumption \eqref{Eq:change_of_measure} for the second, we conclude $J_T(u) = e^{u^2\|\nabla A_{\mu_0}^{-1}\psi\|_2^2/2 + o_{P_0}(1)}$.
The theorem then follows from the convergence of Laplace transforms.
\end{proof}

\subsection{Auxiliary results}

In this section, we present technical results that are used in the main proofs. We have the following local asymptotic normality (LAN) expansion.

\begin{lemma}\label{lem:LAN}
    If $B_0 \in C^{(d/2+1+\kappa)\vee 2}(\T^d)$ and $h \in H^{d/2+1+\kappa}$ for some $\kappa>0$, then
    $$\ell_T(B_0 + h/\sqrt{T}) - \ell_T(B_0) = W_T(h) - \frac{1}{2T} \int_0^T \|\nabla h(X_t)\|^2 dt,$$
    where, under $P_{B_0}$ and as $T\to\infty$,
    \begin{align*}
        & W_T(h) = \frac{1}{\sqrt{T}}\int_0^T \nabla h(X_t).dW_t \to^d N(0,\|\nabla h\|_{\mu_0}^2), \\
        & \frac{1}{2T} \int_0^T \|\nabla h(X_t)\|^2 dt \to^{P_{B_0}} \frac{1}{2} \|\nabla h \|_{\mu_0}^2.
    \end{align*}
\end{lemma}

\begin{proof}
This is a specific instance of Lemma 6 of \cite{nickl2020nonparametric}.
\end{proof}

We next require some additional definitions. Let $L_{B_0}:H^2(\T^d) \to L^2(\T^d)$ be the generator of the diffusion \eqref{Eq:SDE} given by
\begin{align}\label{Eq:generator}
Lu = L_{B_0} u = \frac{1}{2}\Delta u + \nabla B_0.\nabla u.
\end{align}
Further equip $L_{\mu_0}^2 (\T^d) = \left\{ f \in L^2(\T^d): \int_{\T^d} f(x) d\mu_0(x) = 0\right\}$ with the corresponding pseudo-distance
$$d_L^2(f,g) = \sum_{i=1}^d \left\| \partial_{x_i} L_{B_0}^{-1}[f-g] \right\|_\infty^2,$$
which is well-defined since the solution map $L_{B_0}^{-1}:L_{\mu_0}^2(\T^d) \to \dot{H}^2(\T^d)$ acts on $L_{\mu_0}^2(\T^d)$ as soon as $B \in C^2$, see Section 6 of \cite{nickl2020nonparametric} for details. Let $N(\Fcal,d,\tau)$ denote the covering number of $\Fcal$, i.e.~the minimal number of $d$-balls of radius $\tau$ needed to cover a set $\Fcal$. Further define
\begin{align}\label{Eq:entropy_int}
J_\Fcal = \int_0^{D_\Fcal} \sqrt{2\log 2N(\Fcal,6d_L,\tau)} d\tau,
\end{align}
where $D_\Fcal$ is the $d_L$-diameter of $\Fcal$. Using these quantities, we obtain the following lemma controlling the remainder term in the LAN expansion, uniformly over a function class. The constant $\eta>0$ below can be arbitrarily small and does not affect the required regularity in a significant way.

\begin{lemma}\label{Lem:EmpProc}
For some $\eta>0$, suppose $B_0 \in C^{(d/2+\eta)\vee 2}(\T^d)$ and for $M>0$ and $\zeta_T \to 0$, let
$$\Dcal_T \subseteq \{ B \in \dot{C}^2(\T^d): \|B\|_{H^{d+1+\eta}} \leq M, ~ \|B-B_0\|_{H^{d/2+1+\eta}}\leq \zeta_T \}.$$
Let $\gamma_T\in \dot{H}^{d/2+\eta}(\T^d) \cap \dot{C}^1(\T^d)$ be (a sequence of) fixed functions and suppose
$$\Gamma := \limsup_{T \to \infty} \max(\|\gamma_T\|_{C^1},\|\gamma_T\|_{H^{(d/2+\eta)\vee 1}} )<\infty.$$
Then as $T\to\infty$,
$$E_0 \sup_{B\in \Dcal_T} |\G_T[\nabla (B-B_0).\nabla\gamma_T]|  \to 0.$$
\end{lemma}

\begin{proof}
Applying Lemma 1 of \cite{nickl2020nonparametric} with the $\mu_0$-centered function class
$$\Fcal_T=\left\{ g_{B}(x) = \nabla(B-B_0)(x).\nabla\gamma_T(x) - \langle \nabla(B-B_0),\nabla\gamma_T\rangle_{\mu_0}:  B\in\Dcal_T \right\}\cup\{0\},$$
the expected supremum under consideration is upper bounded as
\begin{align}\label{Eq:entropy+remainder}
    E_0 \sup_{B \in \Dcal_T} |\G_T[g_B]| \leq \frac{2}{\sqrt T}\sup_{g_{B}\in\Fcal_T}
    \|L_{B_0}^{-1} g_{B}\|_\infty
    +4\sqrt 2 J_{\Fcal_T}
\end{align}
with $J_{\Fcal_T}$ defined in \eqref{Eq:entropy_int}. Turning to the entropy integral $J_{\Fcal_T}$, for all $B,\bar B\in\Dcal_T$, using the Sobolev embedding $H^{d/2+\kappa} \hookrightarrow L^\infty$ for any $\kappa>0$  and the PDE estimate in Lemma A.1 of \cite{giordano2022nonparametric},
\begin{align*}
    d_{L}(g_B,g_{\bar B}) &\lesssim \|L_{B_0}^{-1}[g_B - g_{\bar B}] \|_{H^{d/2+1+\kappa}} \lesssim \|g_B - g_{\bar B}\|_{H^{d/2-1+\kappa}}
\end{align*}
for $0<\kappa< \min(\eta,1/2)$. Therefore, using the Runst-Sickel lemma (\cite{runst1996}, p.~345 or Lemma 2 of \cite{nickl2020nonparametric}),
\begin{align*}
    d_{L}(g_B,g_{\bar B}) & \lesssim \| \nabla (B-\bar{B}). \nabla \gamma_T - \langle \nabla (B-\bar{B}), \nabla \gamma_T \rangle_{\mu_0} \|_{H^{(d/2-1+\kappa)_+}} \\
    & \lesssim \|\nabla(B-\bar{B})\|_{H^{(d/2-1+\kappa)_+}}\|\nabla\gamma_T\|_{\infty} + \|\nabla(B-\bar{B})\|_\infty \|\nabla \gamma_T\|_{H^{(d/2-1+\kappa)_+}} \\
    & \qquad + |\langle \nabla (B-\bar{B}), \nabla \gamma_T \rangle_{\mu_0}| \|1\|_{H^{(d/2-1+\kappa)_+}} \\
    & \lesssim \|B-\bar{B}\|_{H^{d/2+1+\kappa}} \max(\|\gamma_T\|_{C^1},\|\gamma_T\|_{H^{(d/2+\kappa)\vee 1}} ) \\
    & \lesssim \Gamma \|B-\bar{B}\|_{H^{d/2+1+\kappa}}
\end{align*}
for $T>0$ large enough. Since $0<\kappa <\eta$, the $d_L$-diameter of $\Fcal_T$ thus satisfies $D_{\Fcal_T} \lesssim \Gamma \zeta_T \to 0$ by assumption. Writing $H^r_1$ for the unit ball of $H^r(\T^d)$, the metric entropy is then bounded as, for $T>0$ large enough,
\begin{align*}
    \log N(\Fcal_T,6d_L,\tau)
    & \leq \log N(\Dcal_T,C\Gamma \|\cdot\|_{H^{d/2+1+\kappa}},\tau) \\
    & \leq \log N( M H_1^{d+1+\eta}, \|\cdot\|_{C^{d/2+1+\kappa}},\tau/(C'\Gamma ))\\
    &\lesssim \left( \frac{C'M\Gamma}{\tau} \right)^{\frac{d}{(d+1+\eta)-(d/2+1+\kappa)}} ,
\end{align*}
where the last inequality follows by arguing as in the proof of Theorem 4.3.36 in \cite{ginenickl2016} as soon as $(d+1+\eta)-(d/2+1+\kappa)>d/2$, i.e.~$0<\kappa<\eta$. This yields
\begin{align*}
J_{\Fcal_T} &\lesssim \int_0^{D_{\Fcal_T}}  \sqrt{\log 2} + (M\Gamma /\tau)^{\frac{d/2}{d/2+\eta-\kappa}} d\tau \lesssim D_{\Fcal_T} +  D_{\Fcal_T}^{\frac{\eta-\kappa}{d/2+\eta-\kappa}} \to 0.
\end{align*}
The second term in \eqref{Eq:entropy+remainder} is thus $o(1)$ as $T\to\infty$. For the first term, arguing as above, for all $B \in \Dcal_T$ and $\kappa>0$ small enough,
\begin{align*}
\| L_{B_0}^{-1}g_{B}\|_\infty & \lesssim \|g_B\|_{H^{(d/2-2+\kappa)_+}} \\
& \lesssim \|\nabla(B-B_0)\|_{H^{(d/2-2+\kappa)_+}} \|\nabla \gamma_T\|_\infty + \|\nabla(B-B_0)\|_\infty \|\nabla \gamma_T\|_{H^{(d/2-2+\kappa)_+}} \\
& \lesssim \Gamma \|B-B_0\|_{H^{d/2+1+\kappa}} \lesssim \Gamma \zeta_T \to 0,
\end{align*}
where we have used the Sobolev embedding theorem, Lemma A.1 of \cite{giordano2022nonparametric} and Lemma 2 of \cite{nickl2020nonparametric}. This shows that the right side of \eqref{Eq:entropy+remainder} tends to zero as $T\to\infty$, completing the proof.
\end{proof}

We require the following $L^2(P_0)$-bound for averages of square functions of the diffusion process.

\begin{lemma}\label{Lem:square_functionals}
Suppose $B_0 \in C^{(d/2+1+\kappa)\vee 2}(\T^d)$ and $h\in H^{d/2+\kappa}(\T^d)$ for some $\kappa>0$. Then
\begin{align*}
E_0 \left( \frac{1}{T} \int_0^T h(X_s)^2 ds - \int_{\T^d} h(x)^2 d\mu_0(x) \right)^2 \leq C \left( \frac{1}{T} \|h\|_{H^{d/2+\kappa}}^2 + \frac{1}{T^2} \|h\|_{H^{d/2+\kappa}}^4\right) ,
\end{align*}
where $C$ depends only on $d,\kappa$ and $\|B_0\|_{C^{d/2+1+\kappa}}$.
\end{lemma}

\begin{proof}
Since $t \mapsto t^2$ is a smooth map, the function $f_h(x) = h(x)^2 - \|h\|_{\mu_0}^2 \in L_{\mu_0}^2(\T^d) \cap H^{d/2+\kappa}(\T^d) \subset C(\T^d)$, and moreover,
$$\|f_h\|_{H^{d/2+\kappa}} \leq C \|h\|_{H^{d/2+\kappa}}^2 + \|h\|_{\mu_0}^2\|1\|_{H^{d/2+\kappa}} \leq C' \|h\|_{H^{d/2+\kappa}}^2 $$
for a constant $C'=C'(d,\kappa,\|\mu_0\|_\infty)$ since $H^{d/2+\kappa}$ is an algebra for $\kappa>0$. Let $L^{-1} = L_{B_0}^{-1}$ be the inverse of the generator $L$ defined in \eqref{Eq:generator}, see Section 6 of \cite{nickl2020nonparametric} for its construction. Lemma 11 of \cite{nickl2020nonparametric} implies that $f_h = LL^{-1}[f_h]$ everywhere and $\|L^{-1}[f_h]\|_{H^{d/2+2+\kappa}} \leq C(d,\|B_0\|_{C^{d/2+1+\kappa}}) \|f_h\|_{H^{d/2+\kappa}}$. By the Sobolev embedding theorem, $L^{-1}[f_h] \in C^2$ and so we may apply It\^o's formula to obtain
\begin{align*}
\int_0^T f_h(X_s) ds &= \int_0^T LL^{-1}[f_h](X_s) ds \\
&= L^{-1}[f_h](X_T)-L^{-1}[f_h](X_0) - \int_0^T \nabla L^{-1}[f_h](X_s).dW_s.
\end{align*}
For the first term, we use the bound $\|L^{-1}[f_h]\|_{C^2} \lesssim \|L^{-1}[f_h]\|_{H^{d/2+2+\kappa}} \leq C\|h\|_{H^{d/2+\kappa}}^2$. Using It\^o's isometry,
\begin{align*}
E \left( \int_0^T \nabla L^{-1}[f_h](X_s).dW_s \right)^2 
&= E \int_0^T \|\nabla L^{-1}[f_h](X_s)\|^2 ds\\
&\lesssim T\|L^{-1}[f_h]\|_{C^1} \lesssim T\|h\|_{H^{d/2+\kappa}}^2
\end{align*}
with the same dependence on the constants. Normalizing everything by $1/T$ then gives the result.
\end{proof}

From this lemma, we deduce the limiting covariance in the expansion of the Laplace transform in Lemma \ref{lem:LaplTransAsymp}.

\begin{lemma}\label{Lem:LimCov}
Suppose $B_0 \in C^{(d/2+1+\kappa)\vee 2}(\T^d)$ for some $\kappa>0$ and let $\psi \in \dot{H}^t(\T^d)$ with $t>d/2-1$. Further let $(\gamma_T:T>0) \subset H^{d/2+1+\kappa}(\T^d)$ be a sequence of functions such that $K_\gamma:= \limsup_{T\to\infty} \|\gamma_T\|_{H^{d/2+1+\kappa}} <\infty$ and $\|\nabla A_{\mu_0}^{-1}\psi - \nabla\gamma_T\|_{\mu_0} \to 0$ as $T \to \infty$. Then for $\hat{\Psi}_T$ and $\widetilde{\Psi}_T$ defined in \eqref{Eq:efficient_est} and \eqref{Eq:psi_tilde}, respectively, we have
$\hat{\Psi}_T - \widetilde{\Psi}_T = o_{P_0}(1/\sqrt{T}).$
Furthermore, as $T\to\infty$,
$$\frac{1}{T}\int_0^T \|\nabla\gamma_T(X_t)\|^2 dt \overset{P_0}{\longrightarrow} \|\nabla A_{\mu_0}^{-1} \psi\|_{\mu_0}^2.$$

\end{lemma}

\begin{proof}
Using the definitions \eqref{Eq:efficient_est} and \eqref{Eq:psi_tilde},
\begin{align}\label{Eq:psi_difference}
\hat{\Psi}_T - \widetilde{\Psi}_T = \frac{1}{T} \int_0^T \left[\nabla A_{\mu_0}^{-1}\psi (X_t) - \nabla\gamma_T(X_t) \right].dW_t + o_{P_0}(T^{-1/2}).
\end{align}
We then have
\begin{align*}
E_0 &\left( \frac{1}{T} \int_0^T \left[\nabla A_{\mu_0}^{-1}\psi (X_t) - \nabla\gamma_T(X_t) \right].dW_t \right)^2\\
&= \frac{1}{T^2} E_0 \int_0^T \sum_{i=1}^d [\partial_{x_i} A_{\mu_0}^{-1}\psi(X_s) - \partial_{x_i} \gamma_T(X_s)]^2 ds.
\end{align*}
We may without loss of generality take $\kappa>0$ small enough that $0<\kappa<t-d/2+1$. For such $\kappa$, Lemma \ref{lem:PDE} implies $\|\nabla A_{\mu_0}^{-1}\psi \|_{H^{d/2+\kappa}} \lesssim \|A_{\mu_0}^{-1} \psi\|_{H^{d/2+1+\kappa}} \leq C\|\psi\|_{H^{d/2-1+\kappa}} \leq C \|\psi\|_{H^t}$, where $C$ depends only on $d,\kappa$ and $\|B_0\|_{C^{|d/2-1+\kappa|+1}}$. We thus have that $\|\nabla A_{\mu_0}^{-1}\psi - \nabla\gamma_T\|_{H^{d/2+\kappa}} \lesssim \|\psi\|_{H^t} + K_\gamma<\infty$ for $T>0$ large enough. Applying Lemma \ref{Lem:square_functionals} with each $h_i(x) = \partial_{x_i} A_{\mu_0}^{-1}\psi(x) - \partial_{x_i} \gamma_T(x)$ then gives
\begin{align}\label{Eq:L2_bound}
E_0 \left( \frac{1}{T} \int_0^T \| \nabla A_{\mu_0}^{-1}\psi (X_t) - \nabla\gamma_T(X_t) \|^2 dt - \|\nabla A_{\mu_0}^{-1}\psi - \nabla\gamma_T\|_{\mu_0}^2 \right)^2  \lesssim \frac{1}{T}
\end{align}
for $T>0$ large enough. Therefore,
\begin{align*}
E_0 \left( \frac{1}{T} \int_0^T \left[\nabla A_{\mu_0}^{-1}\psi (X_t) - \nabla\gamma_T(X_t) \right].dW_t \right)^2 &= \frac{1}{T} \|\nabla A_{\mu_0}^{-1}\psi - \nabla\gamma_T\|_{\mu_0}^2 + O(T^{-3/2}) \\
& = o(1/T)
\end{align*}
by assumption. Since convergence in $L^2(P_0)$ implies convergence in $P_0$-probability, the stochastic integral term in \eqref{Eq:psi_difference} is $o_{P_0}(T^{-1/2})$ as desired.

Turning to the second assertion, arguing as for \eqref{Eq:L2_bound} and using Lemma \ref{Lem:square_functionals},
\begin{align}\label{Eq:L2_bound2}
    E_0 \left( \frac{1}{T}\int_0^T \|\nabla A_{\mu_0}^{-1}\psi(X_t)\|^2 dt - \|\nabla A_{\mu_0}^{-1}\psi\|_{\mu_0}^2 \right)^2 \leq \frac{C}{T}
\end{align}
for $C=C(d,\kappa,\|B_0\|_{C^{d/2+\kappa+1}},\|\psi\|_{H^t})$, so that in particular $\frac{1}{T}\int_0^T \|\nabla A_{\mu_0}^{-1}\psi(X_t)\|^2 dt \overset{P_0}{\longrightarrow} \|\nabla A_{\mu_0}^{-1}\psi\|_{\mu_0}^2$. Expanding out the difference of two squares and using Cauchy-Schwarz,
\begin{align*}
& \left| \frac{1}{T}\int_0^T \|\nabla\gamma_T(X_t)\|^2dt - \frac{1}{T}\int_0^T \|\nabla A_{\mu_0}^{-1} (X_t)\|^2 dt \right| \\
& \quad = \left| \frac{1}{T} \int_0^T \sum_{i=1}^d \left( \partial_{x_i} \gamma_T(X_t) - \partial_{x_i} A_{\mu_0}^{-1}\psi (X_t) \right) \left( \partial_{x_i} \gamma_T (X_t) + \partial_{x_i} A_{\mu_0}^{-1}\psi (X_t) \right) dt \right|  \\
& \quad \leq   \left( \frac{1}{T} \int_0^T \sum_{i=1}^d (\partial_{x_i} \gamma_T(X_t) - \partial_{x_i} A_{\mu_0}^{-1}\psi (X_t))^2 dt \right)^{1/2}\\
&\quad\quad\ \times
\left( \frac{1}{T} \int_0^T \sum_{i=1}^d (\partial_{x_i} \gamma_T(X_t) + \partial_{x_i} A_{\mu_0}^{-1}\psi (X_t))^2 dt \right)^{1/2}.
\end{align*}
The first term equals $\|\nabla A_{\mu_0}^{-1}\psi - \nabla\gamma_T\|_{\mu_0}^2 + O_{P_0}(1/\sqrt{T})$ by \eqref{Eq:L2_bound}. Writing $\nabla \gamma_T = (\nabla \gamma_T - \nabla A_{\mu_0}^{-1} \psi) + \nabla A_{\mu_0}^{-1} \psi$, the square of the second term can be upper bounded by a multiple of
\begin{align*}
\frac{1}{T}\int_0^T &\|\nabla \gamma_T(X_t) - \nabla A_{\mu_0}^{-1} \psi(X_t)\|^2 + \| \nabla A_{\mu_0}^{-1} \psi(X_t)\|^2 dt \\
&= \|\nabla A_{\mu_0}^{-1}\psi - \nabla\gamma_T\|_{\mu_0}^2 
+ \|\nabla A_{\mu_0}^{-1}\psi\|_{\mu_0}^2 + O_{P_0}(1/\sqrt{T})
\end{align*}
using \eqref{Eq:L2_bound} and \eqref{Eq:L2_bound2}. Combining the above, the before last display is bounded by
\begin{align*}
    &\left( \|\nabla A_{\mu_0}^{-1}\psi - \nabla\gamma_T\|_{\mu_0}^2 + O_{P_0}(T^{-1/2}) \right)^{1/2} \\
    &\quad\times\left( \|\nabla A_{\mu_0}^{-1}\psi - \nabla\gamma_T\|_{\mu_0}^2 + \|\nabla A_{\mu_0}^{-1}\psi\|_{\mu_0}^2 + O_{P_0}(T^{-1/2}) \right)^{1/2},
\end{align*}
which is $o_{P_0}(1)$ since $\|\nabla A_{\mu_0}^{-1}\psi - \nabla\gamma_T\|_{\mu_0} \to 0$ by assumption. This shows 
$$\frac{1}{T}\int_0^T \|\nabla \gamma_T(X_t)\|^2 dt \overset{P_0}{\longrightarrow} \|\nabla A_{\mu_0}^{-1}\psi\|_{\mu_0}^2$$ 
as $T\to\infty$.
\end{proof}

%
%
%

\subsection{Proof of Theorem \ref{Theo:GaussBvM}: Gaussian priors}

We verify the conditions of Theorem \ref{thm:BvM} with $p=q=2$. First note that by Condition \ref{Condition:BaseGP}, $\Pi$ is supported on $\dot{C}^2(\T^d)$ by construction. Further, since $s>d$, we have $\H_W\subset \dot{H}^{s+1}(\T^d)\subset \dot{H}^{d+1+\kappa}(\T^d)$ for all sufficiently small $\kappa$, whence $\gamma_T\in \dot{H}^{d+1+\kappa}(\T^d)\subset \dot{H}^{d+1/2+\kappa}(\T^d)$. Since $\gamma_T$ satisfies the conditions \eqref{Eq:gamma_T_conditions} by assumption, it remains to show (i) the posterior concentrates on sets $\Dcal_T$ satisfying \eqref{Eq:D_T_conditions} and (ii) the change of measure condition \eqref{Eq:change_of_measure} in order to apply Theorem \ref{thm:BvM}.

(i) For $\eps_T = T^{-s/(2s+d)}$ and $M>0$, define the sets
\begin{equation}
\label{Eq:LocSetGP}
\begin{split}
    \Dcal_T
    =\Dcal_T(M) := \big\{ B\in \dot C^2(\T^d)
    :\|\nabla B - \nabla B_0\|_2\le M \eps_T,& \
    \|B\|_{H^{d+1+\kappa}} \leq M,\\
    &|\langle B,\gamma_T\rangle_{\H_W}| \le M 
    \|\gamma_T\|_{\H_W}\big\}.
\end{split}
\end{equation}
For $M$ large enough, Lemma \ref{Lem:LocSetGP} below implies that $\Pi(\Dcal_T|X^T)\overset{P_0}{\longrightarrow}1$ as $T\to\infty$. By the Poincaré inequality (e.g.~p.~290 in \cite{evans2010}), for all $B\in \Dcal_T \subset \dot{C}^2(\T^d)$ it holds that
$$
    \|B - B_0\|_{H^1}\simeq \|\nabla B - \nabla B_0\|_2
    \lesssim \eps_T,
$$
whence, recalling that $B_0\in \dot{H}^{s+1}(\T^d)\subset \dot{H}^{d+1+\kappa}(\T^d)$ and that $\|B\|_{H^{d+1+\kappa}} \leq M$ for all $B\in\Dcal_T$,  by the Sobolev interpolation inequality (e.g.~Theorems 1.3.3 and 4.3.1 in \cite{Triebel78}), 
$$
    \|B - B_0\|_{H^{d+1/2+\kappa}} \lesssim 
    \|B - B_0\|_{H^1}^\frac{d}{2d+2\kappa} 
    \|B - B_0\|_{H^{d+1+\kappa}}^\frac{d/2+\kappa}{d+\kappa}
    \lesssim \eps_T^\frac{d}{2d+2\kappa} =o(1).
$$
Lastly, since the remainder $r(B,B_0)$ of the functional $\Psi$ satisfies \eqref{Eq:remainder_text} by assumption, we have by Remark \ref{Rem:remainder} that for all $B\in \Dcal_T$,
$$
    \sup_{B\in \Dcal_T} |r(B,B_0)| =O(\eps^2_T)
    =o(1/\sqrt T)
$$
since $s>d$. We conclude that, for all sufficiently large $M>0$, the set $\Dcal_T$ satisfies the condition \eqref{Eq:D_T_conditions} of Theorem \ref{thm:BvM} with the choices $\zeta_T = \eps_T^{d/(2d+2\kappa)}$ and $\xi_T = \sqrt{T}\eps^2_T$.

(ii) It remains to verify the asymptotic invariance property \eqref{Eq:change_of_measure}. For $B_u = B - u\gamma_T/\sqrt T$ and $\Pi_u:=\Lcal(B_u )$, using the Cameron-Martin theorem (e.g.~Theorem 2.6.13 in \cite{ginenickl2016}), 
\begin{align*}
    \frac{\int_{\Dcal_T} e^{\ell_T(B_u )} d\Pi (B) }
    {\int_{\Dcal_T}e^{\ell_T(B)}d\Pi (B)}
    &=\frac{\int_{\Dcal_{T,u}} e^{\ell_T(B')}
    \frac{d\Pi_u}{d\Pi }(B')d\Pi (B') }
    {\int_{\Dcal_T}e^{\ell_T(B)}d\Pi (B)}\\
    &= \frac{\int_{\Dcal_{T,u}} e^{\ell_T(B')}
    e^{-\frac{u}{\sqrt T}\langle \gamma_T, B'\rangle_{\H }
    -\frac{u^2}{2T}\|\gamma_T\|^2_{\H }}d\Pi (B') }
    {\int_{\Dcal_T}e^{\ell_T(B)}d\Pi (B)},
\end{align*}
where $\Dcal_{T,u}:=\{B'=B_u: \ B\in\Dcal_T\}$. Using that $\|\cdot\|_{\H}^2 = T\eps_T^2 \|\cdot\|_{\H_W}^2$ and the definition of $\Dcal_T$,
\begin{align*}
    &\sup_{B'\in\Dcal_{T,u}}
    \Big| \frac{u}{\sqrt T}\langle \gamma_T, B'\rangle_{\H }
    +\frac{u^2}{2T}\|\gamma_T\|^2_{\H } \Big|\\
    &\ \le |u|\sqrt T\eps_T^2
    \sup_{B\in\Dcal_{T}}|\langle \gamma_T,B\rangle_{\H_W}|
    +u^2\eps_T^2
    \|\gamma_T\|^2_{\H_W}
    +\frac{u^2}{2}\eps_T^2
    \|\gamma_T\|^2_{\H_W}\\
    &\leq M|u|\sqrt T\eps_T^2\|\gamma_T\|_{\H_W}
    + \tfrac{3}{2} u^2 \eps_T^2 \|\gamma_T\|_{\H_W}^2
    =o(|u|+u^2),
\end{align*}
since by assumption 
$\|\gamma_T\|_{\H_W}=o(1/(\sqrt T\eps_T^2))$ (which also implies $\|\gamma_T\|_{\H_W}=o(1/\eps_T))$. Hence for all $|u|<1$,
\begin{align*}
    \frac{\int_{\Dcal_T} e^{\ell_T(B_u )} d\Pi (B) }
    {\int_{\Dcal_T}e^{\ell_T(B)}d\Pi (B)}
    &=e^{o(1)}
    \frac{\int_{\Dcal_{T,u}} e^{\ell_T(B')}d\Pi (B') }
    {\int_{\Dcal_T}e^{\ell_T(B)}d\Pi (B)} 
    =e^{o(1)}\frac{\Pi  (\Dcal_{T,u}|X^T)}{\Pi(\Dcal_T|X^T)}.
\end{align*}
As already observed at the beginning of the proof, the denominator in the right hand side satisfies $\Pi(\Dcal_T|X^T)\overset{P_0}{\longrightarrow} 1$. Moreover,
\begin{align*}
    \Dcal_{T,u}^c
    &= \left\{B: \|\nabla (B+\tfrac{u}{\sqrt{T}}\gamma_T) - \nabla B_0\|_2 > M \eps_T \right\} \cup \left\{ B: \|B+\tfrac{u}{\sqrt{T}}\gamma_T\|_{H^{d+1+\kappa}} > M \right\} \\
    & \qquad \qquad \cup \left\{ B: |\langle B+\tfrac{u}{\sqrt{T}}\gamma_T,\gamma_T \rangle_{\H_W}| > M \|\gamma_T\|_{\H_W} \right\}.
\end{align*}
Since $\|\gamma_T\|_2 \lesssim \|\gamma_T\|_{H^{d/2+1+\kappa}} = O(1)$ by assumption, the first set is contained in
$\{B: \|\nabla B - \nabla B_0\|_2 > M \eps_T - C/\sqrt{T} \}$.
The second set is similarly contained in $\{ B: \|B\|_{H^{d+1+\kappa}} > M-o(1) \}$ using $\|\gamma_T\|_{H^{d+1+\kappa}} = o(\sqrt{T})$, while the third is contained in $\{ B: |\langle B,\gamma_T \rangle_{\H_W}| > (M-o(1)) \|\gamma_T\|_{\H_W} \}$ since $\|\gamma_T\|_{\H_W} = o(1/(T\eps_T^2)) = o(1)$. Since $T^{-1/2} = o(\eps_T)$, we conclude that $\Dcal_{T,u}(M)^c \subset \Dcal_T(M/2)^c$ for $T>0$ large enough. For sufficiently large $M>0$, we thus have $\Pi(\Dcal_{T,u}(M)|X^T) \geq \Pi(\Dcal_T(M/2)|X^T) \overset{P_0}{\longrightarrow}1$ by Lemma \ref{Lem:LocSetGP}, which completes the proof of \eqref{Eq:change_of_measure}.
\qed

\begin{lemma}\label{Lem:LocSetGP}
Let $\Pi$ and $B_0$ be as in the statement of Theorem \ref{Theo:GaussBvM}, and let $\Dcal_T$ be the set defined in \eqref{Eq:LocSetGP}. Then for all sufficiently large $M>0$, as $T\to\infty$,
$$
    \Pi(\Dcal_T|X^T)\overset{P_0}{\longrightarrow}1.
$$
\end{lemma}

\begin{proof}
Write $\Dcal_T = \Dcal_{T,1}\cap \Dcal_{T,2}\cap \Dcal_{T,3}$ with $\Dcal_{T,1}:=\{ B\in \dot C^2(\T^d):\|\nabla B - \nabla B_0\|_2\le M \eps_T\}$, $\Dcal_{T,2}:=\{ B\in \dot C^2(\T^d):\|B\|_{H^{d+1+\kappa}} \leq M\}$, and $\Dcal_{T,3}:=\{ B\in \dot C^2(\T^d):|\langle B,\gamma_T\rangle_{\H_W}| \le M \|\gamma_T\|_{\H_W}\big\}$. We show that each set has posterior probability tending to one in $P_{0}$-probability as $T\to\infty$.

For $M>0$ large enough, this holds for $\Dcal_{T,1}$ by Theorem 2.1 of \cite{giordano2022nonparametric}, whose assumptions are satisfied under the conditions of Theorem \ref{Theo:GaussBvM}. Further, since $\Pi_W$ is supported on $\dot{H}^{d+1+\kappa}(\T^d)\cap \dot{C}^{(d/2+\kappa)\vee 2}(\T^d)$ by assumption, arguing exactly as in the proof of Lemma 5.2 of \cite{giordano2022nonparametric}, it follows that for all $K>0$ there exists sufficiently large $M>0$ such that
$$
    \Pi(\Dcal_{T,2}^c)\le e^{-K T\eps_T^2}.
$$
By an analogue of the Theorem 8.20 in \cite{ghosal2017} for the present setting, whose validity is implied by the proof of Theorem 2.1 of \cite{giordano2022nonparametric}, we then have
$\Pi(\Dcal_{T,2}|X^T)\overset{P_0}{\longrightarrow}1$ as $T\to\infty$.

Turning to $\Dcal_{T,3}$, note that if $B\sim\Pi$, then
$\langle B,\gamma_T\rangle_{\H}
=T\eps_T^2\langle B,\gamma_T\rangle_{\H_W}
\sim N(0,\|\gamma_T\|^2_{\H})=N(0,T\eps_T^2\|\gamma_T\|^2_{\H_W})$, and therefore by the standard tail inequality for normal random variables,
\begin{align*}
    \Pi\left( \Dcal^c_{T,3}\right)
    &=
    \Pi\left( B : \frac{|\langle B,\gamma_T\rangle_{\H_W}|}
    {\|\gamma_T\|_{\H_W}}>M
    \right)\\
    &=\Pi\left( B : \frac{|\langle B,\gamma_T\rangle_{\H}|}
    {\|\gamma_T\|_{\H}}>M\sqrt T \eps_T
    \right)
    \le e^{-M^2 T\eps_T^2}.
\end{align*}
This implies that 
$\Pi(\Dcal_{T,3}|X^T)\overset{P_0}{\longrightarrow}1$ as $T\to\infty$, again by an analogue of Theorem 8.20 in \cite{ghosal2017}.
\end{proof}

%
%
%

\subsection{Proof of Theorem \ref{Theo:LaplBvM}: Besov-Laplace priors}

We verify the conditions of Theorem \ref{thm:BvM} with $p=1$ and $q=\infty$. By construction, the support of $\Pi_W$, and hence also of $\Pi$, is equal to the wavelet approximation space
$$V_J=\textnormal{span}(\Phi_{lr}, \ l=0,1,\dots,J,\ r=0,\dots,(2^{ld-1}-1)\vee0)\subset \dot{C}^2(\T^d).$$
Since  $B_0\in \dot{H}^{s+1}(\T^d)$ and $s>d+(d/2)\vee 2$, it holds that $B_0\in \dot C^{(d/2+1+\kappa)\vee 2}(\T^d)$ by the Sobolev embedding. Let $B_{0,T}:=\sum_{l=0}^J\sum_r \langle B_0,\Phi_{lr}\rangle_2\Phi_{lr}$ be the projection of $B_0$ onto $V_J$, and for $\eps_T := T^{-s/(2s+d)}$ and $M>0$, define the sets
\begin{equation}
\label{Eq:LocSetLapl}
\begin{split}
    \Dcal_T
    =\Dcal_T(M) := \big\{ B\in V_J
    :\|\nabla B - \nabla B_{0,T}\|_1\le M \eps_T,& \
    \|B\|_{H^{d+1+\kappa}} \le M
   \big\},
\end{split}
\end{equation}
where $\kappa>0$ is an arbitrarily small constant. For $M$ large enough, Lemma \ref{Lem:LocSetLapl} below implies that $\Pi(\Dcal_T|X^T)\overset{P_0}{\longrightarrow}1$ as $T\to\infty$.

Using \eqref{Eq:B0_approx}, $\|\nabla B - \nabla B_0\|_1\lesssim \eps_T$ for all $B\in \Dcal_T$, which verifies the first requirement in \eqref{Eq:D_T_conditions}. Next, in view of the continuous embedding $W^{1,1}(\T^d)\subset B^1_{1\infty}(\T^d)$ (e.g.~Proposition 4.3.20 in \cite{ginenickl2016}) and the Poincaré inequality, for all $B\in \Dcal_T$,
$$
    \|B - B_{0,T}\|_{B^1_{1\infty}}\lesssim \| B - B_{0,T}\|_{W^{1,1}}\lesssim \|\nabla B - \nabla B_{0,T}\|_1
    \lesssim \eps_T,
$$
and hence, using that $B,B_{0,T}\in V_J$, for any $\kappa>0$,
$$
    \|B - B_{0,T}\|_{B^{d/2+\kappa}_{1\infty}}
    \lesssim 2^{J(d/2+\kappa-1)_+}
    \|B - B_{0,T}\|_{B^1_{1\infty}}
    \lesssim T^{-\frac{s-(d/2+\kappa-1)_+}{2s+d}}.
$$
Using the continuous embedding $B^{d/2+\kappa}_{1\infty}(\T^d)\subseteq L^2(\T^d)$ (e.g.~Proposition 4.3.9 in \cite{ginenickl2016}) and that $B_0\in \dot H^{s+1}$,
$$
    \|B-B_0\|_2 \lesssim \|B - B_{0,T}\|_2 + \|B_{0,T}-B_0\|_2 \lesssim T^{-\frac{s-(d/2+\kappa-1)_+}{2s+d}} + \eps_T.
$$
For $B \in \Dcal_T$, by the Sobolev interpolation inequality (e.g.~Theorems 1.3.3 and 4.3.1 in \cite{Triebel78}), 
\begin{align*}
    \|B - B_0\|_{H^{d+1/2+\kappa}} 
    &\lesssim 
    \|B - B_0\|_2^\frac{d/2}{1+d+\kappa}
    \|B - B_0\|_{H^{d+1+\kappa}}^\frac{1+d/2+\kappa}{1+d+\kappa}\\
    &\lesssim T^{-\frac{[s-(d/2+\kappa-1)_+]d/2}{(2s+d)(1+d+\kappa)}} =o(1),
\end{align*}
since $s>d+(d/2)\vee2$, which verifies the second requirement in \eqref{Eq:D_T_conditions}. Lastly, using the functional remainder assumption \eqref{Eq:remainder_text} and the second last display,
$$
    r(B,B_0)\lesssim T^{-\frac{2(s-(d/2+\kappa-1)_+)}{2s+d}}
    =o(1/\sqrt T),
$$
since $s>d+(d/2)\vee 2 >d/2+2(d/2+\kappa-1)\vee0.$
We conclude that for all sufficiently large $M>0$, the set $\Dcal_T$ satisfies the condition of Theorem \ref{thm:BvM} (with choices $\zeta_T = T^{-\frac{[s-(d/2+\kappa-1)_+]d/2}{(2s+d)(1+d+\kappa)}}$ and $\xi_T = T^{-\frac{s-d/2-2(d/2+\kappa-1)_+}{2s+d}}$ in \eqref{Eq:D_T_conditions}).

Now consider the representor $\psi \in \dot{C}^t(\T^d)$, where $t>(d/2-1)_+$. Using that $B_0\in \dot H^{s+1}(\T^d)$ with $s>d+(d/2)\vee 2$, and $\psi\in \dot C^t(\T^d)\subset \dot B^{(d/2-1+\kappa)_+}_{\infty1}(\T^d)$ (e.g.~p.~335 in \cite{ginenickl2016}) for sufficiently small $\kappa>0$, by Lemma \ref{lem:PDE} there exists a unique element $A^{-1}_{\mu_0}\psi \in  \dot{B}_{\infty 1}^{(d/2+1+\kappa)\vee 2}(\T^d)$ such that $A_{\mu_0}A^{-1}_{\mu_0}\psi = \psi$ almost everywhere and $
\|A^{-1}_{\mu_0}\psi\|_{B_{\infty 1}^{(d/2+1+\kappa)\vee 2}}\lesssim  \|\psi\|_{B_{\infty 1}^{(d/2-1+\kappa)_+}}$.
Take the wavelet projections onto $V_J$:
$$
    \gamma_T:=\sum_{l=0}^ J\sum_r
    \langle A^{-1}_{\mu_0}\psi,\Phi_{lr}\rangle_2\Phi_{lr},
$$
which satisfy $\gamma_T\in \dot H^{t'}(T^d)$ for all $t'\ge 0$. These verify \eqref{Eq:gamma_T_conditions} since by continuous embedding $B^{t'}_{\infty 1}(\T^d)\subset C^{t'}(\T^d)$ (e.g.~p.~347 in \cite{ginenickl2016}),
$\|\gamma_T\|_{H^{d/2+1+\kappa}} \lesssim  \|\gamma_T\|_{B^{d/2+1+\kappa}_{\infty 1}} <\infty,$ 
and for $\kappa>0$ small enough,
\begin{align*}
    \| A_{\mu_0}^{-1}\psi - \gamma_T\|_{W^{1,\infty}}
    \lesssim \| A_{\mu_0}^{-1}\psi - \gamma_T\|_{B^{1}_{\infty 1}}
    &\lesssim 2^{-J(d/2+\kappa)}\|A_{\mu_0}^{-1}\psi\|_{B_{\infty 1}^{d/2+1+\kappa}}
    = o(1/(\sqrt T\eps_T)).
\end{align*}

It remains to show the asymptotic invariance property \eqref{Eq:change_of_measure}. Following the notation and terminology set out in \cite{agapiou2021rates}, the space of `admissible shifts' associated to the re-scaled (truncated) Besov-Laplace prior $\Pi$ is $\Qcal=V_J \cap \dot{L}^2(\T^d)$. Since $\gamma_T\in V_J$, it follows from Proposition 2.7 in \cite{agapiou2021rates} that $\Pi_{u}:=\Lcal(B_{u})$, with $B_u = B - u\gamma_T\sqrt T$, is absolutely continuous with respect to $\Pi$ with Radon-Nikodym derivative
$$
    \frac{d\Pi_u}{d\Pi}(B) 
    =\exp \left\{  T^{\frac{d}{2s+d}} \left(\|B\|_{B^{s+1}_{11}} 
    - \|B+u\gamma_T/\sqrt T\|_{B^{s+1}_{11}} \right) \right\}.
$$
Therefore, it holds that
\begin{align*}
    \frac{\int_{\Dcal_T} e^{\ell_T(B_{u})} d\Pi(B) }
    {\int_{\Dcal_T}e^{\ell_T(B)}d\Pi(B)}
    &=  \frac{\int_{\Dcal_{T,u}} e^{\ell_T(B)}
    e^{T^{d/(2s+d)}(\|B\|_{B^{s+1}_{11}} 
    - \|B+u\gamma_T/\sqrt T\|_{B^{s+1}_{11}})}
    d\Pi(B) }
    {\int_{\Dcal_T}e^{\ell_T(B)}d\Pi (B)},
\end{align*}
where $D_{T,u}:=\{B'=B_u: \ B\in\Dcal_T\}$. Using the reverse triangle inequality and that $B_{\infty 1}^{s+1} \subset B_{11}^{s+1}$, 
\begin{align*}
    \sup_{B\in\Dcal_{T,u}}&
    T^{d/(2s+d)}\Big|\|B\|_{B^{s+1}_{11}} 
    - \|B+u\gamma_T/\sqrt T\|_{B^{s+1}_{11}}\Big|\\
    &\le T^{d/(2s+d)}
    \|u\gamma_T/\sqrt T\|_{B^{s+1}_{1 1}}\\
    &\lesssim T^{-\frac{s-d/2}{2s+d}}\|\gamma_T\|_{B_{\infty 1}^{s+1}}\\
    &\lesssim T^{-\frac{s-d/2}{2s+d}}2^{J(s+1-d/2-1-\kappa)}\|\gamma_T\|_{B_{\infty 1}^{d/2+1+\kappa}}
    \lesssim T^{-\frac{\kappa}{2s+d}}
    =o(1).
\end{align*}
This implies that as $T\to\infty$, for all $u\in\R$,
\begin{align*}
    \frac{\int_{\Dcal_T} e^{\ell_T(B_{u})} d\Pi(B) }
    {\int_{\Dcal_T}e^{\ell_T(B)}d\Pi(B)}
    &=e^{o(1)}\frac{\int_{\Dcal_{T,u}} e^{\ell_T(B)}
    d\Pi(B) }
    {\int_{\Dcal_T}e^{\ell_T(B)}d\Pi (B)} 
    =e^{o(1)}\frac{\Pi(\Dcal_{T,u}|X^T)}{\Pi(\Dcal_T|X^T)}.
\end{align*}
The denominator satisfies $\Pi(\Dcal_T|X^T)\to 1$ in $P_{0}$-probability by Lemma \ref{Lem:LocSetLapl}. Moreover,
\begin{align*}
    \Dcal_{T,u}^c
    & =\Big\{ B: \Big\|\nabla(B+\tfrac{u}{\sqrt T}\gamma_T)-\nabla B_{0,T}\Big\|_2>M \eps_T \Big\}  \cup \Big\{ B: \|B+\tfrac{u}{\sqrt T}\gamma_T\|_{H^{d+1+\kappa}}> M \Big\}.
\end{align*}
Since $t>(d-1)_+$,
$$
    \|\nabla\gamma_T/\sqrt T\|_2\lesssim \|A^{-1}_{\mu_0}\psi\|_{H^1} /\sqrt T \lesssim \|\psi\|_{L^2}/\sqrt T= o(\eps_T),
$$
while using $\gamma_T \in V_J$ gives
$$
    \|\gamma_T/\sqrt T\|_{H^{d+1+\kappa}}
    \lesssim 2^{Jd/2} T^{-1/2} \|\gamma_T\|_{H^{d/2+1+\kappa}}
    \lesssim \eps_T =o(1).
$$ 
Thus $\Dcal_{T,u}(M) \supset \Dcal_T(M/2)$ for $T$ large enough. But for $M>0$ enough, the $P_0$-probability of this last set tends to 1 by Lemma \ref{Lem:LocSetLapl}, which completes the proof.
\qed

\begin{lemma}\label{Lem:LocSetLapl}
Let $\Pi$ and $B_0$ be as in the statement of Theorem \ref{Theo:LaplBvM}, and let $\Dcal_T$ be the set defined in \eqref{Eq:LocSetLapl}. Then for all sufficiently large $M>0$, as $T\to\infty$,
$$
    \Pi(\Dcal_T|X^T)\overset{P_0}{\longrightarrow}1.
$$
\end{lemma}

\begin{proof}
Write $\Dcal_T:= \Dcal_{T,1}\cap \Dcal_{T,2}$, with $\Dcal_{T,1}:=\{B\in V_J :\|\nabla B - \nabla B_{0,T}\|_1\le M \eps_T\}$ and $\Dcal_{T,2}:=\{B\in V_J
:\|B\|_{H^{d+1+\kappa}} \le M \}$. We show that both sets have posterior probability tending to one in $P_{0}$-probability as $T\to\infty$.

Starting with $\Dcal_{T,1}$, note that since $B_0\in H^{s+1}(\T^d)$,
\begin{align}\label{Eq:B0_approx}
    \|\nabla B_0 - \nabla B_{0,T}\|_1
    \lesssim \| B_0 - B_{0,T}\|_{H^1}
    \le 2^{-Js}\|B_0\|_{H^{s+1}}
    \lesssim \eps_T.
\end{align}
Hence, provided that $M$ is sufficiently large,
$$
    \Dcal_{T,1}\supseteq\{ B\in \dot C^2(\T^d):
    \|\nabla B - \nabla B_{0}\|_1\le M\eps_T/2\}
$$
By Theorem 2.4 of \cite{giordano2022nonparametric} (with the choice $p=1$), whose assumptions are precisely recovered under the conditions of Theorem \ref{Theo:LaplBvM}, it then follows that
\begin{equation}
\label{Eq:MedStep}
    \Pi(\Dcal_{T,1}^c|X^T)\le \Pi( B\in \dot C^2(\T^d):
    \|\nabla B - \nabla B_{0}\|_1 > M\eps_T/2|X^T) 
    \overset{P_{0}}{\longrightarrow}0. 
\end{equation}
Further, for sufficiently large $M>0$,
$$
    \Dcal_{T,2}\supseteq
    \Dcal_{T,2}':= \left\{ B = B_1 + B_2: B_1,B_2\in V_J,
    \|B_1\|_\infty\le \tfrac{M}{2}T^{-\frac{s+1}{2s+d}}, 
    \| B_2\|_{B^{s+1}_{11}}\le \tfrac{M}{2} \right\}.
$$
Indeed, since $s>d+(d/2)_+$, for $B_1,B_2\in V_J$ as above,
$$
    \|B_1\|_{H^{d+1+\kappa}}
    \lesssim 2^{J(d+1+\kappa)}T^{-\frac{s+1}{2s+d}}
    \simeq T^{-\frac{s-d-\kappa}{2s+d}} = o(1),
$$
and 
$$
    \|B_2\|_{H^{d+1+\kappa}}
    \lesssim 
    \|B_2\|_{B^{s+1}_{11}}
    \lesssim M,
$$
holding in view of the embedding $B^{s+1}_{11}(\T^d)\subset H^{d+1+\kappa}(\T^d)$ (cf.~eq.~(69) in \cite{LSS09}). By Lemma 5.2 of \cite{giordano2022nonparametric} (with the choice $p=1$), we then have that for all $K>0$, we may choose $M>0$ large enough such that
$$
    \Pi(\Dcal_{T,2}^c)\le 
    \Pi(\Dcal_{T,2}'^c) \le
    e^{-KT\eps_T^2}.
$$
Similarly to the conclusion of the proof of Lemma \ref{Lem:LocSetGP}, we then obtain via an analogue of Theorem 8.20 in \cite{ghosal2017} for the present setting that $\Pi(\Dcal_{T,2}|X^T)\overset{P_0}{\longrightarrow}1$ as $T\to\infty$. Combined with \eqref{Eq:MedStep}, this concludes the proof.
\end{proof}

\subsection{Functional expansions}
\label{Sec:functionals}

In this section, we study conditions under which nonlinear (in $B$) functionals satisfy the approximately linear expansion \eqref{Eq:psi_exp}. Examples \ref{Ex:linear_B} and \ref{Ex:square} follow immediately.

\begin{proof}[Proof of Example \ref{Ex:power}]
    For $\Psi(B) = \int B^q$, using the binomial theorem,
    \begin{align*}
        \Psi(B) - \Psi(B_0) - \langle qB_0^{q-1} - q \int B_0^{q-1} , B-B_0 \rangle_2 = \int \sum_{k=2}^{q} {q \choose k} (B-B_0)^{k} B_0^{q-k}.
    \end{align*}
    For $\|B\|_\infty, \|B_0\|_\infty \leq K$, using the interpolation equality $\|f\|_k \leq \|f\|_2^{2/k} \|f\|_\infty^{1-2/k}$ for any $2 \leq k \leq \infty$, the right-hand side is bounded by a multiple of
    $$\sum_{k=2}^q \|B-B_0\|_k^k \|B_0\|_\infty^{q-k} \leq  \sum_{k=2}^q K^{q-k} \|B -B_0\|_2^2 \|B-B_0\|_\infty^{k-2} \lesssim \|B-B_0\|_2^2$$
    as required.
\end{proof}

The next lemma provides an expansion for linear functionals of the invariant measure $\mu_B = e^{2B}/\int e^{2B}$ as in Example \ref{Ex:linear}, which are nonlinear in the potential $B$.

\begin{lemma}\label{Lem:linear_mu}
    Let $B,B_0 \in \dot{L}^2(\T^d)$ and $\Psi(B) = \int_{\T^d} \mu_B(x) \varphi(x) dx$ for $\varphi \in L^\infty(\T^d)$. Then the functional $\Psi$ satisfies the expansion
    $$\Psi(B)=\Psi(B_0) + \langle 2 \mu_{B_0}[\varphi - \Psi(B_0)], B-B_0 \rangle_2 + r(B,B_0),$$
    where for any $K,M>0$ and $\eps_T \to 0,$
    \begin{align}\label{Eq:remainder}
    \sup_{\substack{B,B_0:\|B\|_\infty, \|B_0\|_\infty \leq K \\ \|B - B_0\|_2 \leq M \eps_T}}|r(B,B_0)| = O(\eps_T^2)
    \end{align}
    as $T\to\infty$. In particular, $\Psi$ satisfies the expansion \eqref{Eq:psi_exp} with $\psi =2\mu_{B_0}[\varphi - \int \mu_{B_0} \varphi] \in \dot{L}^2(\T^d)$.
\end{lemma}

\begin{proof}
Consider $B,B_0 \in \dot{L}^2(\T^d)$ such that $\|B\|_\infty,\|B_0\|_\infty \leq K$ and $\| B- B_0\|_2 \leq M \eps_T$, i.e.~in the supremum over which we consider the remainder term. Write $N_B = \int_{\T^d} e^{2B(x)} dx$ and recall that $\mu_B = e^{2B}/N_B$ from \eqref{Eq:invariant_measure}. Expanding,
\begin{align}\label{Eq:psi_exp_mu}
    \Psi(B) - \Psi(B_0) &= \int (\mu_B - \mu_{B_0}) \varphi \nonumber \\
    &= \int \left[ \frac{e^{2B}}{N_{B_0}} - \frac{e^{2B_0}}{N_{B_0}} + \frac{e^{2B}}{N_B} - \frac{e^{2B}}{N_{B_0}}  \right] \varphi \nonumber \\
    & = \int \frac{e^{2B_0}}{N_{B_0}} \left[ e^{2(B-B_0)}-1 \right] \varphi + \left( \frac{1}{N_B} - \frac{1}{N_{B_0}} \right) \int e^{2B} \varphi.
\end{align}
Let $\rho(x) = e^{2x} - 1 - 2x$, which satisfies $|\rho(x)| \leq 2x^2 e^{2|x|}$ using the Lagrange form of the Taylor expansion remainder. The first term above equals
$$
    2 \int (B-B_0) \mu_{B_0} \varphi + \int \rho(B-B_0) \mu_{B_0} \varphi.
$$
Moreover,
\begin{align*}
\frac{1}{N_B} - \frac{1}{N_{B_0}} = \frac{\int e^{2B_0} - e^{2B}}{N_B N_{B_0}} 
& = \frac{\int e^{2B_0}(1- e^{2(B-B_0)})}{N_B N_{B_0}} \\
& = - \frac{1}{N_B} \int \mu_{B_0} [2(B-B_0) + \rho(B-B_0)]
\end{align*}
and so the second term in \eqref{Eq:psi_exp_mu} equals
$$- \int \mu_{B_0} [2(B-B_0) + \rho(B-B_0)] \int \mu_B \varphi.$$
Substituting these into \eqref{Eq:psi_exp_mu} then yields
\begin{equation}\label{Eq:psi_exp_mu2}
\begin{split}    
    \Psi(B) - \Psi(B_0) &= 2 \int (B-B_0) \mu_{B_0} \varphi + \int \rho(B-B_0) \mu_{B_0} \varphi \\
    & \qquad - 2\int (B-B_0) \mu_{B_0} \int \mu_B \varphi - \int \rho(B-B_0) \mu_{B_0} \int \mu_B \varphi.
\end{split}
\end{equation}
Since $\|\mu_{B_0}\|_\infty \leq e^{4\|B_0\|_\infty} \leq e^{4K}$, the second term in \eqref{Eq:psi_exp_mu2} satisfies
$$\left| \int \rho(B-B_0) \mu_{B_0} \varphi \right| \leq \ \int 2(B-B_0)^2 e^{2|B-B_0|} \mu_{B_0} |\varphi| \leq 2e^{8K} \|\varphi\|_\infty \| B -  B_0\|_2^2. $$
The same bound holds for the fourth term in \eqref{Eq:psi_exp_mu2}, so that 
\begin{align*}
    \Psi(B) - \Psi(B_0) &= 2 \int (B-B_0) \mu_{B_0} \varphi - 2\int (B-B_0) \mu_{B_0} \int \mu_B \varphi  + O(\|B - B_0\|_2^2),
\end{align*}
where the constants in the remainder term depend only on $K$ and $\|\varphi\|_\infty$. Adding and subtracting $2\int (B-B_0) \mu_{B_0} \int \mu_{B_0} \varphi = 2\int (B-B_0) \mu_{B_0} \Psi(B_0)$ to the right hand side,
\begin{align*}
    \Psi(B) - \Psi(B_0) &= 2 \int (B-B_0) \mu_{B_0} \varphi - 2\int (B-B_0) \mu_{B_0} [\Psi(B) - \Psi(B_0)+\Psi(B_0)]  \\
    & \qquad + O(\| B - B_0\|_2^2)\\
    & = O(\eps_T) + [\Psi(B)-\Psi(B_0)]O(\eps_T),
\end{align*}
where all remainder terms are uniform over $K,M,\|\varphi\|_\infty$ and we used $\|B-B_0\|_1 \leq  \| B -  B_0\|_2$. Since $\eps_T \to 0$, we deduce that $|\Psi(B) - \Psi(B_0)| = O(\eps_T)$. Substituting this back into the right side of the last display,
\begin{align*}
    \Psi(B) - \Psi(B_0) &= 2 \int (B-B_0) \mu_{B_0} (\varphi - \Psi(B_0)) + O(\eps_T^2).
\end{align*}
\end{proof}

Many interesting functionals are approximately linear in the invariant measure $\mu$, and Lemma \ref{Lem:linear_mu} allows us to perform a further linearization in $B$ on the linearizations in $\mu$. This is the approach we take for several examples.

\begin{proof}[Proof of Example \ref{Ex:entropy}]
    For the entropy functional $\Psi(B) = \int \mu_B \log \mu_B$,
    $$\Psi(B) - \Psi(B_0) = \int (\mu_B - \mu_{B_0}) \log \mu_{B_0} + \int \mu_B \log \frac{\mu_B}{\mu_{B_0}}.$$
    Applying Lemma \ref{Lem:linear_mu} to the functional $\Phi(B) = \int \mu_B \log \mu_{B_0}$ since $\mu_{B_0} \in L^\infty(\T^d)$,
    $$\Psi(B) - \Psi(B_0) = \langle 2\mu_{B_0} [\log \mu_{B_0} - \Phi(B_0)], B-B_0 \rangle_2 + r_1(B,B_0) + \int \mu_B \log \frac{\mu_B}{\log \mu_{B_0}},$$
    where $r_1$ satisfies \eqref{Eq:remainder}. The last term equals the Kullback-Leibler divergence $KL(\mu_B,\mu_{B_0})$ between $\mu_B$ and $\mu_{B_0}$, which is bounded by a multiple of $h(\mu_B,\mu_{B_0})^2\|\mu_B/\mu_{B_0}\|_\infty$ by Lemma B.2 of \cite{ghosal2017}, where $h$ denotes the Hellinger distance. Arguing as in the proof of Lemma 2.5 of \cite{ghosal2017},
    \begin{align*}
        h(\mu_B,\mu_{B_0})^2 \leq \frac{4 \int e^{2B_0} e^{2|B-B_0|} |B-B_0|^2}{\int e^{2B_0}} \leq 4 e^{8K}\|B-B_0\|_2^2
    \end{align*}
    for $\|B\|_\infty, \|B_0\|_\infty\leq K$, so that $KL (\mu_B,\mu_{B_0}) \lesssim e^{12K} \|B-B_0\|_2^2$. The last term in the second last display therefore satisfies the same bound as \eqref{Eq:remainder}.
\end{proof}

\begin{proof}[Proof of Example \ref{Ex:square_root}]
For $\Psi(B) = \int \sqrt{\mu_B}$,
\begin{align*}
    \Psi(B) - \Psi(B_0) = \frac{1}{2} \int (\mu_B - \mu_{B_0}) \frac{1}{\sqrt{\mu_{B_0}}} - \frac{1}{2}\int \frac{(\sqrt{\mu_B}-\sqrt{\mu_{B_0}})^2}{\sqrt{\mu_{B_0}}}.
\end{align*}
Similar to the last example, the third term is bounded by $2e^{10K} \|B-B_0\|_2^2$. Applying Lemma \ref{Lem:linear_mu} to $\Phi(B) = \tfrac{1}{2}\int \mu_B \mu_{B_0}^{-1/2}$ then gives
$$\Psi(B) - \Psi(B_0) =  \langle \mu_{B_0}[\mu_{B_0}^{-1/2} - \int \sqrt{\mu_{B_0}}] ,B-B_0 \rangle_2 + r(B,B_0),$$
with the remainder term satisfying \eqref{Eq:remainder}.
\end{proof}

\begin{proof}[Proof of Example \ref{Ex:power_mu}]
    As in the proof of Example \ref{Ex:power}, for $\Psi(B) = \int \mu_B^q$,
    $$\Psi(B) - \Psi(B_0) - \langle q\mu_{B_0}^{q-1} - q \int \mu_{B_0}^{q-1}, \mu_B-\mu_{B_0} \rangle_2 = O \left( \|\mu_B - \mu_{B_0}\|_2^2 \right),$$
    where the remainder term on the right side depends only on $K$ and $q$. But $\|\mu_B - \mu_{B_0}\|_2^2 \leq \|\sqrt{\mu_B}+\sqrt{\mu_{B_0}}\|_\infty^2 h(\mu_B,\mu_{B_0})^2 \leq 8e^{12K}\|B-B_0\|_2^2$ by the above. Applying Lemma \ref{Lem:linear_mu} as in the previous examples thus gives the result with $\psi = 2\mu_{B_0} [q \mu_{B_0}^{q-1} - q \int \mu_{B_0}^q]$. 
\end{proof}

\subsection{A PDE estimate}

We require the following regularity estimate for solutions to the Poisson equation $A_{\mu}u = f$, where $A_{\mu}$ is the second order elliptic operator defined in \eqref{Eq:A_mu}. Recall that $\dot{L}^2(\T^d) = \{ f\in L^2(\T^d): \int_{\T^d} f(x) dx = 0\}$.

\begin{lemma}\label{lem:PDE}
Let $t \in\R$ and assume $\mu \in C^{|t-2|+1}(\T^d)$ is strictly positive on $\T^d$. For any $f \in \dot L^2(\T^d)$, there exists a unique solution $A_\mu^{-1}[f] \in  \dot{L}^2(\T^d)$ to the equation $A_\mu u = f$ satisfying $A_\mu A_\mu^{-1}[f]=f$ almost everywhere. Moreover,
\begin{equation*}
\|A_\mu^{-1}[f]\|_{B^t_{pq}} \lesssim \|f\|_{B^{t-2}_{pq}},
\end{equation*}
with constants depending on $t,p,q,d$ and on  $\|1/\mu\|_{B_{\infty\infty}^{|t-2|}}\vee \|\mu\|_{B_{\infty\infty}^{|t-2|+1}}$.
\end{lemma}

\begin{proof}
Let $\{e_k=e^{2\pi i k.\cdot}, k =(k_1, \dots, k_d) \in \mathbb Z^d\}$ denote the trigonometric basis of $L^2(\T^d)$ and $(\psi_j)_{j\geq 0}$ form a Littlewood-Paley resolution of unity such that $\supp(\psi_0) \subset \{x:|x| \leq 2\}$ and $\psi_j (x) = \psi(2^{-j}x)$, $j\geq 1$, for some non-negative Schwartz function $\psi\in \mathcal{S}(\R^d)$ with $\supp(\psi) \subset \{x:1/2 \leq |x| \leq 2\}$ (see p.~81-82 of \cite{ST87} for a full definition and construction). Then for $t\in\R$, $1\leq p,q\leq \infty$, an equivalent norm for the periodic Besov space $B_{pq}^t(\T^d)$ is given by
\begin{equation}\label{norm_equivalence}
\|u\|_{B_{pq}^t} = \left( \sum_{j\geq 0} 2^{tqj} \left\| \sum_{k\in \Z^d} \psi_j(k) \langle u,e_k\rangle_2 e_k \right\|_{L^p(\T^d)}^q \right)^{1/q},
\end{equation}
with the natural modification when $q=\infty$, see p.~162 of \cite{ST87}.

Let $q<\infty$. For any $u\in \Hcal := \dot B_{pq}^t(\T^d)$, since $\langle u,e_0\rangle_2 = \langle u,1\rangle_2 = 0$ and $\langle \Delta u, e_k \rangle_2 = -(2\pi)^2\sum_j k_j^2 \langle u, e_k \rangle_2$ for $k\neq 0$,
\begin{align*}
    \|u\|_{B^t_{pq}}^q 
    &= \sum_{j \geq 0} 2^{tqj} \Big\|\sum_{k \in \mathbb Z^d, k \neq 0} \frac{1}{4\pi^2 |k|^2} \psi_j(k) \langle \Delta u, e_k \rangle_2 e_k \Big\|_{L^p(\T^d)}^q \\
    &= \sum_{j \geq 0} 2^{(t-2)qj} \Big\|\sum_{k \in \mathbb Z^d, k \neq 0} M_j(k)\psi_j(k) \langle \Delta u, e_k \rangle_2 e_k \Big\|_{L^p(\T^d)}^q,
\end{align*}
where $M_j= M(2^{-j} \cdot)$ and $M=\Phi/(4\pi^2|\cdot|^2)$ with $\Phi$ a smooth function supported on $\{x:1/4 \leq |x| \leq 9/4\}$ such that $\Phi =1$ on $\{x:1/2\leq |x|\leq 2\}$. By Lemma \ref{lem:Fourier_mult}, the last display is bounded by
$$
    \sum_{j\geq 0} 2^{(t-2)qj} \Big\|\sum_{k \in \mathbb Z^d, k \neq 0} \psi_j(k) \langle \Delta u, e_k \rangle_2 e_k \Big\|_{L^p(\T^d)}^q \|F^{-1}[M_j]\|_{L^1(\mathbb R^d)}^q,
$$ 
where $F^{-1}$ is the inverse Fourier transform. Since $\Phi\in \mathcal{S}(\R^d)$, so are both $M$ and $F^{-1}M$, which implies $\sup_j \|F^{-1}[M_j]\|_{L^1(\mathbb R^d)}=\|F^{-1}[M]\|_{L^1(\mathbb R^d)}<\infty.$ Using again the norm equivalence \eqref{norm_equivalence}, this yields
\begin{equation*}\label{laplace_besov}
\|u\|_{B_{pq}^t} \leq C \|\Delta u\|_{B_{pq}^{t-2}}
\end{equation*}
for an absolute constant $C>0$. The case $q=\infty$ follows identically. Using the last display, the definition of $A_{\mu}$ and the multiplication inequality for Besov norms, 
\begin{equation}\label{initio}
\begin{split}
\|u\|_{B_{pq}^t} & \lesssim \|(1/\mu) A_\mu u\|_{B_{pq}^{t-2}} + \|(1/\mu) \nabla \mu . \nabla u\|_{B_{pq}^{t-2}} \\
& \lesssim \|1/\mu\|_{B_{\infty\infty}^{|t-2|}} \|A_\mu u\|_{B_{pq}^{t-2}} + \|1/\mu\|_{B_{\infty\infty}^{|t-2|}} \|\mu\|_{B_{\infty\infty}^{|t-2|+1}} \|u\|_{B_{pq}^{t-1}}
\end{split}
\end{equation}
for all $u \in \mathcal H$, with constants depending on $t,p,q,d$. We will deduce from this the inequality
\begin{equation}\label{key}
\|u\|_{B_{pq}^t} \lesssim \|A_\mu u\|_{B_{pq}^{t-2}}, ~ \qquad \forall u \in \mathcal H.
\end{equation}
Indeed, suppose the last inequality it is not true. Then there exists a sequence $u_m \in \mathcal H$ such that $\|u_m\|_{B_{pq}^t}=1$ for all $m$ but $\|A_\mu u_m\|_{B_{pq}^{t-2}} \to 0$ as $m \to \infty$. By compactness, $u_m$ has a subsequence, also denoted by $u_m$, which converges in $\|\cdot\|_{B_{pq}^{t-1}}$-norm to some $u\in \mathcal H$ satisfying $A_\mu u=0$. Using (\ref{initio}) with fixed constant depending only on $K, t,p,q, d$, implies that $u_m$ is also Cauchy in $B_{pq}^t$, and its limit must satisfy $\|u\|_{B_{pq}^t}=1$. However, the only solution $u \in \mathcal H$ to $A_\mu u=0$ on $\T^d$ equals $u=const=0$, contradicting $\|u\|_{B_{pq}^t}=1$ and thereby proving (\ref{key}). 

By the Fredholm property, a solution $u_f$ to $A_\mu u = f$ exists whenever $\int_{\T^d} f  = 0$, and for $f \in H^{t-2}(\T^d)$ any such solution belongs to $H^t(\T^d)$ (see Theorem 3.5.3 in \cite{BJS64}, which is proved for smooth $b$, but the proof remains valid for $b \in C^{|t-2|}(\T^d)$). The weak maximum principle (p.179 in \cite{gilbarg1977elliptic}) now implies that $u_f$ is unique up to an additive constant, and applying (\ref{key}) to the unique selection $u_f=A_\mu^{-1}[f] \in \mathcal H$ completes the proof. 
\end{proof}

\begin{lemma}\label{lem:Fourier_mult}
Let $f = \sum_{k\in \Z^d: \|k\|_2 \leq t} \langle f,e_k\rangle e_k$ be a finite trigonometric polynomial, $m\in C^\infty(\R^d)$ and $\Phi \in \mathcal{S}(\R^d)$ satisfy $\supp(\Phi) \subset \{x:|x| \leq 2\}$ and $\Phi = 1$ on $\{x:|x| \leq 1\}$. Then for every $1\leq p \leq \infty$,
$$\left\| \sum_{k\in \Z^d: \|k\|_2 \leq t} \langle f,e_k\rangle e_k \right\|_{L^p(\T^d)} \leq \|f\|_{L^p(\T^d)} \|F^{-1}[m] * F^{-1}[\Phi(\cdot/t)] \|_{L^1(\R^d)}.$$
\end{lemma}

\begin{proof}
The function $\Phi(\cdot/t)$ is supported on $\{x : |x| \leq 2t\}$ and equal to 1 on $\{x:|x| \leq t\}$. Define
$$M(u) = m(u) \Phi(u/t), \qquad u\in \R^d,$$
which is in $C^\infty(\R^d)$, has compact support and coincides with $m$ on $\{x:|x| \leq t\}$. Since $M \in \mathcal{S}(\R^d)$, the Fourier inversion $F[F^{-1}[M]](k) = M(k)$ holds pointwise. Therefore, for $c_k = \langle f,e_k \rangle_2$,
\begin{align*}
\sum_{\|k\|\leq t} c_k m(k) e_k(x) & = \sum_{\|k\|\leq t} c_kM(k) e_k(x) =  \sum_{\|k\|\leq t} c_k F[F^{-1}[M]](k) e_k(x) \\
& = \sum_{\|k\|\leq t} \int_{\R^d} F^{-1}[M](y) e^{-ik.y} dy~ c_k e_k(x) \\
& =  (2\pi)^d \sum_{\|k\|\leq t} \int_{\R^d} F^{-1}[M](2\pi z) c_k e^{2\pi i (x-z).k} dz \\
& = (2\pi)^d \int_{\R^d} F^{-1}[M](2\pi z) f(x-z) dz.
\end{align*}
Taking the $L^p(\T^d)$ norm of the last display and using Minkowski's integral inequality,
$$\left\| \sum_{\|k\|\leq t} c_k m(k) e_k \right\|_{L^p(\T^d)} \leq \|f\|_{L^p(\T^d)} \|F^{-1}[M]\|_{L^1(\R)},$$
where we have also used a change of variable in the last term. The result then follows since $F^{-1}[M] = F^{-1}[m]*F^{-1}[\Phi(\cdot/t)]$.
\end{proof}

%
\paragraph{Acknowledgements.} We are grateful to the Associate Editor and two anonymous Referees for many insightful comments that lead to an improvement of our manuscript. We would also like to thank Richard Nickl for helpful discussions.  M.~G.~was supported by MUR - Prin 2022 - Grant no.
2022CLTYP4, funded by the European Union – Next Generation EU, and by the ``de Castro'' Statistics Initiative, Collegio Carlo Alberto, Torino. Part of this research was done while K.~R.~was visiting the LPSM at the Sorbonne University in Paris, funded by a CNRS Poste Rouge,
and while M.~G.~was visiting the CEREMADE at the University Paris Dauphine, funded by a mobility grant by Campus France.

%
%
%
%
%


\bibliography{References}

\end{document}